\documentclass[11pt]{amsart}
\usepackage{amsmath}
\usepackage{a4wide}
\usepackage[utf8]{inputenc}
\usepackage{amssymb}
\usepackage{amsopn}
\usepackage{epsfig}
\usepackage{amsfonts}
\usepackage{latexsym}
\usepackage{amsthm}
\usepackage{enumerate}
\usepackage[UKenglish]{babel}
\usepackage{verbatim}
\usepackage{color}
\usepackage{calrsfs}
\usepackage{pgf}
\usepackage{tikz}

 \usetikzlibrary{arrows,automata}
\DeclareMathAlphabet{\pazocal}{OMS}{zplm}{m}{n}

\newtheorem{theorem}{Theorem}[section]
\newtheorem{lemma}[theorem]{Lemma}
\newtheorem{proposition}[theorem]{Proposition}

\theoremstyle{definition}
\newtheorem{definition}[theorem]{Definition}
\newtheorem{example}[theorem]{Example}

\newtheorem{conjecture}{Conjecture}[section]

\newtheorem{main}{Theorem}

\theoremstyle{remark}
\newtheorem{remark}[theorem]{Remark}

\numberwithin{equation}{section}
\newcommand{\N}{\ensuremath{\mathbb{N}}}
\newcommand{\R}{\ensuremath{\mathfrak{R}}}

\renewcommand{\a}{\mathbf{a}}

\newcommand{\G}{\mathfrak{G}}
\newcommand{\set}[1]{\left\{#1\right\}}

\newcommand{\E}{\ensuremath{\mathcal{E}}}

\renewcommand{\u}{\pazocal{U}}
\newcommand{\ub}{\mathbf{U}} 
\newcommand{\ul}{\ensuremath{\mathcal{U}}} 

\newcommand{\VB}{\mathbf{V}} 
\newcommand{\vl}{\mathcal{V}}
\newcommand{\vb}{\mathbf{V}} 

\newcommand{\I}{\mathcal{I}}
\newcommand{\IB}{{\mathbf{I}}}

\newcommand{\W}{\pazocal{W}}
\newcommand{\WB}{\mathbf{W}}

\newcommand{\LB}{\mathfrak{L}}
\renewcommand{\L}{B_*}

\newcommand{\om}{\omega}

\newcommand{\la}{\lambda}

\newcommand{\ga}{\gamma}

\newcommand{\ep}{\varepsilon}
\newcommand{\f}{\infty}

\newcommand{\al}{\alpha}

\newcommand{\si}{\sigma}
\newcommand{\de}{\delta}

\newcommand{\ra}{\rightarrow}

\newcommand{\lle}{\preccurlyeq}
\newcommand{\lge}{\succcurlyeq}

\begin{document}

\title[Entropy, Topological transitivity, and dimension properties]{Entropy, Topological transitivity, and Dimensional properties of unique $q$-expansions}

\author{Rafael Alcaraz Barrera}
\address{Departamento de Matem\'atica Aplicada, Instituto de Matem\'atica e Estat\'istica, Universidade de S\~ao Paulo, Rua do Mat\~ao 1010, Ci\-da\-de U\-ni\-ver\-si\-ta\-ria, 05508-090, S\~ao Paulo SP, Brasil}
\address{\emph{Current Address}: Instituto de F\'isica, Universidad Aut\'onoma de San Luis Potos\'i. Av. Manuel Nava 6, Zona Universitaria, C.P. 78290. San Luis Potos\'i, S.L.P. M\'exico}
\email{rafalba@ime.usp.br \quad \quad ralcaraz@ifisica.uaslp.mx}

\author{Simon Baker}
\address{Mathematics Institute, University of Warwick, Coventry, CV4 7AL, United Kingdom}
\email{simonbaker412@gmail.com}

\author{Derong Kong}
\address{School of Mathematical Science, Yangzhou University, Yangzhou, Jiangsu 225002, People's Republic of China}
\address{\emph{Current Address}: Mathematical Institute, University of Leiden, PO Box 9512, 2300 RA Leiden, The Netherlands}
\email{derongkong@126.com\quad \quad d.kong@math.leidenuniv.nl}

\thanks{ Research of R. Alcaraz Barrera was sponsored by FAPESP 2014/25679-9  and by CONACYT-FORDECYT
265667. Research of S. Baker was supported by the EPSRC grant EP/M001903/1. Research of D. Kong was supported by the NSFC No.~11401516.}

\date{\today}


\begin{abstract}
Let $M$ be a positive integer and $q \in(1,M+1].$ We consider expansions of real numbers in base $q$ over the alphabet $\{0,\ldots, M\}$. In particular, we study the set $\u_{q}$ of real numbers with a unique $q$-expansion, and the set $\ub_q$ of corresponding sequences.

It was shown in \cite[Theorem 1.7]{Komornik_Kong_Li_2015_1} that the function $H$, which associates to each $q\in(1, M+1]$ the topological entropy of $\u_q$, is a Devil's staircase. In this paper we explicitly determine the plateaus of $H$,  and characterize the bifurcation set $\E$ of $q$'s where the function $H$ is not locally constant.  Moreover, we show that $\E$ is a Cantor set of full Hausdorff dimension. We also investigate the topological transitivity of a naturally occurring subshift $(\VB_q, \sigma),$ which has a close connection with open dynamical systems. Finally, we prove that the Hausdorff dimension and box dimension of $\u_q$ coincide for all $q\in(1,M+1]$.
\end{abstract}

\maketitle

\tableofcontents

\section{Introduction}
\label{sec:1}

\noindent Fix a positive integer $M$. For $q\in(1,M+1]$ we call a sequence $(x_i)=x_1x_2\ldots\in  \set{0,1,\ldots,M}^\f$ an \textit{expansion of $x$ in base $q$} (or simply a \textit{$q$-expansion of $x$}) if
\begin{equation}
\label{eq:11}
x=\pi_{q}((x_{i})):=\sum_{i=1}^\f\dfrac{x_i}{q^i}.
\end{equation}
It is easy to see that $x$ has a $q$-expansion if and only if $x \in [0, M/(q-1)].$

Expansions in non-integer bases were pioneered in the papers of R\'{e}nyi \cite{Renyi_1957} and Parry \cite{Parry_1960}. Since these beginnings, expansions in non-integer bases have received much attention and have connections with many areas of mathematics, most notably, ergodic theory, fractal geometry, symbolic dynamics and number theory.

For the standard integer base expansions it is well known that every number has a unique expansion except for a countable set of exceptions that have precisely two. When our base is non-integer the situation is very different, as the following result of Sidorov demonstrates. In \cite{Sidorov_2003} Sidorov proved that if $q\in(1,M+1)$ then Lebesgue almost every $x\in[0, M/(q-1)]$ has a continuum of $q$-expansions (see also, \cite{Dajani_deVries_2007, Sidorov_2007}). Furthermore, for any  $k\in\N\cup\set{\aleph_0}$ there exist $q\in(1,M+1)$ and $x\in[0, M/(q-1)]$ such that $x$ has precisely $k$ different $q$-expansions (cf.~\cite{Erdos_Joo_1992, Sidorov_2009}).

Within expansions in non-integer bases, a particularly well studied topic is the set of real numbers with a unique expansion. For $q\in(1, M+1]$ let
$$
\u_q:=\{x\in[0,M/(q-1)]: x \textrm{ has a unique }q\textrm{-expansion}\}.$$
We call $\u_q$ the \emph{univoque set}. Let $\ub_q:=\pi_{q}^{-1}(\u_{q})$ be the corresponding set of expansions. 
For a sequence $(c_i)\in\set{0,1,\ldots, M}^\f$ we denote its \emph{reflection} by $\overline{(c_i)}:=(M-c_i)$. Then $\ub_q$ is the set of sequences $(x_i)\in\set{0, 1,\ldots, M}^\f$ satisfying the following lexicographic inequalities  (cf.~\cite{DeVries_Komornik_2008})
\[\left\{
\begin{array}{lll}
\si^n((x_i))\prec (\al_i(q)) &\textrm{whenever}& x_n<M,\\
\si^n((x_i)) \succ \overline{(\al_i(q))}&\textrm{whenever}& x_n>0.
\end{array}
\right.\]
Here $\si$ is the left shift and $(\al_i(q))=\al_1(q)\al_2(q)\ldots$ is the lexicographically largest $q$-expansion of $1$ with infinitely many non-zero elements. For more information on the sets $\u_q$ and $\ub_q$ we refer the reader to \cite{Daroczy_Katai_1993, Erdos_Joo_Komornik_1990,  Glendinning_Sidorov_2001, Komornik_2011} and the references therein.

For $q\in(1, M+1]$ we denote by $h(\ub_q)$  the \emph{topological entropy} of $\ub_q$:
 \[
 h(\ub_q):=\lim_{n\ra\f}\frac{\log\# B_n(\ub_q)}{n}.
 \]
 Where $ B_n(\ub_q)$ is the set of length $n$ subwords that appear in sequences of $\ub_q,$ and $\#$ denotes cardinality. Here and throughout this paper we will use the natural base $M+1$ logarithms. The topological entropy quantifies the size and the complexity of $\ub_q$. Note that the topological entropy in a symbolic space is always defined for a \emph{subshift} (see Definition \ref{def:22} below). Although $(\ub_q, \sigma)$  is not always a subshift (cf.~\cite[Theorem 1.8]{DeVries_Komornik_2008}), it was shown in \cite[Lemma 2.1]{Komornik_Kong_Li_2015_1} that the entropy function
 \begin{align*}
 H: \quad& (1, M+1]\longrightarrow[0,1]\\
 & ~\qquad q~~\quad \mapsto \quad h(\ub_q)
 \end{align*}
 is everywhere well-defined. In \cite{Komornik_Kong_Li_2015_1} the authors proved that the entropy function $H$ is a \emph{ Devil's staircase} (see e.g., Figure \ref{Fig:1}), i.e., $H$ satisfies the following:
 \begin{enumerate}
   \item $H$ is increasing, continuous, and surjective as a function from $(1,M+1]$ onto $[0,1]$;
   \item $H'(q)=0$ almost everywhere in $(1, M+1]$;

 \end{enumerate}
Furthermore,  $H(q)>0$ if and only if $q>q_c$, where $q_c=q_c(M)$ is the \emph{Komornik-Loreti constant} (see Equation \eqref{eq:23} below).

 \begin{figure}[h!] 
    \centering
 \includegraphics[width=12cm]{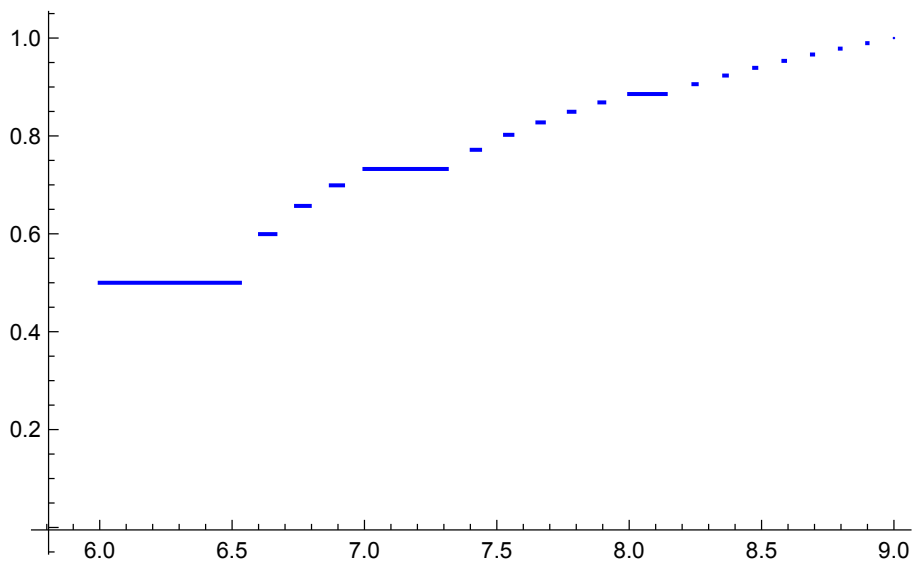}
    \caption{The asymptotic graph of the entropy function $H(q)=h(\ub_q)$ with $M=8$. Here $q_c\approx 5.80676$. For more explanation see Example \ref{ex:513}.}
    \label{Fig:1}
 \end{figure}

In this paper we are interested in characterizing the intervals where $H$ is constant. We are also interested in determining the properties of the \emph{bifurcation set}
 \begin{equation}\label{eq:12}
 \E:=\set{q\in(1,M+1]:  \forall \epsilon>0,\, \exists p\in(q-\epsilon, q+\epsilon)\textrm{ such that }H(p)\neq H(q)}.
 \end{equation}
Then $\E$ is the set of bases where $H$ is not locally constant. We call an element of $\E$ a \emph{bifurcation base}. 

Let us recall from \cite{Komornik_Loreti_2002} the set
$$\ul=\ul(M):=\{q\in(1,M+1]: 1 \textrm{ has a unique } q\textrm{-expansion}\}.$$ We refer to elements of $\ul$ as \emph{univoque bases}. It is well known that the set $\ul$ has a close connection with the univoque set $\u_q$ (cf.~\cite{DeVries_Komornik_2008}). Erd\H{o}s and Jo{\'{o}} proved that $\ul$ has Lebesgue measure zero \cite{Erdos_Joo_1992}, Dar\'{o}czy and K\'{a}tai \cite{Darczy_Katai_1995} later showed that $\ul$ has full Hausdorff dimension (see also, \cite[Theorem 1.6]{Komornik_Kong_Li_2015_1}). Recently, de Vries et al. \cite[Theorem 1.2]{DeVries_Komornik_Loreti_2016} proved that its closure $\overline{\ul}$ is a \emph{Cantor set}, i.e., a nonempty compact set with empty interior and no isolated points. They also obtained
 \[
 (1, M+1]\setminus\overline{\ul}=\bigcup(q_0, q_0^*).
 \]
 Where the union in the above equation is pairwise disjoint and countable.
Note by \cite{Komornik_Loreti_2002} the  {Komornik-Loreti constant} $q_c=q_c(M)$ is the smallest base of $\ul$. So the first connected component of $(1, M+1]\setminus\overline{\ul}$ is $(1, q_c)$.

It was shown in \cite[Lemma 2.11]{Komornik_Kong_Li_2015_1} (see also, \cite[Theorem 2.6]{Kong_Li_2015}) that the entropy function $H$ is constant on each connected component $(q_0, q_0^*)$ of $(1, M+1]\setminus\overline{\ul}$. It is natural to ask what are the maximal intervals $[p_L, p_R]$ for which
\[ H(q)=H(p_L)\quad\textrm{for all }q\in[p_L, p_R]?\]
These maximal intervals $[p_L, p_R]$ are called the \emph{entropy plateaus} of $H$. It is clear that the first entropy plateau is $(1, q_c],$ since $H(q)>0$ if and only if $q>q_c$ (cf.~\cite{Glendinning_Sidorov_2001}, see also, \cite{Kong_Li_Dekking_2010}). One might expect that for each connected component $(q_0, q_0^*)$ of $(q_c, M+1]\setminus\overline{\ul},$ the closed interval $[q_0, q_0^*]$ would be an entropy plateau of $H$.  However this is not true. As we will see, each entropy plateau contains infinitely many connected components of $(q_c, M+1]\setminus\overline{\ul}$ (see Remark \ref{rem:53}). In this paper we give a complete description of the entropy plateaus of $H$ (see Theorem \ref{th2}).

Recall from (\ref{eq:12}) that $\E$ is the set of bifurcation bases. Following on from our investigation into the entropy plateaus of $H,$ we give a complete characterization of the set $\E$. As a consequence of this characterization, we show that $\E$ is a Cantor set of full Hausdorff dimension (see Theorem \ref{th3}).

By \cite[Theorem 1.8]{DeVries_Komornik_2008} it follows that for $q\in(1,M+1]$, $(\ub_q, \sigma)$ may not be a subshift. Inspired by \cite{Komornik_Kong_Li_2015_1} we introduce the set
\begin{equation}\label{eq:13}
\VB_q:=\set{(x_i)\in\set{0,1,\ldots,M}^\f: \overline{(\al_i(q))}\lle\si^n((x_i))\lle(\al_i(q))\textrm{ for all }n\ge 0}.
\end{equation}
It is a consequence of Lemma \ref{lem:24} (see below) that $\VB_q=\emptyset$ for any $q\in(1, q_G)$, where $q_G$ is the generalized golden ratio defined in (\ref{eq:21}).
Moreover, for any $q\in[q_G, M+1]$ it was shown in \cite[Lemma 2.6]{Komornik_Kong_Li_2015_1} that  $(\VB_q, \si)$ is a non-empty subshift, and in \cite[Proposition 2.8]{Komornik_Kong_Li_2015_1} that
 \begin{equation}\label{eq:14}
 h(\VB_q)=h(\ub_q)=H(q) \quad\textrm{for all }q\in[q_G, M+1].
 \end{equation}
As a result of \eqref{eq:14}, to answer many of the questions we are interested in, it is sufficient to study the subshift $(\VB_q, \si)$.

From the perspective of dynamical systems, the subshift $(\VB_q, \si)$ has a close connection with open dynamical systems (cf.~\cite{Alcaraz_Barrera_2014}). For more information on open dynamical systems we refer the readers to \cite{Bundfuss_Kruger_Troubetzkoy_2011}, the survey paper \cite{Sidorov_2003_survey} and the references therein.

Topological transitivity (see Definition \ref{def:23}) is a fundamental property in dynamical systems. In this paper we give a complete characterization of those bases for which the subshift $(\VB_q, \si)$ is topologically transitive (see Theorem \ref{th1}). Moreover, we show that the generalized golden ratio $q_G$ is the smallest base $q\in(1,M+1]$ for which $(\VB_q, \si)$ is transitive. We also determine the smallest base  $q\in(1,M+1]$ for which $(\VB_q, \si)$ is not transitive, and the smallest base $q$ in $(1,M+1]$  for which $h(\VB_q)>0$ and $(\VB_q, \si)$ is transitive.

Until now most of the dimensional results for $\u_q$ have focused on determining its Hausdorff dimension. In this paper we show that the box dimension of $\u_q$ coincides with its Hausdorff dimension for all $q\in(1, M+1]$ (see Theorem \ref{th4}).

The paper is arranged as follows. In Section \ref{sec:2} we recall some relevant properties of unique expansions and state our main results. In Section \ref{sec:3} we prove Theorem \ref{th1} and give a complete characterization of the bases $q\in(1, M+1]$ for which the subshift $(\VB_q, \si)$ is topologically transitive.  In Section \ref{sec:4} we study properties of irreducible and $*$-irreducible intervals and prepare for the proof of Theorem \ref{th2}, which is given in Section \ref{sec:5}. In Section \ref{sec:6}, via an application of Theorem \ref{th2}, we give a characterization of the set $\E$ and prove Theorem \ref{th3}. In our final section we prove Theorem \ref{th4} and conclude that the Hausdorff dimension and box dimension of $\u_q$ coincide for all $q\in(1, M+1]$.

\section{Preliminaries and main results}\label{sec:2}

\noindent In this section we recall some relevant concepts from the study of unique expansions and state our main results. First we introduce some notation from symbolic dynamics. For more information on these topics we refer the reader to \cite{Lind_Marcus_1995}.

\subsection{Symbolic dynamics and unique expansions}
Fix a positive integer $M$.  Let $\set{0,1,\ldots,M}^\f$ be the set of all one sided sequences $(c_i)=c_1c_2\ldots,$ where each term $c_i$ is from the \emph{alphabet} $\set{0,1,\ldots,M}$. Then the pair $(\set{0,1,\ldots, M}^\f, \si)$ is called a \emph{one sided full shift}, where
  $\si: \set{0,1,\ldots,M}^\f \ra \set{0,1,\ldots,M}^\f$ is \emph{the one sided left-shift map} defined by $\si((c_i))=(c_{i+1})$.

By a \emph{word} $\om$ we mean a finite string of digits $\om=\om_1\ldots\om_n$ with each $\om_i\in\set{0,1,\ldots,M}$. The length of $\om$ is denoted by $|\om|$.
Given two finite words $\om=\om_1\ldots \om_n$ and $\de=\de_1\ldots\de_m$, we denote by $\om\de=\om_1\ldots \om_n\de_1\ldots \de_m$ their concatenation. For $k\in\N$ we denote by $\om^k=\overbrace{\om\cdots\om}^{k}$ the $k$ times concatenation of $\om$ with itself, and by $\om^\f=\om\om\cdots$ the infinite concatenation of $\om$ with itself. A sequence $(x_i)\in\set{0,1,\ldots, M}^\f$ is called \emph{periodic} if there exists $n\ge 1$ such that $(c_i)=(c_1\ldots c_n)^\f$. In this case, the smallest such $n$ is called the \emph{period} of $(c_i)$, and the word $c_1\ldots c_n$ is called the {\emph{period block}} of $(c_i)$.

For a sequence $(c_i)\in \set{0,1,\ldots,M}^\f,$ we denote by
\[\overline{(c_i)}=(M-c_1)(M-c_2)\cdots\]
 its \emph{reflection}. Similarly, for a word $\om=\om_1\om_2\ldots\om_n$ its reflection is written as $\overline{\om}=(M-\om_1)(M-\om_2)\cdots(M-\om_n)$. If $\om_n>0$ then we write
 $\om^-=\om_1\ldots\om_{n-1}(\om_n-1).$
  If $\om_n<M$ then we put $\om^+=\om_1\ldots\om_{n-1}(\om_n+1)$.
Throughout this paper we will use the lexicographic ordering on sequences and  words. For two sequences $(c_i)$ and $(d_i)$ we write
\[(c_i)\prec (d_i)\]
 if there exists $n\in\N$ such that $c_i=d_i$ for all $i<n$, and $c_n<d_n$. Moreover, we write $(c_i)\preccurlyeq (d_i)$ if $(c_i)\prec (d_i)$ or $(c_i)=(d_i)$. Symmetrically, we say that $(c_i)\succ(d_i)$ (or $(c_i)\succcurlyeq(d_i)$) if $(d_i)\prec(c_i)$ (or $(d_i)\preccurlyeq(c_i)$).

 For $q\in(1,M+1]$ we denote by $\al(q)=(\al_i(q))$  the \emph{quasi-greedy} $q$-expansion of $1$, i.e., $(\al_i(q))$ is the lexicographically largest \emph{infinite} $q$-expansion of $1$. Here an expansion is called \emph{infinite} if it does not end with an infinite string of zeros.
  The following lexicographic characterization of $\al(q)$ was established in \cite[Theorem 2.2]{Baiocchi_Komornik_2007}.
\begin{lemma}
\label{lem:21}
  The map $q\mapsto \al(q)$ is a strictly increasing bijection from $(1, M+1]$ onto the set of all infinite sequences $(\al_i)$ satisfying
\[
\al_{n+1}\al_{n+2}\ldots\lle \al_1\al_2\ldots\quad\textrm{whenever}\quad \al_n<M.
\]
\end{lemma}

Recall that $\ul=\ul(M)$ is the set of   bases $q\in(1,M+1]$ for which $1$ has a unique $q$-expansion with respect to   the alphabet $\set{0,1,\ldots,M}$. The following lexicographic characterization  of $\ul$ was established in \cite[Theorem 2.5]{DeVries_Komornik_Loreti_2016}:
\[
\ul=\set{q\in(1,M+1): \overline{\al(q)}\prec\si^n(\al(q))\prec \al(q)\textrm{ for all }n\ge 1}\cup\set{M+1}.
\]
Moreover, the following lexicographic description of its closure $\overline{\ul}$  was given in \cite[Theorem 3.9]{DeVries_Komornik_Loreti_2016}:
\[
\overline{\ul}=\set{q\in(1,M+1]: \overline{\al(q)}\prec \si^n(\al(q))\lle \al(q)\textrm{ for all }n\ge 0}.
\]
Inspired by \cite[Definition 3.2]{DeVries_Komornik_Loreti_2016}, we also define
\[
\vl=\set{q\in(1, M+1]: \overline{\al(q)}\lle\si^n(\al(q))\lle \al(q)\textrm{ for all } n\ge 0}.
\]

Therefore, $\ul\subseteq\overline{\ul}\subseteq\vl$.
The topological properties of $\ul$,  $\overline{\ul}$ and $\vl$ were investigated in \cite{DeVries_Komornik_Loreti_2016}. They proved that the difference sets $\overline{\ul}\setminus\ul$ and $\vl\setminus\overline{\ul}$ are both countable. Furthermore, $\overline{\ul}\setminus\ul$ is dense in $\overline{\ul}$, and $\vl\setminus\overline{\ul}$ is discrete and dense in $\vl$. { Indeed, the set $\vl$ is the union of a Cantor set $\overline{\ul}$ and a countable set $\vl\setminus\overline{\ul}$, see \cite{DeVries_Komornik_Loreti_2016}.}
 By \cite[Lemma 3.5]{DeVries_Komornik_Loreti_2016} it follows that the smallest base of $\vl$ is the \emph{generalized golden ratio}
\begin{equation}\label{eq:21}
 q_G =\left\{\begin{array}{lll}
k+1 &\textrm{if} & M=2k,\\
 \frac{k+1+\sqrt{k^2+6k+5}}{2}&\textrm{if}& M=2k+1.
\end{array}\right.
\end{equation}
At the generalized golden ratio the quasi-greedy expansion satisfies $\al(q_G)=k^\f$ if $M=2k$, and $\al(q_G)=((k+1)k)^\f$ if $M=2k+1$. The generalized golden ratio $q_G$ was first introduced by the second author in \cite{Baker_2014}.

Recall from \cite[Definition 1.2.1]{Lind_Marcus_1995}  the following.

\begin{definition}\label{def:22}
Let $X\subseteq\set{0,1,\ldots,M}^\f$. Then $(X, \sigma)$ is called a \emph{subshift} if there is a set $\mathfrak{F}$ of \emph{forbidden words} such that $X=X_{\mathfrak{F}}$, where $$X_{\mathfrak{F}}: = \set{(c_i)\in\set{0,1,\ldots,M}^\f : (c_i) \textrm{ does not containing any word in } \mathfrak{F}}.$$

If $\mathfrak{F}$ can be chosen to be a finite set then $(X, \sigma)$ is called a \emph{subshift of finite type}.  If $\mathfrak{F}=\emptyset$, then $X=X_{\emptyset}=\set{0,1,\ldots,M}^\f$ and $(X, \sigma)$ is called a \emph{full shift}.
\end{definition}

We always write a subshift as a pair $(X, \sigma)$ to emphasise that the left shift $\si$ is an operator on $X$. Clearly, by Definition \ref{def:22} it follows that for every subshift $(X, \sigma)$, the set $X$ is a closed and forward $\si$-invariant subset of $\set{0,1,\ldots M}^\f$, i.e., $\si(X) =X$.

For a subshift $(X, \si)$ we denote by $\L(X)$ the set of finite sub-words of sequences in $X$ together with the empty word $\epsilon$. The set $\L(X)$ is commonly referred to as the \emph{language} of $X$ and words appearing in $\L(X)$ are called \emph{admissible}.

\begin{definition}\label{def:23}
A subshift $(X, \si)$ is said to be \emph{topologically transitive} (or simply \emph{transitive}), if for all $\om, \nu\in\L(X)$ there exists $\de\in\L(X)$ such that $\om\de\nu\in\L(X)$.
\end{definition}

Note that $(\ub_q, \sigma)$ is not always a subshift. Consequently we consider the subshift $(\VB_q, \sigma)$ defined in (\ref{eq:13}), i.e.,
 \[
\VB_q=\set{(x_i)\in\set{0,1,\ldots,M}^\f: \overline{\al(q)}\lle\si^n((x_i))\lle\al(q)\textrm{ for all }n\ge 0}.
\]

\begin{lemma}\label{lem:24}
$\VB_q\ne\emptyset$ if and only if $q\in[q_G, M+1]$.
\end{lemma}
\begin{proof}
First we prove sufficiency. Note by Lemma \ref{lem:21} the map $q\ra \al(q)$ is increasing. This implies that $\VB_{q_G}\subseteq\VB_q$ for all $q\in[q_G, M+1]$. By \cite[Lemma 3.5]{DeVries_Komornik_Loreti_2016} it follows that
$\al(q_G)\in\VB_{q_G}$, and therefore $\al(q_G)\in\VB_q$ for all $q\in[q_G, M+1].$

Now we prove the necessity. If $M=2k$, then by (\ref{eq:21}) and Lemma \ref{lem:21} we have
$\al(q)\prec\al(q_G)=k^\f$ for any $q\in(1,q_G)$. This gives $\al(q)\prec\overline{\al(q)}$, and therefore  $\VB_{q}=\emptyset$.
 Similarly, for $M=2k+1$   we still have $\VB_q=\emptyset$.
\end{proof}

By Lemma \ref{lem:24} it suffices to consider $\VB_q$ for $q\in[q_G, M+1]$.
Now we recall some properties of the subshift $(\VB_q, \si)$ and the set $\vl$.
Observe that $q_G$ and $M+1$ are the smallest and the largest elements of  $\vl$ respectively.  Then by \cite[Theorem 1.3]{DeVries_Komornik_Loreti_2016} it follows  that
\[
[q_G, M+1]\setminus\vl=(q_G, M+1)\setminus\vl=\bigcup(q_\ell, q_r),
\]
where the union on the right hand-side is pairwise disjoint and countable.
Here the open intervals $(q_\ell, q_r)$ are referred to as the \emph{basic intervals} of $[q_G, M+1]\setminus\vl$.

The following properties of the subshift $(\VB_q, \si)$ and the set $\vl$ were essentially established  in \cite[Theorems 1.7 and 1.8]{DeVries_Komornik_2008} (see also, \cite{DeVries_Komornik_Loreti_2016}).
\begin{lemma}\label{lem:25}
\begin{enumerate}
\item
Let $(q_\ell, q_r)$ be a basic interval of $[q_G, M+1]\setminus\vl$. Then $\VB_q=\VB_{q_\ell}$ for all $q\in[q_\ell, q_r)$. Furthermore, $(\VB_{q_\ell}, \si)$ is a subshift of finite type.

\item If $p<q$ and $(p, q]\cap\vl\ne\emptyset,$ then $\VB_p$ is a proper subset of $\VB_q$.

\item For each $q\in\vl\setminus\ul$ there exists a basic interval $(q_\ell, q_r)$ of $[q_G, M+1]\setminus\vl$ such that $q=q_\ell$.
\end{enumerate}
\end{lemma}

Now we consider the transitivity of $(\VB_q, \si)$. By Lemma \ref{lem:25} (1) it follows that $\VB_q$ is \emph{stable}  in any basic interval of $[q_G, M+1]\setminus\vl$, i.e.,   for each basic interval $(q_\ell, q_r)$ of $[q_G, M+1]\setminus\vl$ we have  $\VB_q=\VB_{q_\ell}$ for all $q\in[q_\ell, q_r)$.  So to determine those $q\in[q_G,M+1]$ for which $(\VB_q, \si)$ is transitive, it suffices to determine those $q\in [q_G, M+1]\cap\vl$ for which $(\VB_q, \si)$ is transitive.

Let
\[
 {\vb}= \set{(a_i)\in\set{0,1,\ldots,M}^\f: \overline{(a_i)}\lle\si^n((a_i))\lle(a_i)\textrm{ for all }n\ge 0}.
\]
Note that any sequence in $\vb$ is infinite. Using Lemma \ref{lem:21} one can verify that the map
\begin{align*}
\Phi:\quad& \vl\quad \longrightarrow\quad \vb\\
&q~\quad\mapsto\quad \al(q)
\end{align*}
 is a bijection. Thus the study of bases in $\vl$ is equivalent to the study of sequences in $\vb$.  Note that $q_G$ is the smallest base in $\vl$. By Lemma \ref{lem:21} this implies that $\al(q_G)$ is the lexicographically smallest sequence of $\vb$.

In order to characterize those $q$ for which $(\VB_q, \si)$ is transitive we introduce a special class of sequences in  $\vb$.

 \begin{definition}\label{def:26}
A  sequence $(a_i)\in\vb$ is said to be \emph{irreducible} if
  \[
  a_1\ldots a_j(\overline{a_1\ldots a_j}\,^+)^\f\prec (a_i)\quad\textrm{whenever}\quad (a_1\ldots a_j^-)^\f\in {\vb}.
  \]
  \end{definition}
  { We mention that the name `irreducible' of a sequence $(a_i)\in\vb$ is meaningful since for $q=\Phi^{-1}((a_i))$ the corresponding subshift $(\vb_q, \sigma)$ is irreducible (transitive), see Theorem \ref{th1} below.}
By (\ref{eq:21}) one can verify that $\al(q_G)$ is irreducible, and hence it is the smallest irreducible sequence.

To describe transitivity we also need to introduce the base $q_T$, called the \emph{transitive base}, which is defined implicitly via the equation
  \begin{equation}\label{eq:22}
\al(q_T)=\left\{
\begin{array}
  {lll}
  (k+1)\,k^\f&\textrm{if}& M=2k,\\
  (k+1)\,((k+1)k)^\f&\textrm{if}&M=2k+1.
\end{array}
\right.
\end{equation}
Clearly $\al(q_T)\in\vb$ and therefore $q_T\in\vl$. Indeed we can show that $q_T\in\ul$.
By (\ref{eq:21}) and Lemma \ref{lem:21} we have $q_T>q_G$.

\subsection{Main results} Our first result is for the transitivity of $(\VB_q, \si)$.

\begin{main}\label{th1}
 Let $q\in[q_G, M+1]\cap\vl$. Then $(\VB_q, \si)$ is  transitive if and only if    $\al(q)$ is  irreducible, or $q=q_T$.
 \end{main}
 \begin{remark}\label{rem:27}
   \begin{enumerate}
   \item As previously remarked, to determine those $q\in[q_G,M+1]$ for which $(\VB_q, \si)$ is transitive, it suffices to determine those $q\in [q_G, M+1]\cap\vl$ for which $(\VB_q, \si)$ is transitive. Consequently Theorem \ref{th1} provides a complete classification of those $q\in[q_G,M+1]$ for which $(\VB_q, \si)$ is transitive.
     \item Note that the generalized golden ratio $q_G$ is the smallest base $q$ for which $\al(q)$ is irreducible, and by Lemma \ref{lem:24} we have $\VB_q=\emptyset$ if $q< q_G$. So by Theorem \ref{th1} it is also the smallest base $q$ for which $(\VB_q, \si)$ is transitive.
    In Lemma  \ref{lem:33} and Proposition \ref{prop:36} (see below) we show that the base $q_T$  is the smallest base $q$ for which $(\VB_q, \si)$ is transitive and $h(\VB_q)>0$. This explains why $q_T$ is  referred to as the  {transitive base}. Moreover, we prove that $q_{NT}$ defined in (\ref{eq:31}) is the smallest base $q$ for which $(\VB_q, \si)$ is not transitive (see Lemmas \ref{lem:31} and  \ref{lem:32}).

     \item In Lemma \ref{lem:33} we show   that there exists a subinterval $I\subseteq[q_G, M+1]$ such that $(\VB_q, \si)$ is not transitive for any $q\in I$.
On the other hand,  take $q\in[q_G, M+1]$ such that $\al(q)=(M^n0)^\f$ for some $n\ge 2$, then one can  check that $\al(q)$ is irreducible. Furthermore, by \cite[Theorem 1.2]{DeVries_Komornik_Loreti_2016} we have $q\in\vl\setminus\ul$. Then by Lemma \ref{lem:25} (3) it follows that $q=q_\ell$ is the left endpoint of a basic interval of $[q_G, M+1]\setminus\vl$.  By Lemma \ref{lem:25} (1) and Theorem \ref{th1} this implies that  the following set
 \[
  \mathcal{T}=\set{q\in(1,M+1]: (\VB_q, \si)\textrm{ is a transitive subshift}}
 \]
 contains a subinterval (in fact $\mathcal{T}$ contains infinitely many subintervals).
   \end{enumerate}
 \end{remark}

For $q\in(1, M+1]$ we recall that $H(q)=h(\ub_q)$ is the topological entropy of $\ub_q$.
Now we turn our attention to the entropy plateaus of the function $H$.   Let  $(\tau_i)_{i=0}^\f=01101001\ldots$ be the classical \emph{Thue-Morse sequence} defined as follows: $\tau_0=0$, and if $\tau_i$ has already been defined for some $i\ge 0$, then $\tau_{2i}=\tau_i$ and $\tau_{2i+1}=1-\tau_i$. For more information on the classical Thue-Morse sequence $(\tau_i)_{i=0}^\f$ we refer to the survey paper \cite{Allouche_Shallit_1999}.  By a result of \cite{Komornik_Loreti_2002}, the smallest base of $\ul$ is the Komornik-Loreti constant $q_c$, which is defined using the classical {Thue-Morse sequence} $(\tau_i)_{i=0}^\f$.  To be more explicit,
\begin{equation}\label{eq:23}
\al(q_c)=(\la_i)_{i=1}^\f:=\left\{
\begin{array}{lll}
(k+\tau_i-\tau_{i-1})_{i=1}^\f&\textrm{if}& M=2k,\\
(k+\tau_i)_{i=1}^\f&\textrm{if}& M=2k+1.
\end{array}
\right.
\end{equation}
 { Notice that the Thue-Morse sequence $(\tau_i)_{i=0}^\f$ has indexes starting at $0$, whereas the sequence $(\lambda_i)_{i=1}^\f$ has indexes starting at $1$. } We point out that the sequence $(\lambda_i)_{i=1}^\f$ in \eqref{eq:23} depends on $M$. It follows from the definition of $(\tau_i)_{i=0}^\f$  that  the sequence $(\la_i)$ satisfies the following recursive equations (cf.~\cite{Komornik_Loreti_2002}):
\begin{equation}\label{eq:24}
\la_1\ldots \la_{2^{n+1}}=\la_1\ldots \la_{2^n}\overline{\la_1\ldots \la_{2^n}}^+\quad\textrm{for all}\quad n\ge 0.
\end{equation}
Then the sequence $\al(q_c)$ begins with
\begin{align*}
&(k+1)k(k-1)(k+1)\; (k-1)k(k+1)k\cdots&\textrm{if}&\quad M=2k,\\
&(k+1)(k+1)k(k+1)\; k k (k+1)(k+1)\cdots&\textrm{if}&\quad M=2k+1.
\end{align*}
By (\ref{eq:21})--(\ref{eq:23}) and Lemma \ref{lem:21}  it follows that $q_G<q_c<q_T$.

The significance of $q_c$ in the study of the univoque set $\u_q$ is well demonstrated by  \cite[Theorems 1.1 and 1.2]{Komornik_Kong_Li_2015_1}. It is shown that $H(q)=0$ for all $q\in(1, q_c]$ and
$H(q)>0$ for all $q\in(q_c, M+1]$.
This implies  that the first entropy plateau of $H$ is $(1, q_c]$. Hence, in the following we may restrict out attention to the interval $(q_c, M+1]$.

Note by (\ref{eq:14}) we have $H(q)=h(\VB_q)$ for all $q\in(q_c, M+1]$. As we will see, the characterization of those $q$ for which $(\VB_q, \si)$ is transitive will be useful when it comes to describing the entropy plateaus of $H$. However, for any $q\in(q_c, q_T),$ it will be shown in Lemma \ref{lem:33} that $(\VB_q, \si)$ is not transitive. This makes determining the entropy plateaus in the interval $(q_c, q_T)$ more intricate.

In order to solve this  problem  we introduce the following sequences in $\vb$.
 For $n\in\N$ we define the sequence $\xi(n)=(\xi_i(n))$ by
\begin{equation}\label{eq:25}
\xi(n):=\left\{
\begin{array}
  {lll}
  \la_1\ldots\la_{2^{n-1}}(\overline{\la_1\ldots\la_{2^{n-1}}}\,^+)^\f&\textrm{if}& M=2k,\\
  \la_1\ldots\la_{2^n}(\overline{\la_1\ldots\la_{2^n}}\,^+)^\f&\textrm{if}& M=2k+1.
\end{array}\right.
\end{equation}
Here   the sequence $(\la_i)=\al(q_c)$ is defined in \eqref{eq:23} which depends  on $M$.
Then by (\ref{eq:22}) and \eqref{eq:25} we have $\al(q_T)=\xi(1)$. In Lemma \ref{lem:43} we show that each sequence $\xi(n)$ is an element of $\vb$, and   $\xi(n)$ strictly decreases to $(\la_i)=\al(q_c)$ as $n\ra\f$. Here the convergence of sequences in $\set{0,1,\ldots, M}^\f$ is defined with respect to the order topology.

The following definition should be compared with Definition \ref{def:26} of an irreducible sequence.
\begin{definition}\label{def:28}
A sequence $(a_i)\in\vb$ is said to be \emph{$*$-irreducible} if  there exists $n\in\N$ such that
$\xi(n+1)\lle (a_i)\prec \xi(n),$
 and
\[
a_1\ldots a_j(\overline{a_1\ldots a_j}\,^+)^\f\prec (a_i),
\]
whenever
\[
 (a_1\ldots a_j^-)^\f\in\vb\quad\textrm{and}\quad j>
 \left\{\begin{array}{lll}
 2^n&\textrm{if}& M=2k,\\
 2^{n+1}&\textrm{if}& M=2k+1.
 \end{array}\right.
\]
\end{definition}

{  Notice that for  a $*$-irreducible  sequence $(a_i)\in\vb$ we only need to verify the inequalities $a_1\ldots a_j(\overline{a_1\ldots a_j}^+)^\f\prec (a_i)$  for all large integers $j$ satisfying  $(a_1\ldots a_j^-)^\f\in\vb$. Comparing with Definition \ref{def:26} of irreducible sequence this explains  the name `$*$-irreducible'.}
Building upon the notion of irreducible and $*$-irreducible sequences we introduce the following class of intervals.

\begin{definition}\label{def:29}
An interval $[p_L, p_R]\subseteq(q_c, M+1]$ is called an \emph{irreducible interval} (or, a \emph{$*$-irreducible interval}) if $\al(p_L)$ is irreducible (or $*$-irreducible), and there exists a word $a_1\ldots a_m$ with $a_m<M$ such that
\[\al(p_L)=(a_1\ldots a_m)^\f,\quad \al(p_R)=a_1\ldots a_m^+(\overline{a_1\ldots a_m})^\f.\]
\end{definition}
We point out that the irreducible interval (or $*$-irreducible interval) $[p_L, p_R]$ is well-defined (see Lemma \ref{lem:41} below).

Our second  result is for the   entropy plateaus of $H(q)=h(\ub_q)$.
\begin{main}
  \label{th2}
  The interval $[p_L, p_R]\subseteq(q_c, M+1]$ is an entropy  plateau of $H$ if and only if $[p_L, p_R]$ is an irreducible  or a $*$-irreducible interval.
 \end{main}
\begin{remark}
  \begin{enumerate}
    \item We show in Lemma \ref{lem:410} that for each irreducible or $*$-irreducible interval $[p_L, p_R]$ the left endpoint  is contained in $\overline{\ul}\setminus\ul$ and the right endpoint is contained in $\ul$.

    \item We also investigate  the topological properties of the irreducible and $*$-irreducible  intervals. For example, all of these irreducible and $*$-irreducible intervals are pairwise disjoint (see Lemma \ref{lem:46}). Moreover, the left end point of each irreducible interval can be approximated from below by left end points of irreducible intervals and the right end point of each irreducible interval can be approximated by left end points of irreducible intervals. We also find similar approximations for the end points of a $*$-irreducible interval. (see Proposition \ref{prop:411}).

    \item In Lemma \ref{lem:44} we show that the $*$-irreducible intervals are all contained in $(q_c, q_T)$, and  the irreducible intervals are all contained in $(q_T, M+1)$. In particular, it follows from the proof of Propositions \ref{prop:52} and \ref{prop:511} that the union of all $*$-irreducible intervals covers $(q_c, q_T)$ up to a set of measure zero, and the union of all irreducible intervals covers $(q_T, M+1)$ up to a set of measure zero.  Therefore, by (2), it follows that the Komornik-Loreti constant $q_c$ can be approximated by end points of $*$-irreducible intervals from above, the base $M+1$ can be approximated by end points of $*$-irreducible intervals from below. Moreover, the transitive base $q_T$ can be approximated by end points of irreducible intervals from above and by end points of $*$-irreducible intervals from below. 
    
     \end{enumerate}
\end{remark}

Recall from (\ref{eq:13}) that  $\E$ is the set of bifurcation bases. Then by Theorem \ref{th2} we can immediately  rewrite   $\E$ as
\[
\E=[q_c, M+1]\setminus\bigcup(p_L, p_R),
\]
where the union is taken over all entropy plateaus $[p_L, p_R]$ of $H$. We mention that the union in the above equation is countable and pairwise disjoint.
Note that $H$ is constant on each connected component of $[q_c, M+1]\setminus\overline{\ul}$. This implies that $\E\subseteq\overline{\ul}$. Since $\overline{\ul}$ is a Lebesgue null set,  so is  $\E$.

In the following theorem we characterize the set $\E$ in terms of irreducible and $*$-irreducible sequences and show that $\E$ is a Cantor set of full Hausdorff dimension.
\begin{main}
  \label{th3}
  Let $\E$ be the set of bifurcation  bases. Then
 \begin{enumerate}
 \item $
  \E=\overline{\set{q\in(q_c, M+1]: \al(q)\textrm{ is irreducible or} *\textrm{-irreducible}}}.
 $

  \item $\E$ is a Cantor set and $\dim_H\E=1$.
  \end{enumerate}
\end{main}

Up until now we knew little about the box dimension of $\u_q$. Our final result shows that the box dimension of $\u_q$  coincides with its Hausdorff dimension for any $q\in(1,M+1]$.

\begin{main}\label{th4}
  For all $q\in(1,M+1]$ we have $\dim_B\u_q=  \dim_H\u_q$.
\end{main}

\subsection{List of notation}
\begin{itemize}
\item
\begin{itemize}
\item $\pi_q$ denotes  the projection map:  $\pi_q((x_i))=\sum_{i=1}^\f\frac{x_i}{q^i}$.

\item $\alpha(q)=(\alpha_i(q))$ denotes the quasi-greedy $q$-expansion of $1$, i.e., the lexicographically largest $q$-expansion of $1$ with infinitely many non-zero elements.

\item $\overline{(c_i)}$ denotes the reflection of a sequence $(c_i),$ i.e. $\overline{(c_i)}=(M-c_i).$

\item For a word $w=w_1\ldots w_n,$ we let $w^{-}=w_1\ldots w_{n-1} (w_n-1)$ if $w_n>0,$ and $w^{+}=w_1\ldots w_{n-1} (w_n+1)$ if $w_n<M.$

\item $\sigma$ denotes the left shift.

\item $B_n(X)$ is the set of subwords of length $n$ that appear in elements of some sequence space $X$.

\item $\epsilon$ denotes the empty word. 

\item  $B_{*}(X)=\bigcup_{n=1}^{\infty}B_n(X)\cup \{\epsilon\}$ denotes the language of $X$.

\item $\R(\a)$ denotes the reflection recurrence word of a primitive word $\a$, see Definition \ref{def:311}.

\end{itemize}

\item \begin{itemize}
\item $\u_q$ denotes the univoque set, i.e. those $x$ with a unique $q$-expansion. Then
\[
\u_q=\set{\pi_q((x_i)):     
\begin{array}{lll}
\si^n((x_i))\prec (\al_i(q)) &\textrm{whenever}& x_n<M,\\
\si^n((x_i)) \succ \overline{(\al_i(q))}&\textrm{whenever}& x_n>0.
\end{array}
 }.
\]

\item $\ub_q$ is the set of corresponding $q$-expansions  of elements of $\u_q$, i.e., $\ub_q=\pi_q^{-1}(\u_q)$.

\item  $\VB_q=\set{(x_i)\in\set{0,1,\ldots,M}^\f: \overline{(\al_i(q))}\lle\si^n((x_i))\lle(\al_i(q))\textrm{ for all }n\ge 0}$. 

\item $\W_q=\set{\pi_q((x_i)): \overline{\al(q)}\prec\si^n((x_i))\prec\al(q)\textrm{ for all } n\ge 0}$, see Section \ref{sec:8}.

\item $\WB_q$ denotes the set of corresponding $q$-expansions of elements of $\W_q$, i.e., $\WB_q=\pi_q^{-1}(\W_q)$.  

\item $\ul=\{q\in(1,M+1]: 1 \textrm{ has a unique } q\textrm{-expansion}\}$. Then
\[
\ul =\set{q\in(1,M+1]: \overline{\al(q)}\prec\si^n(\al(q))\prec \al(q)\textrm{ for all } n\ge 1}\cup\set{M+1}.
\]

\item $\overline{\ul}$ denotes the topological closure of $\ul$. Then 
\[
\overline{\ul}=\set{q\in(1,M+1]: \overline{\al(q)}\prec\si^n(\al(q))\lle \al(q)\textrm{ for all } n\ge 0}.
\]

\item $\vl=\set{q\in(1, M+1]: \overline{\al(q)}\lle\si^n(\al(q))\lle \al(q)\textrm{ for all } n\ge 0}.$

\item  $\vb$ denotes the set of corresponding expansions of $1$ for some $q\in\vl$.  Then 
\[{\vb}= \set{(a_i)\in\set{0,1,\ldots,M}^\f: \overline{(a_i)}\lle\si^n((a_i))\lle(a_i)\textrm{ for all }n\ge 0}.\]

\item $H(q)=h(\VB_q)$ denotes the topological entropy of $\VB_q$.

\item $\E$ denotes  the bifurcation set. Then 
\[\E=\set{q\in(1,M+1]:  \forall \epsilon>0,\, \exists p\in(q-\epsilon, q+\epsilon)\textrm{ such that }H(p)\neq H(q)}.\]

\item $\I$ denotes the set of bases $q\in[q_T, M+1]$ such that $\al(q)$ is irreducible, see (\ref{eq:412}).

\item $\I^*$ denotes the set of bases $q\in(q_c, q_T)$ such that $\al(q)$ is $*$-irreducible, see (\ref{eq:412}).
\end{itemize}

\item 
\begin{itemize}
\item $q_G$ denotes the generalized golden ratio defined in equation \eqref{eq:21}.

\item $q_c$ denotes the Komornik-Loreti constant   defined implicitly in \eqref{eq:23}. It is the smallest element of $\ul$.

\item  $(\lambda_i)_{i=1}^{\infty}$ is the quasi-greedy expansion of $q_c$, i.e., $\al(q_c)=(\lambda_i)$.

\item $(\tau_i)_{i=0}^{\infty}$ denotes the classical Thue-Morse sequence.

\item $q_{NT}$ denotes the smallest base $q$ in which $(\VB_q, \si)$ is not transitive, see (\ref{eq:31}). 

\item  $q_T$ denotes the transitive base defined implicitly in equation \eqref{eq:22}.

\item $q_{T_n}$ denotes a base in the interval $(q_c, q_T]$ (see Lemma \ref{lem:43}).

\item $\xi(n)=\al(q_{T_n})$ is the quasi-greedy  expansion given in (\ref{eq:25}):
 \begin{equation*}
\xi(n)=\left\{
\begin{array}
  {lll}
  \la_1\ldots\la_{2^{n-1}}(\overline{\la_1\ldots\la_{2^{n-1}}}\,^+)^\f&\textrm{if}& M=2k,\\
  \la_1\ldots\la_{2^n}(\overline{\la_1\ldots\la_{2^n}}\,^+)^\f&\textrm{if}& M=2k+1.
\end{array}\right.
\end{equation*}
\end{itemize}
\end{itemize}

\section{Topological transitivity of $(\VB_q, \si)$}\label{sec:3}

\noindent In this section we give a complete characterization of those $q$ for which $(\VB_q, \si)$ is transitive and prove Theorem \ref{th1}. Recall that $(\VB_q, \si)$ is the subshift defined in (\ref{eq:13}). Note that by Lemma \ref{lem:24} we have $\VB_q=\emptyset$  for any $q\in(1, q_G)$, where $q_G$ is the generalized golden ratio defined in (\ref{eq:21}). Moreover, by Lemma \ref{lem:25} (1) the set $\VB_q$ is stable in any basic interval of $[q_G, M+1]\setminus\vl$. This means for each connected component $(q_\ell, q_r)$ of $[q_G, M+1]\setminus\vl$ we have $\VB_q=\VB_{q_\ell}$ for any $q\in[q_\ell, q_r)$. Therefore, as previously remarked, to classify those $q$ for which $(\VB_q, \si)$ is transitive, it suffices to classify those $q\in[q_G, M+1]\cap\vl$ for which $(\VB_q, \si)$ is transitive. Our method of proof is based upon techniques from combinatorics on words (cf.~\cite{Allouche_Shallit_2003}). In particular, we prove the sufficiency of Theorem \ref{th1} with the aid of reflection recurrence words (see Definition \ref{def:311} below).

Recall from (\ref{eq:22}) that $q_T$ is the transitive base. The proof of Theorem \ref{th1} for $q\in [q_G,  q_T]$ is much easier than that for $q\in(q_T, M+1]$. For this reason we split the proof of  Theorem \ref{th1} into the following two subsections.

\subsection{Transitivity of $(\VB_q, \si)$ for $q\in[q_G, q_T]$.}
Observe that $q_G$ is the smallest base in $\vl$ and $q_c$ is the smallest base in $\overline{\ul}$. Then by \cite[Theorem 1.3]{DeVries_Komornik_Loreti_2016}  it follows   that  $\vl\cap[q_G, q_c)$ is a discrete set which contains infinitely many elements. Furthermore, we can list these elements $(q_i)$ in an increasing order in the following way:
\[
 q_G=q_1<q_2<q_3<\cdots <q_n<q_{n+1}<\cdots<q_c.
\]
Here for any $n\ge 2$ the base $q_n\in[q_G, q_c)$ admits the quasi-greedy expansion  (cf.~\cite{DeVries_Komornik_Loreti_2016})
\[
\al(q_n)=\left\{
\begin{array}{lll}
(\la_1\ldots\la_{2^{n-2}}\overline{\la_1\ldots\la_{2^{n-2}}})^\f&\textrm{if}& M=2k;\\
(\la_1\ldots\la_{2^{n-1}}\overline{\la_1\ldots\la_{2^{n-1}}})^\f&\textrm{if}& M=2k+1.
\end{array}
\right.
\]

 Denote by $q_{NT}:=q_2$  the second  smallest base in $\vl$ larger than $q_G$. Then
  \begin{equation}\label{eq:31}
\al(q_{NT})=\left\{
\begin{array}
  {lll}
  ((k+1)(k-1))^\f&\textrm{if}& M=2k,\\
  ((k+1)(k+1)kk)^\f&\textrm{if}&M=2k+1.
\end{array}
\right.
\end{equation}
By (\ref{eq:21})--(\ref{eq:23}), (\ref{eq:31}) and Lemma \ref{lem:21} it follows that
\begin{equation}\label{eq:32}
q_G<q_{NT}<q_c<q_T.
\end{equation}
\begin{lemma}
\label{lem:31}
{  The sequence }$\al(q_G)$ is irreducible. Moreover, $(\VB_q, \si)$ is transitive for all $q\in[q_G, q_{NT})$.
\end{lemma}
\begin{proof}
Using (\ref{eq:21}) it is easy to check that $\al(q_G)$ is irreducible. By Lemma \ref{lem:25} (1) we have $\VB_q=\VB_{q_G}$ for all $q\in[q_G, q_{NT})$. So it suffices to prove that $(\VB_{q_G}, \si)$ is  transitive. This can be verified by observing $\VB_{q_G}=\set{k^\f}$ if $M=2k$, and $\VB_{q_G}=\set{((k+1)k)^\f, (k(k+1))^\f}$ if $M=2k+1$.
\end{proof}

In the following lemma we show that for all $q\in[q_{NT}, q_{T}]$ the quasi-greedy expansion $\al(q)$ is not   irreducible.
 \begin{lemma}
   \label{lem:32}
  { The sequence} $\al(q)$ is not irreducible for any $q\in[q_{NT}, q_T]$.
 \end{lemma}
 \begin{proof}
 Since the proof for $M=2k+1$ is similar, we only prove the lemma for $M=2k$.

Let $q\in[q_{NT}, q_T]$. Then by (\ref{eq:22}), (\ref{eq:31}) and Lemma \ref{lem:21} it follows that
 \begin{equation}\label{eq:33}
 ((k+1)(k-1))^\f\lle \al(q)\lle (k+1)\,k^\f.
 \end{equation}
 This implies that $\al_1(q)=k+1$. Observe that $(\al_1(q)^-)^\f=k^\f\in\vb$. But (\ref{eq:33}) implies
 \[ \al_1(q)(\overline{\al_1(q)}\,^+)^\f=(k+1)k^\f\lge \al(q).\]
Therefore, by Definition  \ref{def:26} $\al(q)$ is not irreducible.
 \end{proof}

Now we consider the transitivity of $(\VB_q, \si)$ for  $q\in[q_{NT}, q_T]$.
 \begin{lemma}\label{lem:33}
 { The subshift }$(\VB_q,\si)$ is not transitive for any $q\in[q_{NT}, q_T)$.
 \end{lemma}
 \begin{proof}
Suppose $M=2k+1$. By (\ref{eq:22}) we have $\al(q_T)=(k+1)((k+1)k)^\f$. Take  $q\in [q_{NT}, q_T)$. Then by (\ref{eq:31}) and Lemma \ref{lem:21} it follows that
\begin{equation}\label{eq:34}
((k+1)(k+1)kk)^\f\lle \al(q)\prec (k+1)((k+1)k)^\f.
\end{equation}
So
 there exists $m\ge 2$ such that $\al(q)\prec(k+1)((k+1)k)^m  0^\f$. This implies that the words $(k+1)((k+1)k)^n, k(k(k+1))^n \notin\L(\VB_q)$ for all $n\geq m$. Observe by (\ref{eq:34}) that $\om=(k+1)(k+1)$ and $\nu=((k+1)k)^{m}$ belong to $\L(\VB_q)$. We claim that there is no $\delta\in\L(\VB_q)$ such that $\om\de\nu\in\L(\VB_q)$.

 Suppose on the contrary that $\om\de\nu\in\L(\VB_q)$ for some $\de=\de_1\ldots \de_s\in\L(\VB_q)$. Note that $\de_i\in\set{k,k+1}$ for all $1\le i\le s$. Since $(k+1)((k+1)k)^m\notin\L(\VB_q)$, we have $\de_s=k$. Then $\de_s\nu=k((k+1)k)^m$.  Note that $k(k(k+1))^m\notin\L(\VB_q)$.This gives  $\de_{s-1}=k+1$. By iteration of this reasoning we conclude that
 \[
\de=\de_1\ldots\de_s=\left\{
\begin{array}{lll}
((k+1)k)^j&\textrm{if}& s=2j,\\
k((k+1)k)^j&\textrm{if}& s=2j+1.
\end{array}
\right.
 \]
However, in both cases the word $(k+1)((k+1)k)^m$ occurs in $\om\de\nu$. This implies that $\om\de\nu\notin\L(\VB_q)$, leading to a contradiction with our hypothesis.

Now we assume $M=2k$. Then by (\ref{eq:22}) we have $\al(q_{T})= (k+1)k^\f$. Let $q\in[q_{NT}, q_{T})$. Then by (\ref{eq:31}) and Lemma \ref{lem:21} we have
 \begin{equation}\label{eq:35}
((k+1) (k-1))^\f\lle\al(q)\prec(k+1)k^\f.
\end{equation}
By \eqref{eq:35} there exists $m\ge 1$ such that $\al(q)\prec(k+1)k^m 0^\f$. This implies that $(k+1)k^n, (k-1)k^n\notin\L(\VB_q)$ for all $n\geq m$. Note by (\ref{eq:35}) that $\om=k+1 \in \L(\VB_q)$ and $\nu=k^{m} \in \L(\VB_q)$. By a similar argument to that used in the case where $M=2k+1$, we can prove that  there does not exist $\de\in\L(\VB_q)$ such that $\om\de \nu\in\L(\VB_q)$.
 \end{proof}
 Note that $q_G$ is the smallest element of $\vl$, and $q_{NT}$ is the second smallest element of $\vl$. By Lemma \ref{lem:25} (1) it follows that
  $\VB_q=\VB_{q_G}$ for all $q\in[q_G, q_{NT})$. Note by Lemma \ref{lem:31} that $(\VB_{q_G}, \si)$ is transitive.
Then by Lemma \ref{lem:33} it follows that  $q_{NT}$ is the smallest base $q$ for which $(\VB_q, \si)$ is not transitive.

Now we recall from \cite[Definition 3.1.3]{Lind_Marcus_1995} the following.

\begin{definition}\label{def:34}
A subshift $(X, \sigma)$ is called \emph{sofic} if there exists a labeled graph $\G=(G, \mathfrak{L})$ such that $X$ can be represented by $\G$. A labeled graph $\G=(G, \mathfrak{L})$ is called \emph{right-resolving} if the outgoing edges from the same vertex  carry different labels.
 \end{definition}
 In the following lemma we prove that $(\VB_{q_T}, \si)$ is a transitive sofic subshift.

\begin{lemma}
  \label{lem:35}
{ The subshift $(\VB_{q_T}, \si)$ is sofic and transitive.} Moreover, the topological entropy of $(\VB_{q_T}, \sigma)$ is
  \[
  H(q_T)=\left\{
  \begin{array}
    {lll}
    \log 2&\textrm{if}& M=2k,\\
    \frac{1}{2}\log 2&\textrm{if}& M=2k+1.
  \end{array}\right.
  \]
\end{lemma}
 \begin{proof}

   Suppose $M=2k$. Then by (\ref{eq:23}) we have $\al(q_T)=(k+1)k^\f$. Thus
   \begin{equation}\label{eq:36}
   \VB_{q_T}=\set{(x_i): (k-1)k^\f\lle\si^n((x_i))\lle(k+1)k^\f\textrm{ for all }n\ge 0}.
   \end{equation}

  \begin{figure}[h!]
\centering
\begin{tikzpicture}[->,>=stealth',shorten >=1pt,auto,node distance=4cm,
                    semithick]

  \tikzstyle{every state}=[minimum size=0pt,fill=black,draw=none,text=black]

  \node[state] (A)                    { };
  \node[state]         (B) [ right of=A] { };

  \path[->,every loop/.style={min distance=0mm, looseness=60}]
   (A) edge [loop left,->]  node {$k$} (A)
            edge  [bend left]   node {$k-1$} (B)

        (B) edge [loop right] node {$k$} (B)
            edge  [bend left]            node {$k+1$} (A);
\end{tikzpicture}
 \caption{The picture of the labeled graph $\G_1=(G_1,\LB_1)$.} \label{Fig:2}
\end{figure}
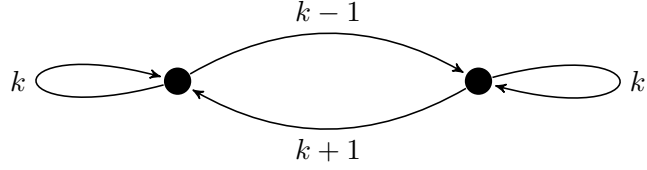
We claim that   $(\VB_{q_T}, \si)$ is a transitive sofic subshift which can be represented by the labeled graph $\G_1=(G_1, \LB_1)$ (see Figure \ref{Fig:2}).

Take a sequence $(x_i)\in\VB_{q_T}$. Then $x_j\in\set{k-1, k, k+1}$ for all $j\geq 1$. If $x_j=k-1$, then by (\ref{eq:36}) it follows that $x_{j+1}\in\set{k, k+1}$. Similarly, if $x_j=k$, then $x_{j+1}\in\set{k-1,k,k+1}$. Moreover, if $x_j=k+1$, then $x_{j+1}\in\set{k-1,k}$. This implies that $(x_i)\in X_{\G_1}$. Therefore $\VB_{q_T}\subseteq X_{\G_1}$.

On the other hand, suppose $(x_i)\notin\VB_{q_T}$. Then there exists $j\in\N$ such that
\[x_j x_{j+1}\ldots\succ (k+1)k^\f\quad\textrm{ or }\quad x_j x_{j+1}\prec (k-1)k^\f.\]
 By symmetry we may assume $x_j x_{j+1}\ldots\succ (k+1)k^\f$. Since $x_i\in\set{k-1, k, k+1}$ for all $i\ge 1$, there must exists a large integer $N\in\N$ such that $$x_j\ldots x_{j+N}=(k+1)k^{N-1}(k+1).$$ Clearly, $(k+1)k^{N-1}(k+1)\notin\L(X_{\G_1})$. This implies that $(x_i)\notin X_{\G_1}$.

Therefore $\VB_{q_T}=X_{\G_1}$ and $(\VB_{q_T}, \sigma)$ is a transitive sofic subshift.
   Note that the labeled graph $\G_1$ is right-resolving. By \cite[Theorem 4.3.3]{Lind_Marcus_1995} it follows that
   \[
   H(q_T)=h(\VB_{q_T})=\log\la_{G_1}=\log 2,
   \]
   where $\la_{G_1}$ is the spectral radius  of the adjacency matrix of $G_1$.

   Now we consider $M=2k+1$. By (\ref{eq:25}) we have $\al(q_T)=(k+1)(k+1)(k(k+1))^\f$. Then
   \[
   \VB_{q_T}=\set{(x_i): kk((k+1)k)^\f\lle \si^n((x_i))\lle(k+1)(k+1)(k(k+1))^\f\textrm{ for all }n\ge 0}.
   \]
By similar arguments used in the case where $M=2k,$ one can verify that  $(\VB_{q_T}, \si)$ is a transitive sofic subshift represented by the labeled graph $\G_2=(G_2, \LB_2)$ (see Figure \ref{Fig:3}).
   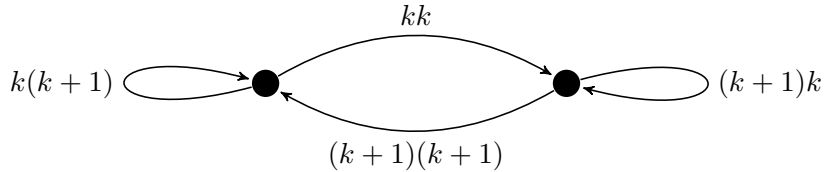
\begin{figure}[h!]
  \centering
\begin{tikzpicture}[->,>=stealth',shorten >=1pt,auto,node distance=4cm,
                    semithick]

  \tikzstyle{every state}=[minimum size=0pt,fill=black,draw=none,text=black]

  \node[state] (A)                    { };
  \node[state]         (B) [ right of=A] { };

  \path[->,every loop/.style={min distance=0mm, looseness=60}]
   (A) edge [loop left,->]  node {$k(k+1)$} (A)
            edge  [bend left]   node {$kk$} (B)

        (B) edge [loop right] node {$(k+1)k$} (B)
            edge  [bend left]            node {$(k+1)(k+1)$} (A);
\end{tikzpicture}
  \caption{The picture of the labeled graph $\G_2=(G_2,\LB_2)$.}
  \label{Fig:3}
\end{figure}
So $(\VB_{q_T}, \si)$ is  a transitive sofic subshift. Note that the labels in Figure \ref{Fig:3} are all of length $2$. This implies that
 $H(q_T)=h(\VB_{q_T})=\frac{1}{2}\log\la_{G_2}=\frac{1}{2}\log 2$.
 \end{proof}

 Combining Lemmas \ref{lem:31}--\ref{lem:35} we have proved Theorem \ref{th1} for $q\in[q_G, q_T]$.
 \begin{proposition}
   \label{prop:36}
   Let $q\in[q_G, q_T]\cap\vl$. Then
 $(\VB_{q},\si)$ is   transitive if and only if $\al(q)$ is irreducible or $q=q_T$.
 \end{proposition}

Note that $q_T>q_c>q_{NT}$, and $H(q)>0$ if and only if $q> q_c$.  Then by Lemmas \ref{lem:33} and  \ref{lem:35}   it follows that $q_T$ is the smallest base $q$ for which $H(q)>0$ and $(\VB_q, \si)$ is transitive.

\subsection{Transitivity of $(\VB_q, \si)$ for $q\in(q_T, M+1]$}

In this subsection we prove Theorem \ref{th1} for $q\in(q_{T}, M+1]\cap\vl$. First we prove the necessity.
 \begin{proposition}\label{prop:37}
 Let $q\in(q_{T}, M+1]\cap\vl$. If $(\VB_q, \si)$ is topologically transitive, then $\al(q)$ is irreducible.
 \end{proposition}
 \begin{proof}

 Take $q\in(q_T, M+1]\cap \vl$.
 Suppose that  $\al(q)=(\al_i)$ is not irreducible. We will show that $(\VB_q, \si)$ is not topologically transitive. By our assumption and Definition \ref{def:26} there must exist an integer $j\ge 1$ such that
 \begin{equation}\label{eq:37}
 (\al_1\ldots\al_j^-)^\f\in\vb\quad\textrm{and} \quad \al_1\ldots\al_j(\overline{\al_1\ldots\al_j}\,^+)^\f\succcurlyeq(\al_i).
 \end{equation}

 Note that $(\al_i)=\al(q)\in\vb$. So $\om:=\al_1\ldots\al_j\in\L(\VB_q)$. Let $(x_i)$ be such that $\om x_1x_2\ldots\in\VB_q$. By (\ref{eq:37}) it follows that
 \[
 \om x_1x_2\ldots=\al_1\ldots\al_j x_1 x_2\ldots\lle(\al_i)\lle\al_1\ldots\al_j(\overline{\al_1\ldots\al_j}\,^+)^\f.
 \]
 This implies that $x_1\cdots x_j\lle \overline{\al_1\ldots\al_j}\,^+=\overline{\om}\,^+$. However, since $\om x_1x_2\ldots\in\VB_q$ then we must have $(x_i)\lge \overline{(\al_i)}$, which yields
  $x_1\ldots x_j\lge\overline{\al_1\ldots \al_j}=\overline{\om}$. Therefore,
 \[
 x_1\ldots x_j =\overline{\om}\quad\textrm{or}\quad x_1\ldots x_j=\overline{\om}\,^+.
 \]

\begin{itemize}
  \item[(i)] If $x_1\ldots x_j=\overline{\om}$, then by (\ref{eq:37}) and $\om x_1x_2\ldots\in\VB_q$ it follows that
  \[
x_1x_2\ldots =\overline{\om}x_{j+1}x_{j+2}\ldots\lge\overline{\al_1\ldots\al_j}(\al_1\ldots\al_j^-)^\f=\overline{\omega}(\omega^-)^\f,
  \]
  which implies that $x_{j+1}\ldots x_{2j}\lge \omega^-$. Since $\om x_1x_2\ldots \in\VB_q$ we have  $x_{j+1}x_{j+2}\ldots\lle(\al_i)=\omega\al_{j+1}\al_{j+2}\ldots$. Then we conclude that
  \[
  x_{j+1}\ldots x_{2j}=\om\quad\textrm{or}\quad x_{j+1}\ldots x_{2j}=\om^-.
  \]
  \begin{itemize}
  \item[(ia)] If $x_{j+1}\ldots x_{2j}=\omega$, then by the same argument as above we have $x_{2j+1}\ldots x_{3j}=\overline{\omega}$ or $x_{2j+1}\ldots x_{3j}=\overline{\omega}^+$.

  \item[(ib)] If $x_{j+1}\ldots x_{2j}=\omega^-$, then $x_{1}\ldots x_{2j}=\overline{\omega}\omega^-$. By \eqref{eq:37} it follows that
  \[
x_1x_2\ldots=\overline{\omega}\omega^-x_{2j+1}x_{2j+2}\ldots\lge\overline{\al_1\ldots\al_j}(\al_1\ldots\al_j^-)^\f=\overline{\omega}(\omega^-)^\f.
  \]
  This implies that $x_{2j+1}\ldots x_{3j}\lge \omega^-$. Note that $x_{2j+1}x_{2j+2}\ldots\lle(\alpha_i)$. We obtain that $x_{2j+1}\ldots x_{3j}=\omega^-$ or $x_{2j+1}\ldots x_{3j}=\omega$.
  \end{itemize}

  \item[(ii)] If $x_1\ldots x_j=\overline{\om}\,^+$, then by (\ref{eq:37}) it follows that
  \[
  \om x_1x_2\ldots =\om\,\overline{\om}\,^+ x_{j+1}x_{j+2}\ldots \lle \al_1\ldots\al_j(\overline{\al_1\ldots\al_j}\,^+)^\f=\om (\overline{\om}\,^+)^\f,
  \]
  which implies that $x_{j+1}\ldots x_{2j}\lle \overline{\om}\,^+$. On the other hand, since $\om x_1x_2\ldots\in\VB_q$ we have $x_{j+1}\ldots x_{2j}\lge \overline{\al_1\ldots \al_j}=\overline{\om}$. Hence,
  \[
  x_{j+1}\ldots x_{2j}=\overline{\om}\quad\textrm{or}\quad  x_{j+1}\ldots x_{2j}=\overline{\om}\,^+.
  \]
\end{itemize}

 By iterating the above arguments it follows that $(x_i)\in X_{\G_3}$, where $(X_{\G_3}, \sigma)$ is a sofic subshift represented by the labeled graph $\G_3=(G_3, \LB_3)$ (see Figure \ref{Fig:4}).
\begin{figure}[h!]
\centering
\begin{tikzpicture}[->,>=stealth',shorten >=1pt,auto,node distance=4cm,
                    semithick]

  \tikzstyle{every state}=[minimum size=0pt,fill=black,draw=none,text=black]

  \node[state] (A)                    { };
  \node[state]         (B) [ right of=A] { };

  \path[->,every loop/.style={min distance=0mm, looseness=60}]
   (A) edge [loop left,->]  node {$\overline{\omega}^+$} (A)
            edge  [bend left]   node {$\overline{\omega}$} (B)

        (B) edge [loop right] node {$\omega^-$} (B)
            edge  [bend left]            node {$\omega$} (A);
\end{tikzpicture}
  \caption{The picture of the  labeled graph $\G_3=(G_3, \LB_3)$.}\label{Fig:4}
\end{figure}
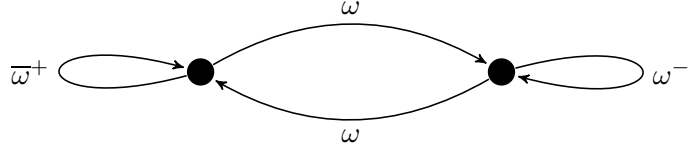

Note   that the word $\omega=\al_1\ldots \al_j$ belongs to $\L(\VB_q)$. Then to prove that $(\VB_q,\si)$ is not transitive it suffices to show that there exists a word $\nu\in\L(\VB_q)$ such that   $\omega$ can not be connected with $\nu$ in $\L(\VB_q)$.  From the above arguments it follows that if a sequence in $\VB_q$ begins with $\om$, then the sequence must be in $X_{\G_3}$ (see also Figure \ref{Fig:4}). So,  to prove that $(\VB_q,\si)$ is not transitive  it suffices to show that
 there exists a word $\nu\in\L(\VB_q)$    not appearing  in $\L(X_{\G_3})$. We split the proof of this fact into the following two cases.

Case (I).  $M=2k$. Since $q>q_{T}$,  by (\ref{eq:22}) and Lemma \ref{lem:21} we have
\begin{equation}
\label{eq:38}
\al(q) \succ\al(q_{T})=(k+1)k^\f.
 \end{equation}
 So the word $k^s\in\L(\VB_q)$ for all $s\in\N$. If $\al_1=k+1$, then by (\ref{eq:37}) and (\ref{eq:38}) we have $j\ge 2$. As a consequence of this, we see by examining Figure \ref{Fig:4} that the word $\nu=k^{2j}\in\L(\VB_q)$ cannot appear in  $\L(X_{\G_3})$. If $\al_1>k+1$, then we can similarly prove that $\nu=k^{2j}\in\L(\VB_q)$ is not contained in $\L(X_{\G_3})$.

Case (II). $M=2k+1$.  Note that $q>q_{T}$. Then
\begin{equation}\label{eq:39}
(\al_i) \succ\al(q_{T})=(k+1)((k+1)k)^\f.
 \end{equation}
 This implies that the word $((k+1)k)^s\in\L(\VB_q)$ for all $s\in\N$. If $\al_1=k+1$, then by (\ref{eq:37}) and (\ref{eq:39}) we have $j\ge 3$. By examining Figure \ref{Fig:4} we see that the word $\nu=((k+1)k)^{2j}\in\L(\VB_q)$ does not occur in $\L(X_{\G_3})$. If $\al_1>k+1$, then we could again deduce that $\nu=((k+1)k)^{2j}\in\L(\VB_q)$ is not contained in $\L(X_{\G_3})$.
 \end{proof}

 Now we turn to prove the sufficiency of Theorem \ref{th1} for $q\in(q_T, M+1]$. Take $q\in(q_T, M+1]$ such that $\al(q)$ is irreducible.  We want to show that for any two admissible words $\om, \nu\in\L(\VB_q)$ there exists a word $\de$ such that $\om\de\nu\in\L(\VB_q)$.

 In the following lemma we investigate possible suffixes of a given word in $\L(\VB_q)$.
\begin{lemma}
  \label{lem:38}
  Let $q\in(q_c, M+1]\cap \vl$. Then for  any word $\om\in\L(\VB_q)$ and any integer $m>
  |\om|$ there exists $\eta\in \L(\VB_q)$ such that $\al_1(q)\ldots\al_m(q)$ or $\overline{\al_1(q)\ldots\al_m(q)}$ is a suffix of $\omega\eta\in B_*(\VB_q)$.
\end{lemma}
\begin{proof}
Take $\om=\om_1\ldots\om_n\in\L(\VB_q)$ and fix an integer $m> n$. To prove our result it suffices to find $(\eta_i)\in \VB_q$ such that $\om \eta_1 \eta_2\ldots \in \VB_q$ and $\om \eta_1\eta_2\ldots \eta_l$ satisfies the desired properties for some $l\geq 1$. Since $\om\in\L(\VB_q)$ we have
\begin{equation}
  \label{eq:310}
  \overline{\al_1(q)\ldots\al_{n-i}(q)}\lle \om_{i+1}\ldots\om_n\lle \al_1(q)\ldots\al_{n-i}(q)\quad\textrm{for all}\quad i=0,1,\ldots, n-1.
\end{equation}
If strict inequalities hold in (\ref{eq:310})  for all $i\in\set{0,1,\ldots,n-1}$, then the lemma follows by taking $(\eta_i)=(\al_i(q))$ or $(\overline{\al_i(q)})$ and $l=m$.

Otherwise,  let $s\in\set{0,1,\ldots,n-1}$ be the smallest integer for which
\[
\om_{s+1}\ldots\om_n=\al_1(q)\ldots\al_{n-s}(q)\quad\textrm{or}\quad \om_{s+1}\ldots\om_n=\overline{\al_1(q)\ldots\al_{n-s}(q)}.
\]
By symmetry we may assume $\om_{s+1}\ldots\om_n=\al_1(q)\ldots\al_{n-s}(q)$. If $s=0$, then $\om=\al_1(q)\ldots\al_{n}(q)$, and therefore the lemma follows by taking $(\eta_i)=(\al_{n+i})$ and $l=m-n$. If $s\in\set{1,\ldots,n-1}$, then by the minimality  of $s$ and (\ref{eq:310}) it follows that
  \[
      \overline{\al_1(q)\ldots\al_{n-i}(q)}\prec\om_{i+1}\ldots\om_n\prec\al_1(q)\ldots\al_{n-i}(q)\quad\textrm{for all}\quad 0\le i<s.
  \]
Therefore, the lemma follows by taking $(\eta_i)=(\al_{n-s+i}(q))$ and $l=m-n+s$.
\end{proof}

We also investigate  possible prefixes of  a given word  in $\L(\VB_q)$.
\begin{lemma}
  \label{lem:39}
  Let $q\in(q_T, M+1]$ and $m\in\N$.
Then for   any $\nu\in\L(\VB_q)$  there exists $\ga\in\L(\VB_q)$  such that $k^m$ is a prefix of $\ga\nu\in\L(\VB_q)$ if $M=2k$, or $((k+1)k)^m$ is a prefix of $\ga\nu \in\L(\VB_q)$ if $M=2k+1$.
\end{lemma}
\begin{proof}
Take $m\in\N$. Suppose that $M=2k$. Note that $q_T< q\le M+1$. Then by (\ref{eq:22}) and  Lemma \ref{lem:21} we have $\al(q)\succ\al(q_T)= (k+1)k^\f$. This implies that
\[
\overline{\al_1(q)}<k<\al_1(q).
\]
Thus, the lemma follows by taking $\ga=k^m$.

Now we assume $M=2k+1$. Since $q\in(q_T, M+1]$, we have by Lemma \ref{lem:21} that $\al(q)\succ\al(q_T)= (k+1)((k+1)k)^\f$. This implies that
\[
\overline{\al_1(q)\al_2(q)}\prec(k+1)k\prec\al_1(q)\al_2(q).
\]
Let $\nu=\nu_1\ldots\nu_n\in\L(\VB_q)$.  Then the lemma follows by taking $\ga=((k+1)k)^m(k+1)$ if $\nu_1\le k$, or  by taking $\ga=((k+1)k)^m$ if $\nu_1\ge k+1$.
\end{proof}

By Lemmas \ref{lem:38} and \ref{lem:39}, to prove the transitivity of $(\VB_q, \si)$ it suffices to connect words in $\L(\VB_q)$ that end with
 \[\al_1(q)\ldots\al_m(q)\quad\textrm{or}\quad \overline{\al_1(q)\cdots \al_m(q)},\quad m\in\N,\]
  with words in  $\L(\VB_q)$ that begin with $k^m$ if $M=2k$, or begin with $((k+1)k)^m$ if $M=2k+1$. We will show this by using some  techniques from combinatorics on words (see e.g.,~\cite{Allouche_Shallit_2003}).

\begin{definition}\label{def:310}
A word $\a=a_1\ldots a_m$ is called \emph{primitive} if
\begin{equation*}
\overline{a_1\ldots a_{m-i}}\prec a_{i+1}\ldots a_m \lle a_1\ldots a_{m-i}\quad\textrm{for all }0\le i<m.
\end{equation*}
\end{definition}
 { We point out that the name `primitive' of a word will make sense, since later in Proposition \ref{prop:317} we will use primitive words to construct a bridge between two admissible words in $(\vb_q,\sigma)$ with $\al(q)$ irreducible.}
In the following definition we introduce a subword based on a primitive word. The subwords will be useful in our subsequent proofs.
\begin{definition}\label{def:311}
For a primitive word $\a=a_1\ldots a_m,$ we call the truncated word
$$ \R(\a):=a_1\ldots a_s$$ the \emph{reflection recurrence word} of $\a$, where
 $s\in\set{0,1,\ldots,m-1}$ is the smallest integer satisfying
 \begin{equation}\label{eq:311}
 a_{s+1}\ldots a_m^-=\overline{a_1\ldots a_{m-s}}.
 \end{equation}
 \end{definition}
 We point out that when $s=0$ the reflection recurrence word $\R(\a)=\epsilon$ is the empty word. If  (\ref{eq:311}) does not hold for any $s\in\{0,1,\ldots,m-1\}$, then we set $\R(\a)=\a$, and in this case we have
 \[\overline{a_1\ldots a_{m-i}}\prec a_{i+1}\ldots a_m^-\prec a_1\ldots a_{m-i}\quad \textrm{for all}\quad 0\le i<m.\]

Before beginning our investigation into reflection recurrence words let us give an example.
\begin{example}\label{ex:312}
Let $M=1$. We consider the word
\[\a=a_1\ldots a_{12}=111001\;000111.\]
By Definition \ref{def:310} it follows that $\a$ is primitive. Note that
\[
a_{7}\ldots a_{12}^-=000110=\overline{111001}=\overline{a_1\ldots a_6},
\]
and   $s=6$ is the smallest integer for which $a_{s+1}\ldots a_{12}^-=\overline{a_1\ldots a_{12-s}}$. By Definition \ref{def:311} this implies that
\[
\R(\a)=a_1\ldots a_6=111001.
\]
In terms of Definition \ref{def:310} the reflection recurrence word $\R(\a)$ is also a primitive word.

By a similar argument to that given above we can prove that the reflection recurrence word of $\R(\a)$ is given by
\[
\R^2(\a):=\R(\R(\a))=\R(111001)=111.
\]
Again, $\R^2(\a)$ is primitive. Repeating this reasoning we have
\[
\R^3(\a):=\R(\R^2(\a))=\R(111)=11,
\]
which is also  a primitive word.  Furthermore, $\R^4(\a)=\R(\R^3(\a))=1$, again a primitive word. However, $\R^5(\a)=\R(\R^4(\a))=\epsilon$ is the empty word.
\end{example}

The previous example motivates us to investigate  the following properties of reflection recurrence words and  primitive words.
First we give bounds  for the length of a reflection recurrence word in term of the length of a primitive word $\a$.
\begin{lemma}\label{lem:313}
Let $\a$ be a primitive word with  $|\a|\ge 2$. Then
$\frac{|\a|}{2}\le |\R(\a)|\le|\a|$.
\end{lemma}
\begin{proof}
Let $\a=a_1\ldots a_m$ be a primitive word with $m\ge 2$. Clearly, $|\R(\a)|\le m$. So it remains to prove $|\R(\a)|\ge m/2$. Suppose on the contrary that $\R(\a)=a_1\ldots a_s$ with $0\le s<m/2$. By Definition \ref{def:310} it follows that $a_1\ge \overline{a_m}>\overline{a_1}$. Since $a_1>\overline{a_{1}}$ it follows from Definition \ref{def:311}  of a reflection recurrence word that we cannot have $s=0$. It remains to show that we cannot have $0< s<m/2.$

Since $\R(\a)=a_1\ldots a_s$, we have  $a_{s+1}\ldots a_m^-=\overline{a_1\ldots a_{m-s}}$. Then we can rewrite $\a$ as
\[
\a=a_1\ldots a_s\, a_{s+1}\ldots a_m=a_1\ldots a_s\overline{a_1\ldots a_{m-s}}\,^+.
\]
This implies that
$
a_{2s+1}\ldots a_m =\overline{a_{s+1}\ldots a_{m-s}}\,^+=a_1\ldots a_{m-2s}^+,
$
leading to a contradiction with the primitivity of $\a$.
\end{proof}

Example \ref{ex:312} suggests that the reflection recurrence word $\R(\a)$ of a primitive word $\a$ is also primitive. The following lemma shows that this is the case.

\begin{lemma}\label{lem:314}
Let $\a=a_1\ldots a_m$ be a primitive word with $m\ge 2$. Then $\R(\a)$ is also primitive.
\end{lemma}
\begin{proof}
Let $\R(\a)=a_1\ldots a_s$.  If $s=m$, then $\R(\a)=\a$ is primitive. In the following we assume $s<m$. Then by Lemma \ref{lem:313} we have $s\ge m/2\ge 1$.  By the primitivity of $\a$ it follows that
\[
\overline{a_1\ldots a_{s-i}}\lle a_{i+1}\ldots a_s\lle a_1\ldots a_{s-i}\quad\textrm{for all }0\le i<s.
\]
So to prove the primitivity of $\R(\a)$ it suffices to show that $a_{i+1}\ldots a_s\ne \overline{a_1\ldots a_{s-i}}$ for all $0\le i<s$.

Suppose on the contrary that $a_{i+1}\ldots a_s=\overline{a_1\ldots a_{s-i}}$ for some $0\le i<s$. By the primitivity of $\a$ it follows that
\begin{equation}\label{eq:312}
\overline{a_{s-i+1}\ldots a_{m-i}}\lle a_{s+1}\ldots a_m^-.
\end{equation}
Since $\R(\a)=a_1\ldots a_s$ we must have $a_{s+1}\ldots a_m^-=\overline{a_1\ldots a_{m-s}}$. Again by the primitivity of $\a$ we have $a_{s-i+1}\ldots a_{m-i}\lle a_1\ldots a_{m-s}$. Taking reflection this implies that
\begin{equation}\label{eq:313}
a_{s+1}\ldots a_m^-=\overline{a_1\ldots a_{m-s}}\lle\overline{a_{s-i+1}\ldots a_{m-i}}.
\end{equation}By (\ref{eq:312}), (\ref{eq:313}) and our assumption $a_{i+1}\ldots a_s=\overline{a_1\ldots a_{s-i}}$ we may conclude that
\[a_{i+1}\ldots a_m^-=\overline{a_1\ldots a_{m-i}}.\]
Since $i<s$ this contradicts the fact that $s$ is the smallest integer satisfying \[a_{s+1}\ldots a_m^-=\overline{a_1\ldots a_{m-s}}.\]
\end{proof}

Given a primitive word $\a$ and $n\in\N$, we denote by $\R^n(\a)=\R(\R^{n-1}(\a))$ the reflection recurrence word of $\R^{n-1}(\a)$, where we set $\R^0(\a)=\a$. Example \ref{ex:312} also implies the following lemma.
\begin{lemma}\label{lem:315}
Let $\a=a_1\ldots a_m$ be a primitive word with $m\ge 2$. Then there exists $j\in\set{0,1,\ldots, m}$ such that either $\R^{j+1}(\a)=\R^j(\a)$ with $|\R^j(\a)|\ge 2$, or $\R^j(\a)=a_1\overline{a_1}\,^+$.
\end{lemma}
\begin{proof}
By Lemma \ref{lem:313} we have $m/2\le |\R(\a)|\le m$. This implies that the length $|\R^j(\a)|$ is decreasing with respect to $j$. If there exists $j\ge 0$ such that $|\R^{j+1}(\a)|=|\R^{j}(\a)|\ge 2$, then $\R^{j+1}(\a)=\R^j(\a)$ and we are done. Otherwise,  there exists $0\le j\le m$ such that $|\R^j(\a)|=2$ and $|\R^{j+1}(\a)|=1$. It is a consequence of Definition \ref{def:311} of a reflection recurrence word that we then must have $\R^j(\a)=a_1a_2=a_1\overline{a_1}\,^+$.
\end{proof}

For a word $\om$ and a sequence $(c_i)$ we write $\om\prec (c_i)$ if $\om  \prec c_1\ldots c_{|\om|}$. Symmetrically, we write $\om\succ (c_i)$ if $\om \succ c_1\ldots c_{|\om|}$.

The following lemma is proved using Lemma \ref{lem:314}.

\begin{lemma}\label{lem:316}
Let $q\in(q_T, M+1]\cap \vl$ be such that $\al(q)$ is irreducible.  Then $q\in\overline{\ul}$ and there exists infinitely many integers $m\ge 2$ such that $\a=\al_1(q)\ldots\al_m(q)$ is primitive.

Furthermore, for each of these $m$ there exists a large integer $N=N(m)$ such that for all $u\in\N$ and any $j\ge 0$  we have
\begin{align}
\overline{\al_1(q)\ldots\al_N(q)}&\prec\si^j(\a(\overline{\a}\,^+)^\f)\prec\al_1(q)\ldots\al_N(q),\label{eq:314}\\
  \overline{\al_1(q)\ldots\al_N(q)}&\prec\si^j((\a^-)^u(\R(\a)^-)^\f)\prec\al_1(q)\ldots\al_N(q),\label{eq:315}
\end{align}
and symmetrically,
\begin{align*}
 \overline{\al_1(q)\ldots\al_N(q)}&\prec\si^j(\overline{\a}\,( {\a}^-)^\f)\prec\al_1(q)\ldots\al_N(q),\\
  \overline{\al_1(q)\ldots\al_N(q)}&\prec\si^j((\overline{\a}\,^+)^u(\overline{\R(\a)}\,^+)^\f)\prec\al_1(q)\ldots\al_N(q).
\end{align*}
\end{lemma}
\begin{proof}
Let $q\in(q_T, M+1]\cap \vl$ be such that $\al(q)$ is irreducible.
For simplicity we  write $(\al_i)=(\al_i(q))$. Let us emphasise that since $q\in(q_T,M+1]$ we must have $\al_{1}>\overline{\al_1}$. We claim that $q\in(q_T, M+1]\cap\overline{\ul}$. For if $q\in\vl\setminus\overline{\ul}$, then by  \cite[Theorem 1.3]{DeVries_Komornik_Loreti_2016} it would follow that $\al(q)$ has the form $\al(q)=(\al_1\ldots\al_n\overline{\al_1\ldots\al_n})^\f$ with $(\al_1\ldots\al_n^-)^\f\in\vb$. Which contradicts the fact that $\al(q)$ is irreducible.
 Therefore  $q\in(q_c, M+1]\cap\overline{\ul}$. By \cite[Lemma 4.1]{Komornik_Loreti_2007} it follows that there exist  infinitely many  integers $m$ such that $\a=\al_1\ldots\al_m$ is primitive.

Now we turn our attention to proving the inequalities.
By symmetry we only prove (\ref{eq:314}) and (\ref{eq:315}).  Fix a large integer $m\ge 2$ for which $\a=\al_1\ldots \al_m$ is primitive. Then by Definition \ref{def:310} it follows that
\begin{equation}
  \label{eq:317}
  \al_{i+1}\ldots\al_m^-\al_1\ldots\al_i\succ \overline{\al_1\ldots\al_m}\quad\textrm{and}\quad \al_{i+1}\ldots\al_m^-\prec\al_1\ldots\al_{m-i}
\end{equation}
for all $0\le i<m$. This implies that $(\a^-)^\f\in\vb$.
By Lemma \ref{lem:314}  $\R(\a)$ is also primitive. Then $(\R(\a)^-)^\f\in\vb$. Since $(\al_i)$ is  irreducible,  we have $\a(\overline{\a}\,^+)^\f\prec(\al_i)$ and $\R(\a)(\overline{\R(\a)}\,^+)^\f\prec(\al_i)$. Let $N=N(m)\ge m$ be a large integer satisfying
 \begin{equation}
   \label{eq:316}
   \a(\overline{\a}\,^+)^\f\prec\al_1\ldots\al_N\quad\textrm{and}\quad \R(\a)(\overline{\R(\a)}\,^+)^\f\prec\al_1\ldots\al_N.
 \end{equation}

First we prove (\ref{eq:314}). Since $\al_1>\overline{\al_1}$ we satisfy the left hand side of \eqref{eq:314} when $j=0$. Moreover \eqref{eq:316} proves the right hand side of (\ref{eq:314}) when $j=0$. Therefore \eqref{eq:314} holds for $j=0$. For $0<j<m$ it follows from the primitivity of $\a$ that
\[\al_{j+1}\ldots\al_m\overline{\al_1\ldots\al_j}\prec\al_1\ldots\al_m\quad\textrm{and}\quad \al_{j+1}\ldots\al_m\succ\overline{\al_1\ldots\al_{m-j}}.\] Therefore \eqref{eq:314} holds for $0<j<m$.
To prove \eqref{eq:314} holds for $j\ge m$ we start by taking reflections in \eqref{eq:317} that
\begin{equation}
  \label{eq:318}
  \overline{\al_{i+1}\ldots\al_m}\,^+\overline{\al_1\ldots\al_i}\prec \al_1\ldots\al_m\quad\textrm{and}\quad \overline{\al_{i+1}\ldots\al_m}\,^+\succ \overline{\al_1\ldots\al_{m-i}}
\end{equation}for all $0\le i<m$. \eqref{eq:318} implies that \eqref{eq:314} holds for $j\geq m$.

We now prove (\ref{eq:315}). Let $(x_i):=(\a^-)^u(\R(\a)^-)^\f$ with
  $u\in\N$. Suppose $\R(\a)=\al_1\ldots \al_s$.  Then  $(x_i)$ can be written as
 \begin{align*}
 (x_i)&=(\al_1\ldots\al_m^-)^{u-1}\al_1\ldots\al_s\al_{s+1}\ldots\al_m^-(\al_1\ldots\al_s^-)^\f\\
 &=(\al_1\ldots\al_m^-)^{u-1}\al_1\ldots\al_s\overline{\al_1\ldots\al_{m-s}}(\al_1\ldots\al_s^-)^\f.
\end{align*}
If $s=m$, then $(x_i)=(\al_1\ldots\al_m^-)^\f$ and (\ref{eq:315}) follows by (\ref{eq:317}).

Now we assume $s<m$. By (\ref{eq:317}) it follows that (\ref{eq:315}) holds for $0\le j<(u-1)m$. Note that $(x_i)$ is eventually periodic. So to prove (\ref{eq:315}) it suffices to consider $(u-1)m\le j<um+s$. We split the proof into the following four cases.
\begin{itemize}
\item $(u-1)m\le j< (u-1)m+s$. Then $i=j-(u-1)m\in\set{0,1,\ldots,s-1}$. By the primitivity of $\a$ and $\R(\a)$ it follows that
\[
\al_{i+1}\ldots\al_s\succ\overline{\al_1\ldots\al_{s-i}}\quad\textrm{and}\quad\al_{i+1}\ldots\al_m^-\prec\al_1\ldots\al_{m-i}.
\]
Therefore (\ref{eq:315}) holds for $(u-1)m\le j< (u-1)m+s$.

\item $j=(u-1)m+s$. Then $i=j-(u-1)m=s$. By Lemma \ref{lem:313} we know that $s\ge m/2$.

 If $s=m/2$, then
$x_{j+1}x_{j+2}\ldots= \overline{\al_1\ldots\al_s}(\al_1\ldots\al_s^-)^\f$. Therefore,  by (\ref{eq:316}) and using that $ \overline{\al_1}<\al_1$ we have have (\ref{eq:315}).

If $s>m/2$ then $0<2s-m<s$. Again by the primitivity of $\a$ and $\R(\a)$ we obtain
\[
\al_{s+1}\ldots\al_m^-=\overline{\al_1\ldots\al_{m-s}},\quad \al_1\ldots\al_{2s-m}\succ\overline{\al_{m-s+1}\ldots\al_s}
\]
and $\al_{s+1}\ldots\al_m^-\prec\al_1\ldots\al_{m-s}$. These inequalities combine to prove \eqref{eq:315} when $s>m/2$. Combining the above we have proved (\ref{eq:315}) for $j=(u-1)m+s$.

\item $(u-1)m+s<j<u m$. Then  $i=j-(u-1)m\in\{s,\ldots, m\}$. Note that $s\ge m/2$. Therefore $0<s-m+i<s$. By the primitivity of $\a$ and $\R(\a)$ it follows that
\[
\al_{i+1}\ldots\al_m^-\lge\overline{\al_1\ldots\al_{m-i}}, \quad\al_1\ldots\al_{s-m+i}\succ\overline{\al_{m-i+1}\ldots\al_s}
\]
and $\al_{i+1}\ldots\al_m^-\prec\al_1\ldots\al_{m-i}$. These inequalities prove (\ref{eq:315}) for $(u-1)m+s<j<u m$.

\item $um\le j<um+s$. In this case one can verify, using the primitivity of $\R(\a),$ that
\[\al_{i+1}\ldots \al_s^-\al_1\ldots\al_i\succ\overline{\al_1\ldots \al_s}\quad\textrm{and}\quad\al_{i+1}\ldots\al_s^-\prec\al_1\ldots\al_{s-i}\]
for all $0\le i<s$. Therefore $(\ref{eq:315})$ holds for $um\le j<um+s$.
\end{itemize}
\end{proof}

Now we prove the sufficiency of Theorem \ref{th1} for $q\in(q_T, M+1]$.
\begin{proposition}
  \label{prop:317}
  Let $q\in(q_{T}, M+1]\cap \vl$. If $\al(q)$ is irreducible then $(\VB_q, \si)$ is transitive.
\end{proposition}
\begin{proof}
Since the proof for $M=2k+1$ is similar, we only prove this proposition for $M=2k$.

Let $q\in(q_T, M+1]\cap\vl$ be such that $\al(q)=(\al_i)$ is irreducible. Fix $\om, \nu\in\L(\VB_q)$ and take $(x_i) \in \VB_q$ such that $\nu$ is a prefix of $(x_i)$. By Lemma \ref{lem:316} there exists a large integer $m>\max\set{|\om|, 2}$ such that $\a=\al_1 \ldots\al_m$ is primitive.
  Since $(\al_i)=\al(q)$ is irreducible, there exists a large integer $N\ge m$ such that
  \begin{equation}\label{eq:319}
  \al_1\ldots\al_j(\overline{\al_1\ldots\al_j}\,^+)^\f\prec\al_1\ldots\al_N\quad\textrm{whenever}\quad(\al_1\ldots\al_j^-)^\f\in\vb \textrm{ and }j\le m.
  \end{equation}

By Lemma \ref{lem:38}  there exists $\eta\in\L(\VB_q)$ such that $\a$ or $\overline{\a}$ is a suffix of $\om\eta\in\L(\VB_q)$. By symmetry we may assume that $\a$ is a suffix of $\om\eta$.  Moreover, by Lemma \ref{lem:39} there exists $\ga\in\L(\VB_q)$ such that $k^N$ is a prefix of $\ga\nu\in\L(\VB_q)$.

Note that from (\ref{eq:319}) it follows that
$$  \a(\overline{\a}^+\,)^{\f} \prec \al_1\ldots\al_N,\quad \R(\a)(\overline{\R(\a)}^+)^\f\prec\al_1\ldots\al_N.$$
Then from (\ref{eq:314}) of Lemma \ref{lem:316} we obtain
$$\overline{\al_1\ldots\al_N} \prec \sigma^j(\a(\overline{\a}^+\,)^{\f}) \prec \al_1\ldots\al_N\quad\textrm{for all}\quad j\ge 0.$$
Moreover from the symmetric version of (\ref{eq:315}) in  Lemma \ref{lem:316}  it follows that
$$\overline{\al_1\ldots\al_N} \prec \sigma^j((\overline{\a}\,^+)^u(\overline{\R(\a)}\,^+)^{\f}) \prec \al_1\ldots\al_N$$
 for every $u \in \N$ and every $j \ge 0$.  This  implies
  \begin{equation}\label{eq:350}
  \overline{\al_1\ldots\al_N}\prec\si^i(\a\,(\overline{\a}\,^+)^N(\overline{\R(\a)}\,^+)^\f)\prec\al_1\ldots\al_N\quad\textrm{for all }i\ge 0.
  \end{equation}

\vspace{1em}Firstly let us assume $\R(\a)=\a$. Then
 \begin{equation}\label{eq:351}
 \overline{\al_1\ldots\al_{m-i}}\prec\al_{i+1}\ldots\al_m^-\prec\al_1\ldots\al_{m-i}\quad\textrm{ for all }i=0,1,\ldots, m-1.
 \end{equation}
We claim that if $\R(\a)=\a$ then $\om\eta(\overline{\a}\,^+)^N\ga\nu\in\L(\VB_q)$. Note that $\ga\nu\in\L(\VB_q)$. Then, since $\nu$ is the prefix of $(x_i)$ and $(x_i) \in \VB_q$ it suffices to show that
\begin{equation}
  \label{eq:352}
  \overline{\al_1\ldots \al_N}\prec \si^i(\om\eta(\overline{\a}\,^+)^N\ga \nu)\prec\al_1\ldots\al_N\quad\textrm{for all }0\le i<|\om|+|\eta|+N m.
\end{equation}

\vspace{0.7em}Consider $0\le i<|\om|+|\eta|-m$ firstly. Note that $\a$ is  a suffix of $\om\eta$. Then, from Lemma \ref{lem:38} we know that $\om\eta \in \L(\VB_q)$. This implies that $$\overline{\al_1\ldots \al_N}\prec \si^i(\om\eta(\overline{\a}\,^+)^N\ga \nu)\prec\al_1\ldots\al_N$$
for all $0\le i<|\omega|+|\eta|-m$.

\vspace{0.7em}Now, consider $|\om|+|\eta|-m\le i< |\om|+|\eta|+Nm$. Then from  (\ref{eq:351}) we obtain
$$\overline{\al_1\ldots\al_{m-i}}\prec\overline{\al_{i+1}\ldots\al_m}\,^+\prec\al_1\ldots\al_{m-i}\quad\textrm{ for all }i=0,1,\ldots, m-1.$$
 Then
 $$\overline{\al_1\ldots \al_N}\prec \si^i(\om\eta(\overline{\a}\,^+)^N\ga \nu)\prec\al_1\ldots\al_N.$$
  This implies that (\ref{eq:352}) holds.
Then we can conclude that $\om\eta(\overline{\a}\,^+)^N\ga\nu\in\L(\VB_q)$.

\vspace{1em}Now we assume $\R(\a)\ne\a$. Note that $|\a| \ge 2$. Then by Lemma \ref{lem:313} we have $|\R(\a)|\ge 1$.

\vspace{0.7em}If $|\R(\a)|=1$, then $\a=\al_1\overline{\al_1}\,^+$. Note that $\al(q)\succ\al(q_T)=(k+1)k^\f$. Then $\R(\a)=\al_1\ge k+1$. It is easy to show that $(\al_1\overline{\al_1}\,^+)^jk^N \in \L(\VB_q)$ for every $j \in \N$. Hence, by (\ref{eq:350}) and the same arguments as above it follows that
\[\om\eta(\overline{\a}\,^+)^N(\overline{\R(\a)}\,^+)^N\ga\nu\in\L(\VB_q).\]

\vspace{0.7em} Suppose that $|\R(\a)|\ge 2$. From (\ref{eq:319}) we obtain that
 $$\R(\a)(\overline{\R(\a)}^+)^\f\prec\al_1\ldots\al_N, \quad \R^2(\a)(\overline{\R^2(\a)}^+)^\f\prec\al_1\ldots\al_N.$$
  Recall that the word $\R(\a)$ is primitive by Lemma \ref{lem:314}.  Applying  the symmetric version of Lemma \ref{lem:316} (\ref{eq:315}) to $\R(\a)$  it follows that
   $$\overline{\al_1\ldots \al_N} \prec \sigma^j(\overline{\R(\a)}\,^+)^N(\overline{\R^2(\a)}\,^+)^\f \prec \al_1\ldots \al_N \textrm{ for every } j \in \N.$$
By (\ref{eq:350}) this implies that
        \[
 \overline{\al_1\ldots\al_N}\prec\si^j(\a\,(\overline{\a}\,^+)^N(\overline{\R(\a)}\,^+)^N(\overline{\R^2(\a)}\,^+)^\f)\prec\al_1\ldots\al_N\quad\textrm{for all }j\ge 0.
        \]
        By Lemma \ref{lem:315} there exists $\ell\in\set{1,2,\ldots,m}$ such that  $\R^{\ell+1}(\a)=\R^{\ell}(\a)$ or $\R^\ell(\a)=\al_1\overline{\al_1}\,^+$. Therefore, by repeating the arguments explained for the previous cases we conclude that
        \[\om\eta(\overline{\a}\,^+)^N(\overline{\R(\a)}\,^+)^N\ldots(\overline{\R^{\ell+1}(\a)}\,^+)^N\ga\nu\in\L(\VB_q).\]
This completes the proof.
\end{proof}

\begin{proof}
  [{\bf Proof of Theorem \ref{th1}}]
  The theorem follows by Propositions \ref{prop:36}, \ref{prop:37} and \ref{prop:317}.
\end{proof}

\noindent\textbf{Final remarks on topological transitivity. }

First we include some explicit examples of irreducible and non-irreducible sequences in $\vb$ (see Table \ref{tab:1}). Here the parameter $s$ is always a member of the set $\set{k+1, \ldots , M}$, and $t$ is a member of $\set{\overline{s}+1, \ldots, s-1}$.

\begin{table}[h!]
\centering  
\small
\begin{tabular}{c c c c}
\hline                     
$\al(q)$ & conditions & irreducible\\ [0.5ex] 
\hline \hline                 
$(s^j\overline{s})^\f$ &  $j \geq 2$ & yes \\
\hline
$(s^jt^l)^{\f}$ & \shortstack{ $j \ge 1, 0\le l \le j$}\ & yes \\
\hline
$(s^j\overline{s}^j)^\infty$&  $j \ge 2$  & no \\
\hline
\end{tabular}
\vspace{0.7em} \caption{Examples of irreducible sequences and non-irreducible sequences.}\label{tab:1}
\end{table}

In terms of the definitions of the sets $\ul$ and $\vl$ at the beginning of Section \ref{sec:2}, it follows  that   each sequence in   Table \ref{tab:1}  corresponds to a unique base $q\in\vl\setminus\ul$.
Therefore, by Lemma \ref{lem:25} and Theorem \ref{th1} we immediately obtain an interesting feature about the topological transitivity of the subshifts $(\VB_q, \sigma)$, i.e.,  it is neither a generic or exceptional property with respect to Lebesgue measure:
$$0<\textrm{Leb}\left(\set{q \in (1, M+1]: (\VB_q,\sigma) \textrm{ is transitive}}\right) < M+1.$$

Moreover, as we have show in Table $\ref{tab:1}$ for any $s\in\set{k+1,\ldots,M}$ and any $j\ge 2$ the  sequence $(s^j \overline{s})^\f$ is irreducible, and it converges to  $s^\f$ as $j\ra\f$. So, there exists $q\in \vl\setminus\ul$ arbitrarily close to $s+1$ for which $\alpha(q)$ is irreducible. Similarly, from Table \ref{tab:1} it follows that  for any $j\ge 2$ the sequence $(s^j \overline{s}^j)^\f$ is not   irreducible, and it converges to $s^\f$ as $j\ra\f$. So there also exists  $q\in \vl\setminus {\ul}$ arbitrarily close to $s+1$ for which $\alpha(q)$ is not irreducible. Consequently, the transitivity of
the subshifts $(\VB_q, \sigma)$ doesn't become a generic or exceptional property as we approach $s+1$, i.e., for all $\ep>0$
$$0<\textrm{Leb}\left(\set{q \in [s+1-\ep, s+1+\ep]: (\VB_q,\sigma) \textrm{ is transitive}}\right) < 2\ep.$$
This observation demonstrates the richness of the topological dynamics of the family of subshifts $(\VB_q, \sigma)$.

By Lemma \ref{lem:35}  the subshift  $(\VB_{q_T}, \sigma)$ is a transitive sofic  subshift  which is not a shift of finite type. It would be interesting to characterize the set of $q$ for which $(\VB_q, \sigma)$ is a transitive sofic subshift. Similarly, it would be interesting to characterize the set of $q$ for which $(\VB_q, \sigma)$ has the specification property. Another interesting problem to study would be to explicitly calculate the Lebesgue measure of the transitive bases and the non transitive bases.

\section{Irreducible and $*$-irreducible intervals}\label{sec:4}

\noindent Recall from Definition \ref{def:29} that an interval $[p_L, p_R]\subseteq(q_c, M+1]$ is  an  {irreducible (or $*$-irreducible) interval} if  the quasi-greedy expansion $\al(p_L)$ is irreducible (or $*$-irreducible), and there exists a word $a_1\ldots a_m$ with $a_m<M$ such that
\begin{equation}\label{eq:41}
\al(p_L)=(a_1\ldots a_m)^\f,\quad  \al(p_R)=a_1\ldots a_m^+(\overline{a_1\ldots a_m})^\f.
\end{equation}
In this case we call $[p_L, p_R]$   the \emph{irreducible (or $*$-irreducible) interval generated by $a_1\ldots a_m$}.

 \begin{lemma}\label{lem:41}
Let $[p_L, p_R]$ be an irreducible or $*$-irreducible interval generated by $a_1\ldots a_m$. Then $m$ is the smallest period of $\al(p_L)$, and
\[
\overline{a_1\ldots a_{m-i}}\lle a_{i+1}\ldots a_m\prec a_1\ldots a_{m-i}\quad \textrm{for all}\quad 1\le i<m.
\]
Furthermore, the interval $[p_L, p_R]$ is well-defined.
\end{lemma}
\begin{proof}
Let $[p_L, p_R]$ be an irreducible or $*$-irreducible interval generated by $a_1\ldots a_m$. First we show that $m$ is the least period of $\al(p_L)$.   Suppose $j<m$ is a period of $\al(p_L)$. Then $a_1\ldots a_m=(a_1\ldots a_j)^na_1\ldots a_{r}$ for some $n\ge 1$ and $r\in\set{1,\ldots,j}$. By Lemma \ref{lem:21} it follows that the sequence
 \[
 a_1\ldots a_m^+(\overline{a_1\ldots a_m})^\f=(a_1\ldots a_j)^n a_1\ldots a_{r}^+(\overline{a_1\ldots a_m})^\f
 \]
 can not be the quasi-greedy expansion of $1$ for some base $q$. This leading to a contradiction with  (\ref{eq:41})  that $\al(p_R)=a_1\ldots a_m^+(\overline{a_1\ldots a_m})^\f.$
  So $m$ is the least period of $\al(p_L)=(a_1\ldots a_m)^\f$.

   Now we turn to prove the inequalities. Since $m$ is the least period of $\al(p_L)$, the greedy $p_L$-expansion of $1$ is $a_1\ldots a_m^+0^\f$. Which implies
 \begin{equation}\label{eq:42}
 a_{i+1}\ldots a_m\prec a_{i+1}\ldots a_m^+\lle a_1\ldots a_{m-i}\quad\textrm{for all} \quad 1\le i<m.
 \end{equation}
 On the other hand, since $(a_1\ldots a_m)^\f=\al(p_L)\in\vb$ we have
 \begin{equation}\label{eq:43}
 a_{i+1}\ldots a_m\lge \overline{a_1\ldots a_{m-i}}\quad\textrm{for all}\quad 1\le i<m.
 \end{equation}
 Then the inequalities in the lemma follows by (\ref{eq:42}) and (\ref{eq:43}).

 Finally, we show that the interval $[p_L, p_R]$ is well-defined. Then it suffices to show that $\al(p_R)=a_1\ldots a_m^+(\overline{a_1\ldots a_m})^\f$ is indeed a quasi-greedy expansion of $1$ for some base $p_R$. By (\ref{eq:42}) and (\ref{eq:43}) it follows that
 \[
 \si^n(a_1\ldots a_m^+(\overline{a_1\ldots a_m})^\f)\lle a_1\ldots a_m^+(\overline{a_1\ldots a_m})^\f\quad\textrm{for all}\quad n\ge 0.
 \]
 By Lemma \ref{lem:21} we conclude that $p_R >p_L$ is well-defined.
\end{proof}
Inspecting  Lemma \ref{lem:41} and Definition \ref{def:310} it follows that for each word $a_1\ldots a_m$ which generates an irreducible or $*$-irreducible interval the associated   word $a_1\ldots a_m^+$ is primitive.

In this section we will investigate the topological properties of irreducible and $*$-irreducible intervals. In particular, we show that irreducible and $*$-irreducible intervals are pairwise disjoint. We show that the end points of each irreducible (repectively $*$-irreducible) intervals can be approximated  by the left end points of irreducible (respectively $*$-irreducible) intervals from above and below. Moreover,  we show that every irreducible interval is a subset of $(q_T, M+1)$, and every $*$-irreducible intervals is contained in $(q_c, q_T)$. We will show in Section \ref{sec:5} that irreducible intervals cover almost every point of $(q_T, M+1)$, and $*$-irreducible intervals cover almost every point of $(q_c, q_T)$.

Recall that $(\la_i)=\al(q_c)$ is the generalized Thue-Morse sequence defined in (\ref{eq:23}).  The following property for $(\la_i)$ was established in \cite[Lemma 4.2]{Kong_Li_2015}.
\begin{lemma}
  \label{lem:42}
For all  $n\ge 0$ we have
  \[
\overline{\la_1\ldots\la_{2^n-i}}\prec\la_{i+1}\ldots\la_{2^n}\lle\la_1\ldots\la_{2^n-i}\quad\textrm{for all }0\le i<2^n.
\]
\end{lemma}
Then by Lemma \ref{lem:42} and Definition \ref{def:310} it follows that for each $n\ge 0$ the word $\la_1\ldots \la_{2^n}$ is primitive.
Recall from (\ref{eq:25}) that $\xi(n)$  is defined by
\[
\xi(n)=\left\{
\begin{array}
  {lll}
  \la_1\ldots\la_{2^{n-1}}(\overline{\la_1\ldots\la_{2^{n-1}}}\,^+)^\f&\textrm{if}& M=2k,\\
  \la_1\ldots\la_{2^n}(\overline{\la_1\ldots\la_{2^n}}\,^+)^\f&\textrm{if}& M=2k+1.
\end{array}
\right.
\]
Here we emphasise that the sequence $(\la_i)$ depends on $M$.
By (\ref{eq:22}) we have $\al(q_T)=\xi(1)$.

\begin{lemma}
  \label{lem:43}
  For each positive integer $n$ there exists a unique $q_{T_n}\in(q_c, q_T]\cap\ul$ such that
  \[
  \al(q_{T_n})=\xi(n).
  \]
  Moreover, $(q_{T_n})_{n=1}^{\infty}$ is strictly decreasing and converges to $q_c$ as $n\ra\f$.
\end{lemma}
\begin{proof}
Take $n\in\N$. By Lemma \ref{lem:42} it follows that
\[\overline{\xi(n)}\prec\si^i(\xi(n))\prec\xi(n)\quad \textrm{for all }i\ge 1.\]
Therefore, by the lexicographic characterization of $\ul$ there exists a unique $q_{T_n}\in(q_c, q_T]\cap \ul$ such that  $\al(q_{T_n})=\xi(n)$. In particular, $q_{T_1}=q_T$.

In the following we prove that $\set{\xi(n)}_{n=1}^\f$ is strictly decreasing. By (\ref{eq:24}) it follows that
\begin{align*}
  \la_1\ldots\la_{2^n}(\overline{\la_1\ldots\la_{2^n}}\,^+)^\f&=\la_1\ldots\la_{2^{n-1}}\overline{\la_1\ldots\la_{2^{n-1}}}\,^+(\overline{\la_1\ldots\la_{2^{n-1}}}\la_1\ldots\la_{2^{n-1}})^\f\\
  &\prec\la_1\ldots\la_{2^{n-1}}(\overline{\la_1\ldots\la_{2^{n-1}}}\,^+)^\f.
\end{align*}
Then by  (\ref{eq:25}) we have $\xi(n+1)\prec\xi(n)$ for all $n\ge 1$. Note that $\xi(n)\ra(\la_i)$ as $n\ra\f$ with respect to  the order topology. Moreover, the map which sends $\al(q)\ra q$ when restricted to $\ul$ is continuous. So
by Lemma \ref{lem:21} we conclude that $q_{T_n} \searrow q_c$ as $n\ra\f$.
\end{proof}
Note by Lemma \ref{lem:43} that $q_{T_1}=q_T$ and    $q_{T_n}\searrow q_c$ as $n\ra\f$. By (\ref{eq:32}) it follows that
\begin{equation}\label{eq:44}
q_G<q_{NT}<q_c<\cdots<q_{T_{n+1}}<q_{T_n}<\cdots  <q_{T_1}=q_T.
\end{equation}
Then the sequence $(q_{T_n})$   form a partition of $(q_c, M+1]$:
\begin{equation}\label{eq:45}
(q_c, M+1]=[q_T, M+1]\cup\bigcup_{n=1}^\f[q_{T_{n+1}}, q_{T_n}),
\end{equation}
where the unions on the right hand side are pairwise disjoint.

In the following lemma we show that the irreducible intervals are contained inside $(q_T, M+1)$, and the $*$-irreducible interval are contained in $(q_{T_{n+1}}, q_{T_n})$ for some $n\in\N$.

\begin{lemma}
  \label{lem:44}
  \begin{enumerate}
  \item Each irreducible interval is contained in $(q_T, M+1)$.
  \item  Let $[p_L, p_R]$ be a $*$-irreducible interval generated by $a_1\ldots a_m$. Then
 there exists a unique $n\in\N$ such that $[p_L, p_R]\subseteq(q_{T_{n+1}}, q_{T_n})$.
   Moreover, $m\ge 3\cdot 2^{n-1}$ if $M=2k$, and $m\ge 3\cdot 2^n$ if $M=2k+1$.
 \end{enumerate}
\end{lemma}
\begin{proof}
Since the proof for $M=2k+1$ is similar, we only consider the case where $M=2k$.

First we prove (1). Note by Lemma \ref{lem:32}  that $\al(q)$ is not irreducible for any $q_c<q\le q_T$. So any irreducible interval must belong to $(q_T, M+1]$. Observe that $\al(M+1)=M^\f$. It is clear from the definition that for any irreducible interval $[p_L,p_R]$ we must have $\al(p_R)\neq M^{\infty}$. Therefore $M+1$ cannot be the right endpoint of an irreducible interval. This establishes (1).

Now we prove (2). Suppose that $[p_L, p_R]$ is a $*$-irreducible interval generated by $a_1\ldots a_m$. Then $\al(p_L)=(a_1\ldots a_m)^\f$ is $*$-irreducible. By Definition \ref{def:28} there exists a unique $n\in\N$ such that $\xi(n+1)\lle\al(p_L)\prec \xi(n)$. Thus, by Lemma \ref{lem:43}
 there exists a unique $n\in\N$ such that $p_L\in[q_{T_{n+1}}, q_{T_n})$. Observe by Lemma \ref{lem:42} that $\al(q_{T_{n+1}})=\lambda_1\ldots \la_{2^n} (\overline{\la_1\ldots\la_{2^n}}^+)^{\infty}$ is not periodic. So $p_L \in (q_{T_{n+1}}, q_{T_n})$.  By Lemma \ref{lem:21} we have $\al(q_{T_{n+1}}) \prec \al(p_L)\prec\al(q_{T_n})$. This implies that
\begin{equation}\label{eq:46}
  \la_1\ldots\la_{2^n}(\overline{\la_1\ldots \la_{2^n}}\,^+)^\f \prec (a_1\ldots a_m)^\f\prec \la_1\ldots\la_{2^{n-1}}(\overline{\la_1\ldots\la_{2^{n-1}}}\,^+)^\f.
  \end{equation}

Note by (\ref{eq:24}) that $\la_{1}\ldots \la_{2^n}=\la_1\ldots \la_{2^{n-1}}\overline{\la_1\ldots\la_{2^{n-1}}}\,^+$. Then
 \[a_1\ldots a_{3\cdot 2^{n-1}}=\la_1\ldots \la_{2^{n-1}}(\overline{\la_1\ldots\la_{2^{n-1}}}\,^+)^2\quad\textrm{or}\quad \la_1\ldots \la_{2^{n-1}} \overline{\la_1\ldots\la_{2^{n-1}}}\,^+ \overline{\la_1\ldots\la_{2^{n-1}}}. \] Therefore we can conclude that $m\ge 3\cdot 2^{n-1}$. To finish the proof it remains to show that $p_R<q_{T_n}$, or equivalently, to show $\al(p_R)<\xi(n)$.

Let us write $m=j\cdot 2^{n-1}+r$ with $j\ge 3$ and $r\in\set{0,\ldots, 2^{n-1}-1}$.   By (\ref{eq:46}) we have
\begin{equation}\label{eq:47}
a_1\ldots a_m\lle \la_1\ldots \la_{2^{n-1}}(\overline{\la_1\ldots\la_{2^{n-1}}}\,^+)^{j-1}\overline{\la_1\ldots\la_r}.
\end{equation}
Note by  Lemma \ref{lem:42}  that
\begin{equation}\label{eq:48}
\overline{\la_{r+1}\ldots \la_{2^{n-1}}}^+\lle\la_1\ldots \la_{2^{n-1}-r},\quad \overline{\la_1\ldots \la_r}\prec \la_{2^{n-1}-r+1}\ldots \la_{2^{n-1}}.
\end{equation}
This implies that equality in (\ref{eq:47}) cannot hold. Therefore,
 \begin{equation*}
 a_1\ldots a_m \prec \la_1\ldots \la_{2^{n-1}}(\overline{\la_1\ldots\la_{2^{n-1}}}\,^+)^{j-1}\overline{\la_1\ldots\la_r},
 \end{equation*}
which implies
\begin{equation}\label{eq:49}
 a_1\ldots a_m^+ \lle \la_1\ldots \la_{2^{n-1}}(\overline{\la_1\ldots\la_{2^{n-1}}}\,^+)^{j-1}\overline{\la_1\ldots\la_r}.
\end{equation}

 If strict inequality holds in (\ref{eq:49}) then $\al(p_R)=a_1\ldots a_m^+(\overline{a_1\ldots a_m})^\f \prec\al(q_{T_n})$. If  $$a_1\ldots a_m^+ = \la_1\ldots \la_{2^{n-1}}(\overline{\la_1\ldots\la_{2^{n-1}}}\,^+)^{j-1}\overline{\la_1\ldots\la_r},$$ then by (\ref{eq:48}) it follows that
  \begin{align*}
  a_1\ldots a_m^+(\overline{a_1\ldots a_m})^\f& = \la_1\ldots\la_{2^{n-1}} (\overline{\la_1\ldots\la_{2^{n-1}}}\,^+)^{j-1}\overline{\la_1\ldots\la_r}\,(\overline{\la_1\ldots \la_{2^{n-1}}}\overline{a_{2^{n-1}+1}\ldots a_m})^{\f}\\
  &\prec  \la_1\ldots\la_{2^{n-1}}(\overline{\la_1\ldots\la_{2^{n-1}}}\,^+)^\f.
  \end{align*}
Again, this implies $\al(p_R)\prec\al(q_{T_n})$. Then from Lemma \ref{lem:21} we obtain $p_R<q_{T_n}$. Therefore, we conclude that $[p_L, p_R]\subseteq(q_{T_{n+1}}, q_{T_n})$.
\end{proof}
By Lemma \ref{lem:44} it follows that   each $*$-irreducible interval $[p_L, p_R]$ is contained in a unique subinterval $(q_{T_{n+1}}, q_{T_n})$. In this case, we call $[p_L, p_R]$ an \emph{$n$-irreducible interval}.

Note by Table \ref{tab:1}  that there are infinitely many irreducible intervals. In the following lemma we show that  there also exist  infinitely many $*$-irreducible intervals.

\begin{lemma}\label{lem:45}
For any positive integers $n$ and $j$ the sequence
\[
(\la_1\cdots\la_{2^n}(\overline{\la_1\ldots\la_{2^n}}\,^+)^j\, \overline{\la_1\ldots \la_{2^n}})^\f
\]
is $*$-irreducible.
\end{lemma}
\begin{proof}
Since the proof for $M=2k$ is similar, we assume $M=2k+1$. Then $\xi(n)=\la_1\ldots \la_{2^n}(\overline{\la_1\ldots \la_{2^n}}\,^+)^\f$. By (\ref{eq:24}) it follows that
\[
  (\la_1\cdots\la_{2^n}(\overline{\la_1\ldots\la_{2^n}}\,^+)^j\, \overline{\la_1\ldots \la_{2^n}})^\f\prec \la_1\ldots\la_{2^n}(\overline{\la_1\ldots \la_{2^n}}^+)^\f=\xi(n)
\]
and
\begin{align*}
  (\la_1\cdots\la_{2^n}(\overline{\la_1\ldots\la_{2^n}}\,^+)^j\, \overline{\la_1\ldots \la_{2^n}})^\f&=(\la_1\ldots \la_{2^{n+1}}(\overline{\la_1\ldots \la_{2^n}}^+)^{j-1}\overline{\la_1\ldots \la_{2^n}})^\f\\
  &\succ\la_1\ldots \la_{2^{n+1}}(\overline{\la_1\ldots \la_{2^{n+1}}}^+)^\f=\xi(n+1).
\end{align*}
Denote by $(a_i):= (\la_1\cdots\la_{2^n}(\overline{\la_1\ldots\la_{2^n}}\,^+)^j\, \overline{\la_1\ldots \la_{2^n}})^\f$. In terms of Definition \ref{def:28} we will prove that for any $m=\ell 2^n+r$ with $\ell\ge 2$ and $r\in\set{0,1,\ldots, 2^n-1}$ we have
\begin{equation}\label{eq:410}
a_1\ldots a_m(\overline{a_1\ldots a_m}^+)^\f\prec (a_i)
\end{equation}
whenever $(a_1\ldots a_m^-)^\f\in\vb.$

First we assume $r\in\set{1,\ldots, 2^n-1}$. By the definition of $(a_i)$ we distinguish between the following two cases.
\begin{itemize}
\item $a_{m-r+1}\ldots a_m=\la_1\ldots \la_r$. By Lemma \ref{lem:42} it follows that
\[
a_{r+1}\ldots a_{2^n}=\la_{r+1}\ldots \la_{2^n}\succ \overline{\la_1\ldots \la_{2^n-r}}=\overline{a_1\ldots a_{2^n-r}}.
\]
So (\ref{eq:410}) holds in this case.

\item $a_{m-r+1}\ldots a_m=\overline{\la_1\ldots \la_r}$. Then
$
a_{m-r+1}\ldots a_m^-\prec \overline{\la_1\ldots\la_r}=\overline{a_1\ldots a_r},
$
which implies that $(a_1\ldots a_m^-)^\f\notin\vb$.
\end{itemize}

 Now we assume $r=0$. By the definition of $(a_i)$ we distinguish between the following  three cases.
 \begin{itemize}
 \item $a_{m-2^n+1}\ldots a_m=\la_1\ldots \la_{2^n}$. Then $a_{m+1}\ldots a_{m+2^n}=\overline{\la_1\ldots \la_{2^n}}^+\succ \overline{a_1\ldots a_{2^n}}$. This implies (\ref{eq:410}).

 \item $a_{m-2^n+1}\ldots a_m=\overline{\la_1\ldots \la_{2^n}}^+$. If $a_{m+1}\ldots a_{m+2^n}=\overline{\la_1\ldots \la_{2^n}}^+$, then by the same reasoning as in the previous case we have (\ref{eq:410}). If $a_{m+1}\ldots a_{m+2^n}=\overline{\la_1\ldots \la_{2^n}}$, then $a_{m+2^n+1}\ldots a_{m+2^{n+1}}=\la_1\ldots\la_{2^n}$. This implies that
 \[
 a_{m+1}\ldots a_{m+2^{n+1}}=\overline{\la_1\ldots \la_{2^n}}\la_1\ldots \la_{2^n}\succ\overline{a_1\ldots a_{2^{n+1}}}.
 \]
 So (\ref{eq:410}) also holds in this case.

 \item $a_{m-2^n+1}\ldots a_m=\overline{\la_1\ldots \la_{2^n}}$. Then $a_{m+1}\ldots a_{m+2^n}=\la_1\ldots \la_{2^n}\succ \overline{a_1\ldots a_{2^n}}$ which yields (\ref{eq:410}).
 \end{itemize}

 Therefore, we conclude from (\ref{eq:410}) that $(a_i)=(\la_1\cdots\la_{2^n}(\overline{\la_1\ldots\la_{2^n}}\,^+)^j\, \overline{\la_1\ldots \la_{2^n}})^\f$ is $*$-irreducible.
\end{proof}

  In the following lemma we show that the irreducible and $*$-irreducible intervals are pairwise disjoint.
\begin{lemma}
  \label{lem:46}
  \begin{enumerate}
  \item
  The irreducible   intervals are pairwise disjoint.
  \item
  The   $*$-irreducible intervals are pairwise disjoint.
  \end{enumerate}
\end{lemma}
\begin{proof}
First we prove (1).

Let $I(a_1\ldots a_m)=[p_L, p_R]$ and $I(b_1\ldots b_n)=[q_L, q_R]$ be two irreducible intervals generated by $a_1\ldots a_m$ and $b_1\ldots b_n$, respectively. Suppose on the contrary that $[p_L, p_R]\cap[q_L, q_R]\ne\emptyset$. We may assume $q_L\in[p_L, p_R]$. Note that $q_L\ne p_L$. Then by Lemma \ref{lem:21} we have $\al(p_L)\prec\al(q_L)\lle\al(p_R)$, i.e.,
  \begin{equation}
    \label{eq:411}
    (a_1\ldots a_m)^\f\prec(b_1\ldots b_n)^\f\lle a_1\ldots a_m^+(\overline{a_1\ldots a_m})^\f.
  \end{equation}
This implies $n\ge m$ and
$
a_1\ldots a_m\lle b_1\ldots b_m\lle a_1\ldots a_m^+.
$
By an easy check of  (\ref{eq:411}) we have  $n>m$.  Note that $\si^i((b_1\ldots b_n)^\f)\lle (b_1\ldots b_n)^\f$ for any $i\ge 0$. Then by (\ref{eq:411}) it follows that $b_1\ldots b_m\ne a_1\ldots a_m$. Therefore,
\[n>m\quad\textrm{ and }\quad b_1\ldots b_m=a_1\ldots a_m^+.\]
We have $(b_1\ldots b_m^-)^\f=(a_1\ldots a_m)^\f\in\vb$, by the irreducibility of $(b_1\ldots b_n)^\f$ we must therefore have
\[
a_1\ldots a_m^+(\overline{a_1\ldots a_m})^\f\prec (b_1\ldots b_n)^\f,
\] leading to a contradiction with (\ref{eq:411}).

Now we prove (2).  Let $I^*(c_1\ldots c_s)$ and $I^*(d_1\ldots d_t)$ be two $*$-irreducible intervals generated by $c_1\ldots c_s$ and $d_1\ldots d_t$ respectively. By Lemma \ref{lem:44} we may assume that both intervals are $n$-irreducible intervals for some $n\in\N$. Moreover, the periods $s, t>2^n$ if $M=2k$, and $s, t>2^{n+1}$ if $M=2k+1$.

By adapting the proof of (1) and using the fact that $s, t>2^n$ if $M=2k,$ and $s, t>2^{n+1}$ if $M=2k+1$, we can prove that $I^*(c_1\ldots c_s)\cap I^*(d_1\ldots d_t)=\emptyset$.
\end{proof}

{ In the following we are going to show that the left end point of each irreducible interval can be approximated by end points of irreducible intervals from  below, whereas the right end point of each irreducible interval can be approximated by end points of irreducible intervals from above. Furthermore, the similar approximation properties hold for  $*$-irreducible interval.}

Let
\begin{equation}\label{eq:412}
\begin{split}
 \I:&=\set{q\in[q_T, M+1]: \al(q)\textrm{ is irreducible}}\\
 \I^*:&=\set{q\in(q_c, q_T): \al(q)\textrm{ is }*\textrm{-irreducible}}.
 \end{split}
\end{equation}

\begin{lemma}
  \label{lem:47}
 $\I\cup \I^*\subseteq\overline{\ul}$. Moreover, $q_{T_n}\in \I^*$ for all $n\ge 2$.
\end{lemma}
\begin{proof}
Since the proof for  $M=2k+1$ is similar, we only consider the case where $M=2k$.

By the first part of the proof of Lemma \ref{lem:316} it follows that $\I\subseteq\overline{\ul}$.
Now we prove that $\I^*\subseteq\overline{\ul}$.
Fix $q\in \I^*$. Then $\al(q)=(\al_i)\in\vb$ and therefore $q\in\vl$. So it suffices to show that $q\notin\vl\setminus\overline{\ul}$.

 Suppose on the contrary that $q\in\vl\setminus\overline{\ul}$. Then by \cite[Theorem 1.3]{DeVries_Komornik_Loreti_2016} there exists $j\ge 1$ such that
 \begin{equation}\label{eq:413}
 \al(q)=(\al_1\ldots \al_j\,\overline{\al_1\ldots \al_j})^\f\quad\textrm{with}\quad (\al_1\ldots \al_j^-)^\f\in\vb.
 \end{equation}
By Definition \ref{def:28} of $*$-irreducible there exists $n\in\N$ such that
 \begin{equation}
   \label{eq:414}
   \la_1\ldots\la_{2^n}(\overline{\la_1\ldots\la_{2^n}}\,^+)^\f \lle(\al_1\ldots\al_j\,\overline{\al_1\ldots\al_j})^\f\prec \la_1\ldots\la_{2^{n-1}}(\overline{\la_1\ldots\la_{2^{n-1}}}\,^+)^\f.
 \end{equation}
Moreover, for the same $n$ as above, the definition of $*$-irreducible means that $j\le  2^n$. Then by (\ref{eq:414}) and (\ref{eq:24}) it follows that $\al_1\ldots\al_j=\la_1\ldots\la_j$. It is a consequence of Lemma \ref{lem:42} that
\[
(\al_1\ldots \al_j\,\overline{\al_1\ldots\al_j})^\f=(\la_1\ldots\la_j\,\overline{\la_1\ldots\la_j})^\f\prec   \la_1\ldots\la_{2^n}(\overline{\la_1\ldots\la_{2^n}}\,^+)^\f,
\]
leading to a contradiction with (\ref{eq:414}).
So $q\in\overline{\ul}$ and we conclude that $\I^*\subseteq\overline{\ul}$.

Finally, we   prove $q_{T_{n+1}}\in \I^*$ for all $n\in\N$. Note  that
\[(a_i):=\al(q_{T_{n+1}})=\la_1\ldots\la_{2^n}\,(\overline{\la_1\ldots\la_{2^n}}\,^+)^\f.\]
 Suppose $(a_1\ldots a_j^-)\in\vb$ with $j>2^n$. Then we write $j=u\cdot 2^n+r$ with $u\in\N$ and $r\in\set{1,2,\ldots, 2^n}$.
If $1\le r<2^n$, then by Lemma \ref{lem:42} it follows that
  \begin{align*}
    a_1\ldots a_j(\overline{a_1\ldots a_j}\,^+)^\f&=\la_1\ldots \la_{2^n}(\overline{\la_1\ldots \la_{2^n}}\,^+)^{u-1}\overline{\la_1\ldots \la_r} \big(\overline{\la_1\ldots\la_{2^n}}\overline{a_{2^n+1}\ldots a_j}^+\big)^\f\\
    &\prec\la_1\cdots\la_{2^n} (\overline{\la_1\ldots\la_{2^n}}\,^+)^\f=\;(a_i).
  \end{align*} If $r=2^n$, then it is easy to check that $a_1\ldots a_j(\overline{a_1\ldots a_j}\,^+)^\f\prec(a_i)$.

  Hence, $\al(q_{T_{n+1}})$ is $*$-irreducible. This implies $q_{T_{n+1}}\in \I^*$.
\end{proof}

By Lemma \ref{lem:47} we have $\I\cup \I^*\subseteq\overline{\ul}$. In the following we show that
 any  $q\in(\overline{\ul}\setminus  \I)\cap[q_T, M+1]$   belongs to a unique irreducible interval, and any $q\in(\overline{\ul}\setminus \I^*)\cap(q_c, q_T)$ belongs to a unique $*$-irreducible interval.

First we need the following technical lemma.
\begin{lemma}
  \label{lem:48}
  Let $(a_1\ldots a_m)^\f, (b_1\ldots b_n)^\f\in\vb$. If
  \begin{equation}\label{eq:415}
  (a_1\ldots a_m)^\f\prec (b_1\ldots b_n)^\f \prec a_1\ldots a_m^+(\overline{a_1\ldots a_m})^\f,
  \end{equation}
  then $b_1\ldots b_n^+(\overline{b_1\ldots b_n})^\f\prec a_1\ldots a_m^+(\overline{a_1\ldots a_m})^\f.$
\end{lemma}
\begin{proof}
Let $(c_i):=a_1\ldots a_m^+(\overline{a_1\ldots a_m})^\f$. Then by (\ref{eq:415})  and the same arguments as in the proof of  Lemma  \ref{lem:44} we have
\[n>m,\quad b_1\ldots b_m=a_1\ldots a_m^+,\quad\textrm{and}\quad  b_1\ldots b_{n-1} \preccurlyeq c_1\ldots c_{n-1}.\]
 If $b_1\ldots b_{n-1}\prec c_1\ldots c_{n-1}$, then we are done.

\vspace{0.7em}Now we assume $b_1\ldots b_{n-1}=c_1\ldots c_{n-1}$, and write $n=um+r$ with $u\ge 1$ and $r\in\set{1,2,\ldots,m}$. We claim that $r=m$ and
\[
b_1\ldots b_{n}=a_1\ldots a_m^+(\overline{a_1\ldots a_m})^{u-1}\overline{a_1\ldots a_m^+}.
\]
We will prove the claim by observing the following two cases.

\begin{enumerate}[(I)]
 \item Suppose that $1\le r<m$. Since
 \[b_1\ldots b_{n-1}=c_1\ldots c_{n-1}=a_1\ldots a_m^+(\overline{a_1\ldots a_m})^{u-1}\overline{a_1\ldots a_{r-1}},\]
 by \eqref{eq:415} and $(b_1\ldots b_{n})^\f\in\vb$ it follows that
\begin{equation}\label{eq:416}
    b_1\ldots b_{n}=a_1\ldots a_m^+(\overline{a_1\ldots a_m})^{u-1}\overline{a_1\ldots a_r}.
\end{equation}
Observe that (\ref{eq:416}) holds since if $b_n \neq \overline{a_r}$ then $b_n < \overline{a_r}$ which implies that
\[
b_{n-r+1}\ldots b_n\prec \overline{a_1\ldots a_r}=\overline{b_1\ldots b_r}
\]
This leads to a contradiction with $(b_1\ldots b_n)^\f\in\vl$.
 Note that $(a_1\ldots a_m)^\f\in\vb$. Then $a_1\ldots a_m^+\succ\overline{a_{r+1}\ldots a_m a_1\ldots a_r}$.
Hence,
  \[
  (b_1\ldots b_{n})^\f=(a_1\ldots a_m^+(\overline{a_1\ldots a_m})^{u-1}\overline{a_1\ldots a_r})^\f\succ  a_1\ldots a_m^+(\overline{a_1\ldots a_m})^\f.
  \]
Which contradicts (\ref{eq:415}).

\item Suppose that $r=m$. Observe that $(b_1\ldots b_{n})^\f\in\vb$. Then by using $b_1\ldots b_{n-1}=c_1\ldots c_{n-1}$ in (\ref{eq:415}) it follows that
  \[
  b_1\ldots b_{n}=a_1\ldots a_m^+(\overline{a_1\ldots a_m})^{u-1}\overline{a_1\ldots a_m^+}\quad\textrm{or}\quad  a_1\ldots a_m^+(\overline{a_1\ldots a_m})^{u}.
  \]
So to prove the claim it suffices to show that $b_1\ldots b_{n}\ne a_1\ldots a_m^+(\overline{a_1\ldots a_m})^{u}$. Suppose on the contrary that $b_1\ldots b_{n}=a_1\ldots a_m^+(\overline{a_1\ldots a_m})^{u}$. Then
  \[
  (b_1\ldots b_{n})^\f\succ a_1\ldots a_m^+(\overline{a_1\ldots a_m})^\f.
  \]
  Which contradicts (\ref{eq:415}).
\end{enumerate}
By the above we see that our claim is correct. The fact that $b_1\ldots b_n^+(\overline{b_1\ldots b_n})^\f\prec a_1\ldots a_m^+(\overline{a_1\ldots a_m})^\f$ is then an immediate consequence of this claim.
\end{proof}

\begin{lemma}
  \label{lem:49}
  \begin{enumerate}
  \item
 Let $q\in(\overline{\ul}\setminus  \I)\cap[q_T, M+1]$. Then there exists a unique irreducible interval $I$ such that $q\in I$.

 \item  Let $q\in(\overline{\ul}\setminus \I^*)\cap(q_c, q_T)$. Then there exists a unique $*$-irreducible interval $I$ such that $q\in I$.
 \end{enumerate}
\end{lemma}
\begin{proof}
Since the proof for $M=2k+1$ is similar, we only prove the lemma for $M=2k$.

First we prove (1). Take $q\in(\overline{\ul}\setminus  \I)\cap[q_T, M+1]$. Then $\al(q)=(\al_i)\in\vb$ is not irreducible. So there exists a smallest  integer $m\ge 1$ such that
  \begin{equation}
    \label{eq:417}
    (\al_1\ldots\al_m^-)^\f\in\vb\quad\textrm{and}\quad (\al_i)\lle\al_1\ldots\al_m(\overline{\al_1\ldots\al_m}\,^+)^\f.
  \end{equation}
    We start our proof by showing that $(c_i):=(\al_1\ldots\al_m^-)^\f$ is irreducible.

    First we assume $m=1$. By using $(\al_1^-)^\f\in\vb$ in (\ref{eq:417}) we must have $\al_1^-\ge \overline{\al_1}^+$, i.e., $\al_1-1\ge 2k+1-\al_1$. So, $\al_1\ge k+1$, and then $(\al_1^-)^\f$ is irreducible.

    Now we assume $m\ge 2$.
      Suppose $(c_1\ldots c_j^-)^\f\in\vb$. We will prove in the following two cases that
      \begin{equation}
        \label{eq:418}
        c_1\ldots c_j(\overline{c_1\ldots c_j}\,^+)^\f\prec (c_i).
      \end{equation}

Case (I). $j\ge m$.  Then we write $j=um+r$ with $u\ge 1$ and $r\in\set{0,1,\ldots,m-1}$.
If $r=0$, then (\ref{eq:418}) follows since $\al_1>\overline{\al_1}$ and $m\ge 2$ imply that
\[
(\al_1\ldots\al_m^-)^u( \overline{(\al_1\ldots\al_m^-)^u}\,^+)^\f\prec(\al_1\ldots\al_m^-)^\f.
\]

If $r\in\set{1,2,\ldots, m-1}$. Then since $(c_1\ldots c_j^-)^\f=((\al_1\ldots \al_m^-)\al_1\ldots\al_r^-)^\f\in\vb$ we have that
\[
\al_{r+1}\ldots \al_m^-\al_1\ldots\al_r\succ \al_{r+1}\ldots \al_m^-\al_1\ldots \al_r^-\lge \overline{\al_1\ldots\al_m^-}.
\]
This implies that
\begin{align*}
c_1\ldots c_j(\overline{c_1 \ldots c_j}\,^+)^\f&=(\al_1\ldots\al_m^-)^u\al_1\ldots \al_r(\overline{\al_1\ldots \al_m^-} \overline{c_{m+1}\ldots c_j}^+)^\f\\
&\prec (\al_1\ldots\al_m^-)^u\al_1\ldots \al_r \al_{r+1}\ldots \al_m^-\al_1\ldots \al_r \cdots\\
&=(\al_1\ldots \al_m^-)^\f=\; (c_i).
\end{align*}

Therefore, (\ref{eq:418}) holds when $j\ge m$.

\vspace{0.7em}Case (II). $0<j<m$.
Suppose on the contrary that (\ref{eq:418}) fails, i.e., $$\al_1\ldots\al_j(\overline{\al_1\ldots\al_j}\,^+)^\f\lge(\al_1\ldots\al_m^-)^\f.$$ Then
$
(\al_1\ldots\al_j^-)^\f\prec (\al_1\ldots\al_m^-)^\f\lle \al_1\ldots\al_j (\overline{\al_1\ldots\al_j}\,^+)^\f.
$
By Lemma \ref{lem:48} this implies that
\[
\al_1\ldots\al_m(\overline{\al_1\ldots\al_m}\,^+)^\f\prec\al_1\ldots\al_j(\overline{\al_1\ldots\al_j}\,^+)^\f.
\]
By (\ref{eq:417}) we have $(\al_i)\prec \al_1\ldots\al_j(\overline{\al_1\ldots\al_j}\,^+)^\f$. This leads to a contradiction with the minimality of $m$ defined in (\ref{eq:417}).

We have shown that $(\al_1\ldots\al_m^-)^\f$ is irreducible. Therefore by (\ref{eq:417}) it follows that $q$ belongs to the irreducible interval $I(\al_1\ldots\al_m^-)$. By Lemma \ref{lem:46} we may conclude that the interval $I(\al_1\ldots \al_m^-)$ is unique.

\vspace{1em}Now we turn to prove (2). By similar arguments to those used above, we can show that for each $n\ge 1$  every $q\in(\overline{\ul}\setminus \I^*)\cap(q_{T_{n+1}}, q_{T_n})$ belongs to a unique $*$-irreducible interval. Note by Lemma \ref{lem:47} that $q_{T_n}\in \I^*$ for all $n\ge 2$. Therefore, by Lemmas \ref{lem:44} and \ref{lem:46} it follows that any $q\in(\overline{\ul}\setminus \I^*)\cap(q_c, q_T)$ belongs to a unique $*$-irreducible interval.
\end{proof}

In our next result we show that for an irreducible or $*$-irreducible interval its left endpoint belongs to $\overline{\ul}\setminus\ul$ and its right endpoint belongs to $\ul$.
\begin{lemma}
  \label{lem:410}
  Let $[p_L, p_R]$ be an irreducible or $*$-irreducible interval. Then
    $p_L\in\overline{\ul}\setminus\ul$ and $p_R\in\ul$.
\end{lemma}
 \begin{proof}
Let $[p_L, p_R]$ be an irreducible or $*$-irreducible interval generated by $a_1\ldots a_m$. Then $p_L\in \I\cup \I^*$.    By  Lemma  \ref{lem:47} it follows that $p_L\in\overline{\ul}$. Note that $\al(p_L)=(a_1\ldots a_m)^\f$. Then by the formulations of $\ul$ and $\overline{\ul}$ given in Section \ref{sec:2}  it follows that $p_L\in\overline{\ul}\setminus\ul$.

By Lemma \ref{lem:41} it follows that
    \[
   \overline{a_1\ldots a_{m-i}}\lle a_{i+1}\ldots a_m\prec a_{i+1}\ldots a_m^+\lle a_1\ldots a_{m-i}\quad\textrm{for all }1\le i<m.
   \]
   This implies that
   \[
   \overline{a_1\ldots a_m^+}(a_1\ldots a_m)^\f\prec\si^n(a_1\ldots a_m^+(\overline{a_1\ldots a_m})^\f )\prec a_1\ldots a_m^+(\overline{a_1\ldots a_m})^\f
      \]
   for all $n\ge 1$. So $p_R\in\ul$.
 \end{proof}

{ Now we show that the left end point  of each irreducible interval can be approximated by end points of irreducible intervals from below, whereas the right end point of each irreducible interval can be approximated by end points of irreducible intervals from above. Similarly, for $*$-irreducible intervals we have the same approximation properties.}

\begin{proposition}
  \label{prop:411}
Let $[p_L, p_R]\subseteq(q_c, M+1]$ be an irreducible (respectively $*$-irreducible) interval. Then there exists a sequence of irreducible (respectively $*$-irreducible) intervals $\set{[p_L(n), p_R(n)]}_{n=1}^\f$ such that $p_L(n)$ strictly  increases to  $p_L$ as $n\ra\f$. Moreover, there exists another sequence of irreducible (respectively $*$-irreducible) intervals $\set{[p_L'(n), p_R'(n)]}_{n=1}^\f$ such that $p_L'(n)$ strictly decreases to   $p_R$ as $n\ra\f$.
 \end{proposition}
 \begin{proof}
 Since the proof for $*$-irreducible intervals is similar, we only give the proof for irreducible intervals.

 Let $[p_L, p_R]$ be an irreducible interval.   By Lemma \ref{lem:410} we have $p_L\in\overline{\ul}\setminus\ul$ and $p_R\in\ul$. Then, since $\ul$ is a perfect set and $\overline{\ul} \setminus \ul$ is a countable dense subset of $\overline{\ul}$ (see \cite[Theorem 1.1]{Komornik_Loreti_2007}) there exist sequences $(p_i), (r_i)\in\overline{\ul}\setminus\ul$ such that
   \begin{equation}
     \label{eq:419}
     p_i\nearrow p_L\quad\textrm{and}\quad r_i\searrow p_R\quad\textrm{as}\quad i\ra\f.
   \end{equation}
  Now we construct a sequence of irreducible intervals $\set{[p_L(n), p_R(n)]}_{n=1}^{\infty}$ in terms of $(p_i)$ such that $p_L(n)$ strictly increases to $p_L$ as $n\ra\f$.

 First we claim that there exists a unique  irreducible interval $I_1=[p_L(1),p_R(1)]$ containing $p_1$. Note that $p_1\in\overline{\ul}\setminus\ul$.  If $\al(p_1)$ is irreducible, then by noting that $\al(p_1)$ is periodic we set $I_1=[p_L(1), p_R(1)]$ with its left endpoint $p_L(1)=p_1$. If $\al(p_1)$ is not irreducible, then let $I_1=[p_L(1), p_R(1)]$ be the unique irreducible interval containing $p_1$ given to us by Lemma \ref{lem:49}. So $I_1=[p_L(1),p_R(1)]$ is well-defined.

   By Lemma \ref{lem:46} we know that $I_1\cap [p_L, p_R]=\emptyset$. By (\ref{eq:419}) this implies that $p_R(1)<p_L$. So by (\ref{eq:419}) there exists a least integer $N_1$ such that $p_{N_1}\in I_1$ but $p_{N_1+1}\notin I_1$. Now let $I_2=[p_L(2), p_R(2)]$ be the unique irreducible  interval containing $p_{N_1+1}\in\overline{\ul}\setminus\ul$. By a similar argument to that given above we can verify that $I_2$ is well-defined. By Lemma \ref{lem:46} it follows that $I_1\cap I_2=\emptyset$ and $I_2\cap [p_L, p_R]=\emptyset$. By (\ref{eq:419}) this implies that $p_L(1)<p_L(2)<p_L$. Again, by (\ref{eq:419}) there exists a least integer $N_2>N_1$ such that $p_{N_2}\in I_2$ but $p_{N_2+1}\notin I_2$.

   By iteration of the above arguments we construct a sequence of irreducible intervals $\set{[p_L(n), p_R(n)]}$ such that $p_L(n)$ strictly increases to $p_L$ as $n\ra\f$.

   In a similar way as in the construction of the sequence $\set{[p_L(n), p_R(n)]}_{n=1}^{\infty}$ one can also construct the sequence of irreducible intervals $\set{[p_L'(n), p_R'(n)]}_{n=1}^{\f}$ in terms of $(r_i)$ such that $p_L'(n)$ strictly decreases to $p_R$ as $n\ra\f$.
 \end{proof}

\section{Entropy plateaus of $H$}\label{sec:5}

\noindent Note by \cite[Theorems 1.1 and 1.2]{Komornik_Kong_Li_2015_1} that $H(q)=0$ if $q\le q_c$, and $H(q)>0$ if $q>q_c$. This implies  that the first plateau of $H$ is $(1, q_c]$. Furthermore,  the  entropy function $H$ is constant on each connected component $(q_0, q_0^*)$ of   $(q_c, M+1]\setminus\overline{\ul}$. This implies that the endpoints of each plateau of $H$   belong to $\overline{\ul}$. Therefore, to determine the entropy plateaus of $H,$ it suffices to consider $q\in(q_c, M+1]\cap\overline{\ul}$. By Lemma \ref{lem:33} it follows that $(\VB_q, \si)$ is not transitive for any $q\in(q_c, q_T)$. This makes determining the plateaus of $H$ in $(q_c, q_T)$ more intricate. For this reason we investigate the entropy plateaus of $H$ in $(q_c, q_T)$ and $[q_T, M+1]$ separately.

 \subsection{Entropy plateaus of $H$ in $[q_T, M+1]$}
 In this part we give the proof of Theorem \ref{th2} for $q\in[q_T, M+1]$.

Note by Lemma \ref{lem:44} that every irreducible interval is a subset of $(q_T, M+1]$. Firstly, we show that the entropy function $H$ is constant on each irreducible interval.

\begin{lemma}
  \label{lem:51}
  Let $[p_L, p_R]\subseteq(q_T, M+1]$ be an irreducible interval. Then
  \begin{enumerate}
  \item $(\VB_{p_L}, \si)$ is a transitive subshift of finite type.
  \item $H(q)=H(p_L)$ for all $q\in[p_L, p_R]$.
  \end{enumerate}
\end{lemma}
 \begin{proof}
Let $[p_L, p_R]$ be an irreducible interval generated by $a_1\ldots a_m$. First we prove $(1)$. Note that $\al(p_L)=(a_1\ldots a_m)^\f$ and $\al(p_L)$ is irreducible. Then by Theorem \ref{th1} it follows that $(\VB_{p_L}, \si)$ is a transitive subshift of finite type.

Now we prove $(2)$, i.e., $H$ is constant in $[p_L, p_R]$. Note that $\VB_p\subseteq\VB_q$ for all $p<q$. This implies that $H(p_L)= h(\VB_{p_L})\leq h(\VB_q)=H(q)$ for all $q\in[p_L, p_R]$. Thus it suffices to show $h(\VB_{p_R})\le h(\VB_{p_L})$.

We know that $(\VB_{p_L},\si)$ is a transitive subshift of finite type. Denote by $A$ the adjacency matrix of the directed graph which represents $(\VB_{p_L}, \si)$, and let $\la=\la(A)$ be the spectral radius  of $A$. Recall that $\# B_n(\VB_{p_L})$ denotes the number of different words of length $n$ appearing in sequences of $\VB_{p_L}$. Then by the Perron-Frobenius theorem (cf.~\cite[Proposition 4.2.1]{Lind_Marcus_1995})  there exist two constants $C_1, C_2>0$ independent of $n$ such that
   \begin{equation}
     \label{eq:51}
     C_1 \la^{ {n} }\le \# B_n(\VB_{p_L})\le C_2 \la^{ {n} }\quad\textrm{for all}\quad n\ge 1.
   \end{equation}

Note that $p_L\in\overline{\ul}$ and $p_L>q_T$. Then by Lemma \ref{lem:25} (2) it follows that $(\VB_{q_T}, \sigma)$ is a proper subshift of $(\VB_{p_L}, \sigma)$.  By \cite[Corollary 4.4.9]{Lind_Marcus_1995} and Lemma \ref{lem:35} it follows that
   \begin{equation*}
    \log \la =h(\VB_{p_L}) > h(\VB_{q_T})=\left\{
   \begin{array}{lll}
   \log 2&\textrm{if}& M=2k,\\
   \frac{\log 2}{2}&\textrm{if}& M=2k+1.
   \end{array}
   \right.
   \end{equation*}
This implies that $\la\geq  2^{\frac{1}{m}}$ for any $m\ge 2$.  We claim that 
 \begin{equation}\label{eq:52}
     \la\ge 2^{\frac{1}{m}}\quad \textrm{for all }m\ge  1.
   \end{equation}
By the above arguments  it suffices to prove the inequality in (\ref{eq:52}) for $m=1$. Observe that when $m=1$ we have $\al(p_L)=(\al_1)^\f\succ\al(q_T)$. This implies $\al_1>\overline{\al_1}$. So, for $m=1$ we have $\set{\overline{\al_1}, \al_1}^\f\subseteq\VB_{p_L}$, and therefore,
\[
\log \la=h(\VB_{p_L})\ge h(\set{\overline{\al_1}, \al_1}^\f)=\log 2.
\]
This proves  (\ref{eq:52})   for $m=1$.

Take $(c_i)\in\VB_{p_R}\setminus\VB_{p_L}$. Observe that $\al(p_L)=(a_1\ldots a_m)^\f$ and $\al(p_R)=a_1\ldots a_m^+(\overline{a_1\ldots a_m})^\f$.  Then by the definition of $\VB_q$ there exists $j\ge 0$ such that $c_{j+1}\ldots c_{j+m}=a_1\ldots a_m^+$ or its reflection $\overline{a_1\ldots a_m^+}$. Note that the sequence $c_{j+1}c_{j+2}\ldots$ belongs to
   \[
   \VB_{p_R}=\set{(x_i): \overline{a_1\ldots a_m^+}(a_1\ldots a_m)^\f\lle\si^n((x_i))\lle a_1\ldots a_m^+(\overline{a_1\ldots a_m})^\f\textrm{ for all }n\ge 0}.
   \]
   By  the same arguments as those used in the proof of Proposition \ref{prop:37} it can be shown that that $c_{j+1}c_{j+2}\ldots\in X_{{\G_4}}$, where $(X_{\G_4}, \sigma)$ is the sofic subshift represented by   the labeled graph $\G_4=(G_4, \LB_4)$ (see Figure \ref{Fig:5}).
\begin{figure}[h!]
  \centering
 \begin{tikzpicture}[->,>=stealth',shorten >=1pt,auto,node distance=4cm,
                    semithick]

  \tikzstyle{every state}=[minimum size=0pt,fill=black,draw=none,text=black]

  \node[state] (A)                    { };
  \node[state]         (B) [ right of=A] { };

  \path[->,every loop/.style={min distance=0mm, looseness=60}]
   (A) edge [loop left,->]  node {$\overline{a_1\ldots a_m}$} (A)
            edge  [bend left]   node {$\overline{a_1\ldots a_m^+}$} (B)

        (B) edge [loop right] node {$a_1\ldots a_m$} (B)
            edge  [bend left]            node {$a_1\ldots a_m^+$} (A);
\end{tikzpicture}
  \caption{The picture of the labeled graph $\mathfrak G_4=(G_4, \LB_4)$.}\label{Fig:5}
\end{figure}
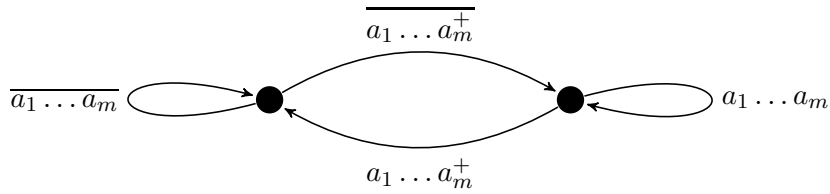

Note that the labeled graph $\G_4=(G_4, \LB_4)$ is right-resolving, and the length of each label is $m$. Then by direct calculation we obtain  $h(X_{\G_4})=\frac{1}{m}\log 2$. Since  $(X_{\G_4}, \sigma)$  is a transitive sofic subshift, by the Perron-Frobenius theorem  there exist two constants $C_3, C_4>0$ such that
  \begin{equation}
    \label{eq:53}
    C_3 \,2^{\frac{n}{m}}\le \# B_n(X_{\G_4})\le C_4\, 2^{\frac{n}{m}}\quad\textrm{for all}\quad n\ge 1.
  \end{equation}
  Therefore, by (\ref{eq:51})--(\ref{eq:53}) it follows that
  \[\begin{split}
    \# B_n(\VB_{p_R})\le\sum_{j=0}^n\# B_j(\VB_{p_L})\# B_{n-j}(X_{\G_4})&\le C_2 C_4\sum_{j=0}^n \la^j 2^{\frac{n-j}{m}}\\
    &\le   C_2 C_4\sum_{j=0}^n\la^{ {j} }\la^{n-j}\le C_2 C_4 (n+1)\la^{{n} }.
  \end{split}\]
  This implies that
  \[h(\VB_{p_R})=\lim_{n\ra\f}\frac{\log \# B_n(\VB_{p_R})}{n}\le  {\log\la}=h(\VB_{p_L}).\]
 \end{proof}
Now we prove Theorem \ref{th2} for  $q\in [q_T, M+1]$.
\begin{proposition}
  \label{prop:52}
{  The interval} $[p_L, p_R]\subseteq[q_T, M+1]$ is an entropy plateau of $H$ if and only if $[p_L, p_R]$ is an irreducible interval.
\end{proposition}
\begin{proof}
  First we prove the sufficiency. Let $[p_L, p_R]\subseteq[q_T, M+1]$ be an irreducible interval. Then by Lemma \ref{lem:51} it follows that $H(q)=H(p_L)$ for any $q\in[p_L, p_R]$. So to prove that $[p_L, p_R]$ is an entropy plateau it suffices to show that $H(q)\ne H(p_L)$ for any $q\notin[p_L, p_R]$. We split the proof into the following two cases.

  Case (I). $q<p_L$. By Lemma \ref{lem:410} we  have $p_L\in\overline{\ul}\subseteq\vl$. Then by Lemma \ref{lem:25} (2) it follows that $(\VB_q, \sigma)$ is a proper subshift of $(\VB_{p_L}, \sigma)$.   Note by Theorem \ref{th1} that $(\VB_{p_L}, \si)$ is a transitive subshift of finite type. So by \cite[Corollary 4.4.9]{Lind_Marcus_1995} we have
  \[
  H(q) =h(\VB_q)<h(\VB_{p_L})=H(p_L).
  \]

  Case (II). $q>p_R$. By Proposition \ref{prop:411} there exists $r\in(p_R, q)$ such that $r$ is the left end point of an irreducible interval. Then $\al(r)$ is periodic and irreducible. By Theorem \ref{th1} it follows that $(\VB_r, \si)$ is a transitive subshift of finite type. Note that $r\in\vl$. Then by Lemma \ref{lem:25} (2)  $(\VB_{p_L}, \si)$ is proper subshift  of $(\VB_r, \si)$. Again by \cite[Corollary 4.4.9]{Lind_Marcus_1995} we have
  \[
  H(q)\ge H(r)=h(\VB_r)>h(\VB_{p_L})=H(p_L).
  \]

  By Cases (I) and (II) we  conclude that $[p_L, p_R]$ is an entropy plateau of $H$.

 Now we turn to prove the necessity. By the sufficiency part of our proposition we know that each irreducible interval is an entropy plateau. So it suffices to prove that the union of all irreducible intervals covers $[q_T, M+1]$ up to a set of measure zero.

  Note that $q_T\in\overline{\ul}$ and $M+1\in\overline{\ul}$. Then $[q_T, M+1]\setminus\overline{\ul}=(q_T, M+1)\setminus\overline{\ul}$. Let $(q_0, q_0^*)$ be a basic interval of $[q_T, M+1]\setminus\overline{\ul}$. Note by \cite[Theorem 2.3]{Kong_Li_2015} there exists a word $a_1\ldots a_m$ such that $ \al(q_0)=(a_1\ldots a_m)^\f\in\vb$ and
  \[
 \al(q_0^*)= a_1\ldots a_m^+\,\overline{a_1\ldots a_m}\,\overline{a_1\ldots a_m^+}\cdots\prec a_1\ldots a_m^+(\overline{a_1\ldots a_m})^\f.
  \]
  This implies that $\al(q_0^*)$ is not irreducible, i.e., $q_0^*\notin \I$. Observe that $q_0^*\in\overline{\ul}$. Then by Lemma \ref{lem:49} there exists a unique irreducible interval $I$ such that $q_0^*\in I$. By \cite[Lemma 2.11]{Komornik_Kong_Li_2015_1} we know that $H(q)=H(q_0^*)$ for all $q\in[q_0, q_0^*]$. Moreover, by the sufficiency part of our proposition it follows that $I$ is the largest interval such that $H(q)=H(q_0^*)$ for all $q\in I$. So, $(q_0, q_0^*)\subseteq I$. Therefore,
  \[
  [q_T, M+1]\setminus\overline{\ul}=\bigcup(q_0, q_0^*)\subseteq \bigcup I,
  \]
  where the union on the right hand side is taken over all irreducible intervals. Note by \cite[Theorem 1.6]{Komornik_Kong_Li_2015_1} that   $\overline{\ul}$ is a Lebesgue null set. This implies that the union of all irreducible intervals covers $[q_T, M+1]$ up to a set of measure zero.
\end{proof}
\begin{remark}
  \label{rem:53}
\begin{itemize}
\item  By the proof of the necessity of Proposition \ref{prop:52} it follows that for each entropy plateau $[p_L, p_R]\subset[q_T, M+1]$  there exists $q_0^*\in(p_L, p_R)\cap\overline{\ul}$. By \cite[Theorems 1.1 and 1.2]{Kong_Li_Lu_Vries_2016} it follows that
  \[
  \dim_H\overline{\ul}\cap(p_L, p_R)>0.
  \]
  This implies that $(p_L, p_R)$ contains infinitely many connected components of $[q_T, M+1]\setminus\overline{\ul}$.

  \item Note that for each plateau interval $[p_L, p_R]\subset[q_T, M+1]$ the subshift $(\VB_{p_L}, \si)$ is  transitive. Moreover, by Proposition \ref{prop:411} it follows that for any $q\in\I$ there exists a sequence of irreducible intervals $([p_L(n), p_R(n)])$ such that $p_L(n)\in\vl\setminus\ul$ converges to $q$ as $n\ra\f$. On the other hand, the proof of the necessity of Proposition \ref{prop:52} also implies that there exists a sequence $(r_n)$ such that $r_n\in(\vl\setminus\ul)\setminus\I$ converges to $q$ as $n\ra\f$. Therefore, by Lemma \ref{lem:25} we conclude that the transitivity of $(\VB_q,\si)$ does not become a generic or exceptional property as we approach any $q\in\I$, i.e., for any $q\in \I$ and any $\ep>0$ we have
  \[
  0<Leb\left(\set{p\in[q-\ep, q+\ep]: (\VB_p, \si)\textrm{ is transitive}}\right)<2\ep.
  \]
  This strengthens the observation at the end of Section \ref{sec:3}.
  \end{itemize}
\end{remark}

\subsection{Entropy plateaus of $H$ in $(q_c, q_T)$}

Now we prove Theorem \ref{th2} for $q\in(q_c, q_T)$. By Lemma \ref{lem:44} we know that the $*$-irreducible intervals are all contained in $(q_c, q_T)$. Moreover, for each $*$-irreducible interval $[p_L, p_R],$ there exists a unique $n\in\N$ such that $[p_L, p_R]\subseteq(q_{T_{n+1}}, q_{T_n})$. In this case the interval $[p_L, p_R]$ is called an \emph{$n$-irreducible interval}.

In Lemma \ref{lem:35} we calculated the topological entropy of $(\VB_{q_{T_1}}, \sigma)$. In the following lemma we calculate the topological entropy of $(\VB_{q_{T_n}}, \sigma)$.
\begin{lemma}
  \label{lem:54}
  For each $n\in\N$ the subshift $(\VB_{q_{T_n}}, \si)$    contains  a unique transitive subshift of full topological entropy. Moreover,
  \[
  H(q_{T_n})=\left\{
  \begin{array}
    {lll}
    \frac{1}{2^{n-1}}\log 2&\textrm{if}& M=2k,\\
    \frac{1}{2^n}\log 2&\textrm{if}& M=2k+1.
  \end{array}\right.
  \]
\end{lemma}
\begin{proof}
  Since the proof for $M=2k+1$ is similar, we only give the proof for $M=2k$.

   By (\ref{eq:25}) and Lemma \ref{lem:43} it follows that
  \begin{equation}\label{eq:54}
  \al(q_{T_n})=\xi(n)=\la_1\ldots\la_{2^{n-1}}(\overline{\la_1\ldots\la_{2^{n-1}}}\,^+)^\f.
  \end{equation}
Let $(X_{\G_5}, \sigma)$ be the sofic subshift  represented by the right-resolving labeled graph ${\G_5}=(G_5, \LB_5)$ (see Figure \ref{Fig:6}). Then by Lemma \ref{lem:42} it follows that for all $(x_i)\in X_{\G_5}$ we have
 \[
\overline{\la_1\ldots\la_{2^{n-1}}}( \la_1\ldots\la_{2^{n-1}} ^-)^\f \lle \si^j((x_i))\lle \la_1\ldots\la_{2^{n-1}}(\overline{\la_1\ldots\la_{2^{n-1}}}\,^+)^\f
 \]for all $j\ge 0$. By (\ref{eq:54}) this implies that $X_{\G_5}\subseteq\VB_{q_{T_n}}$. Observe that the labeled graph ${\G_5}$ is right-resolving, and each label in $\LB_5$ has length $2^{n-1}$. Then by \cite[Theorem 4.3.3]{Lind_Marcus_1995} it follows that
   \begin{equation}\label{eq:55}
h(X_{\G_5})=\frac{\log 2}{2^{n-1}}.
\end{equation}
\begin{figure}[h!]
  \centering
\begin{tikzpicture}[->,>=stealth',shorten >=1pt,auto,node distance=4cm,
                    semithick]

  \tikzstyle{every state}=[minimum size=0pt,fill=black,draw=none,text=black]

  \node[state] (A)                    { };
  \node[state]         (B) [ right of=A] { };

  \path[->,every loop/.style={min distance=0mm, looseness=60}]
   (A) edge [loop left,->]  node {$\lambda_1\ldots \lambda_{2^{n-1}}^-$} (A)
            edge  [bend left]   node {$\lambda_1\ldots \lambda_{2^{n-1}}$} (B)

        (B) edge [loop right] node {$\overline{\lambda_1\ldots\lambda_{2^{n-1}}}^+$} (B)
            edge  [bend left]            node {$\overline{\lambda_1\ldots\lambda_{2^{n-1}}}$} (A);
\end{tikzpicture}
  \caption{The labeled edge graph ${\G_5}=(G_5, \LB_5)$.}\label{Fig:6}
\end{figure}
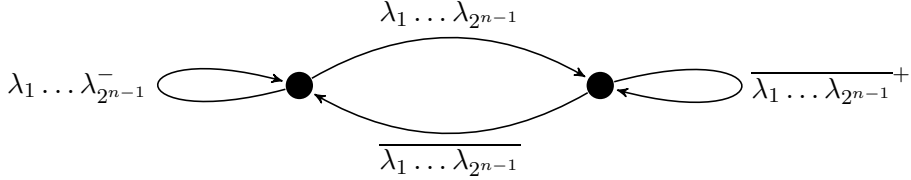

Note that if $(x_i)\in\VB_{q_{T_n}}$ contains the word $\la_1\ldots\la_{2^{n-1}}$ or its reflection $\overline{\la_1\ldots\la_{2^{n-1}}}$, then by the same argument used in the proof of Proposition \ref{prop:37} it follows that $(x_i)$ ends with an element of $X_{\G_5}$. Now we consider a new subshift $(Y, \si)$ contained in $(\VB_{q_{T_n}}, \si).$ This subshift consists of those elements of $\VB_{q_{T_n}}$ satisfying the additional property that the words $\la_1\ldots\la_{2^{n-1}}$ and $\overline{\la_1\ldots\la_{2^{n-1}}}$ are forbidden. Observe by (\ref{eq:23}) that $\al(q_c)\succ (\la_1\ldots \la_{2^{n-1}}^-)^\f$. This implies that
 \begin{align*}
   Y&\subseteq\set{(x_i): (\overline{\la_1\ldots\la_{2^{n-1}}}\,^+)^\f\lle \si^j((x_i))\lle (\la_1\ldots\la_{2^{n-1}}^-)^\f\textrm{ for all }j\ge 0}\\
   &\subseteq\set{(x_i): \overline{\al(q_c)}\lle\si^j((x_i))\lle\al(q_c) \textrm{ for all }j\ge 0}\\
   &=\VB_{q_c}.
 \end{align*}
  Note  that $h(\VB_{q_c})=0$. This implies that $h(Y)=0$.

 Therefore we may conclude that $(X_{\G_5}, \sigma)$ is the unique transitive subshift of $(\VB_{q_{T_n}}, \sigma)$ of full topological entropy. By (\ref{eq:55}) we also have
  \[
  H(q_{T_n})=h(\VB_{q_{T_n}}) =h(X_{\G_5})=\frac{\log 2}{2^{n-1}}.
  \]
  This completes the proof.
\end{proof}

Now we show that $H$ is constant on each $*$-irreducible interval.
\begin{lemma}
\label{lem:55}
  Let $[p_L, p_R]\subseteq(q_c, q_T)$ be a $*$-irreducible interval. Then $H(q)=H(p_L)$ for all $q\in[p_L, p_R]$.
\end{lemma}
\begin{proof}
Since the proof for $M=2k+1$ is similar, we only prove the lemma for $M=2k$.

Let $[p_L, p_R]$ be a $*$-irreducible interval generated by $a_1\ldots a_m$. By Lemma \ref{lem:44} there exists $n\in\N$ such that  $[p_L, p_R]\subseteq(q_{T_{n+1}}, q_{T_n})$. Moreover,
 $m\ge 3\cdot 2^{n-1}>2^n$.
Note that the entropy function $H$ is non-decreasing. So it suffices to prove $H(p_R)\le H(p_L)$.

  Observe that $\al(p_L)$ is periodic. Therefore $(\VB_{p_L}, \si)$ is a subshift of finite type. Let $G$ be the directed  graph representing $(\VB_{p_L}, \si)$, and denote by $\la$ its spectral radius.  Suppose that $G$ has $s$ distinct strongly connected components. Then by \cite[Theorem 4.4.4]{Lind_Marcus_1995} there exist constants $C_1$ and $C_2$ such that
  \begin{equation}
    \label{eq:56}
    C_1\la^n\le\# B_n(\VB_{p_L})\le C_2 n^s\la^n\quad\textrm{for all }n\ge 1.
  \end{equation}
  Note that $\VB_{q_{T_{n+1}}}\subseteq\VB_{p_L}$. Then by Lemma \ref{lem:54} it follows that
 \[
  \frac{\log 2}{2^n}=h(\VB_{q_{T_{n+1}}})\le h(\VB_{p_L})=\log \la.
\]
  This together with $m>2^n$ implies
  \begin{equation}
    \label{eq:57}
    \la\ge 2^{\frac{1}{2^n}}>2^{\frac{1}{m}}.
  \end{equation}

  Take $(c_i)\in\VB_{p_R}\setminus\VB_{p_L}$. Note that $\al(p_L)=(a_1\ldots a_m)^\f$ and  $\al(p_R)=a_1\ldots a_m^+(\overline{a_1\ldots a_m})^\f$. Then there exists $j\ge 0$ such that $c_{j+1}\ldots c_{j+m}=a_1\ldots a_m^+$ or $c_{j+1}\ldots c_{j+m}=\overline{a_1\ldots a_m^+}$.  By analogous reasoning to that used in the proof of Proposition \ref{prop:37} we can deduce that $c_{j+1}c_{j+2}\ldots \in X_{\G_4}$, where $(X_{\G_4}, \sigma)$ is the sofic subshift represented by the labeled graph ${\G_4}=(G_4, \LB_4)$ in Figure \ref{Fig:5}. Note that the topological entropy of $(X_{\G_4}, \sigma)$ is $h(X_{\G_4})=\frac{1}{m}\log 2$. Clearly, $(X_{\G_4}, \si)$ is a transitive sofic subshift.   Then by (\ref{eq:57}) and the Perron-Frobenius theorem there exists two constants $C_3$ and $C_4$ such that
  \begin{equation}
    \label{eq:58}
   C_3 2^{\frac{n}{m}}\le \# B_n(X_{\G_4})\le C_4 2^{\frac{n}{m}}<C_4 \la^n\quad\textrm{for all }n\ge 1.
  \end{equation}

  Therefore, by (\ref{eq:56}) and (\ref{eq:58}) it follows that
   \[\begin{split}
    \# B_n(\VB_{p_R})&\le\sum_{j=0}^n\# B_j(\VB_{p_L})\# B_{n-j}(X_{\G_4}) \le C_2 C_4\sum_{j=0}^n j^s \la^{ {j} } \la^{n-j}
    \le C_2 C_4 (n+1)n^{s}\la^{{n} }.
  \end{split}\]
  This implies that
  \begin{align*}
    h(\VB_{p_R}) =\lim_{n\ra\f}\frac{\log \# B_n(\VB_{p_R})}{n}
    &\le\lim_{n\ra\f}\frac{\log (C_2C_4)+\log(n+1)+s\log n+n\log \la}{n} \\
    &=\log\la=h(\VB_{p_L}).
  \end{align*}
  This completes the proof.
\end{proof}

Note by Lemma \ref{lem:33} the subshift $(\VB_q, \si)$ is not transitive for any $q\in(q_c, q_T)$. In the following we show that for each $q$ inside a   $*$-irreducible interval, the subshift $(\VB_q, \sigma)$ contains a unique transitive subshift of finite type of full topological entropy.

Recall from Definition \ref{def:310} that a word $\a=a_1\ldots a_m$ is primitive if for all $0\le i<m$ we have
\[
\overline{a_1\ldots a_{m-i}}\prec a_{i+1}\ldots a_m\lle a_1\ldots a_{m-i}.
\]
Moreover, by Definition \ref{def:311} the reflection recurrence word of $\a$ is $\R(\a)=a_1\ldots a_s$, where $s\in\set{0,1,\ldots,m-1}$ is the smallest integer satisfying
$
 a_{s+1}\ldots a_m^-=\overline{a_1\ldots a_{m-s}},
 $
assuming such an $s$ exists. When $s=0$ then $\R(\a)$ is the empty word. If such an $s$ doesn't exist then $\R(\a)=\a.$

\begin{lemma}
  \label{lem:56}
  Let  $[p_L, p_R]$ be an $n$-irreducible interval.
  \begin{enumerate}
\item If $M=2k$, then the word $
  \al_1(p_L)\ldots \al_j(p_L)
  $ with $2^n< j<2^{n+1}$ is primitive if and only if $j=3\cdot 2^{n-1}$ and
 \[\al_1(p_L)\ldots a_{3\cdot 2^{n-1}}(p_L)=\la_1\ldots \la_{2^{n-1}}(\overline{\la_1\ldots \la_{2^{n-1}}}\,^+)^2. \]Moreover, $\R(\la_1\ldots \la_{2^{n-1}}(\overline{\la_1\ldots \la_{2^{n-1}}}\,^+)^2)=
 \la_1\ldots \la_{2^n}.
$

\item If $M=2k+1$, then the word $
  \al_1(p_L)\ldots \al_j(p_L)
  $ with $2^{n+1}< j<2^{n+2}$ is primitive if and only if $j=3\cdot 2^{n}$ and
 \[\al_1(p_L)\ldots a_{3\cdot 2^{n}}(p_L)=\la_1\ldots \la_{2^{n}}(\overline{\la_1\ldots \la_{2^{n}}}\,^+)^2. \]
 Moreover, $\R(\la_1\ldots \la_{2^{n}}(\overline{\la_1\ldots \la_{2^{n}}}\,^+)^2)=
 \la_1\ldots \la_{2^{n+1}}.
$
\end{enumerate}
\end{lemma}
\begin{proof}
Since the proof of (2) for $M=2k+1$ is similar, we only give the proof of (1) for $M=2k$.

Let $(\al_i)=\al(p_L)$.  Note that $q_{T_{n+1}}<p_L<q_{T_n}$. By (\ref{eq:25}), (\ref{eq:24}) and Lemma \ref{lem:21} it follows that
\begin{equation}\label{eq:59}
\la_1\ldots\la_{2^{n-1}} \overline{\la_1\ldots\la_{2^{n-1}}}\,^+(\overline{\la_1\ldots\la_{2^{n-1}}}\, \la_1\ldots\la_{2^{n-1}})^\f
\prec (\al_i)\prec \la_1\ldots\la_{2^{n-1}}(\overline{\la_1\ldots\la_{2^{n-1}}}\,^+)^\f.
\end{equation}
This implies that
\[
\al_1\ldots\al_{3\cdot 2^{n-1}}=\la_1\ldots\la_{2^{n-1}} \overline{\la_1\ldots\la_{2^{n-1}}}\,^+ \,\overline{\la_1\ldots\la_{2^{n-1}}}\quad \textrm{or}\quad \la_1\ldots\la_{2^{n-1}}( \overline{\la_1\ldots\la_{2^{n-1}}}\,^+)^2.
\]
Again by (\ref{eq:59}), (\ref{eq:24}) and using that $(\al_i)\in\vb$ it follows that
\begin{align*}
   \al_1\ldots\al_{2^{n+1}-1}&=\la_1\ldots\la_{2^{n-1}} \overline{\la_1\ldots\la_{2^{n-1}}}\,^+ \,\overline{\la_1\ldots\la_{2^{n-1}}}\la_1\ldots\la_{2^{n-1}-1}\\
   &=\la_1\ldots\la_{2^n}\overline{\la_1\ldots \la_{2^n-1}},
\end{align*}
or
\begin{align*}
  \al_1\ldots\al_{2^{n+1}-1}&=\la_1\ldots\la_{2^{n-1}}( \overline{\la_1\ldots\la_{2^{n-1}}}\,^+)^2 \overline{\la_1\ldots\la_{2^{n-1}-1}}\\
  &=\la_1\ldots\la_{2^n}\, \overline{\la_1\ldots\la_{2^{n-1}}}\,^+ \overline{\la_1\ldots\la_{2^{n-1}-1}}.
\end{align*}

This implies that for all $2^n<j<2^{n+1}$ with $j\ne 2^n+2^{n-1}$ the word $\al_1\ldots \al_j$ has a prefix $\la_1\ldots \la_{2^n}$ and a suffix of the form $\overline{\la_1\ldots \la_\ell}$ for some $1\le \ell<2^n$. By
Definition \ref{def:310} it follows that  $\al_1\ldots\al_j$ cannot be primitive.
Furthermore, for  $j=2^n+2^{n-1}$ the word $\al_1\ldots\al_{2^n+2^{n-1}}$ is primitive if and only if it equals $\la_1\ldots\la_{2^{n-1}}( \overline{\la_1\ldots\la_{2^{n-1}}}\,^+)^2$.

Finally, observe that $\la_1\ldots\la_{2^n}=\la_1\ldots\la_{2^{n-1}}\overline{\la_1\ldots\la_{2^{n-1}}}\,^+$. Then by Lemma \ref{lem:42} it follow that the reflection recurrence word is given by
\[
\R(\la_1\ldots \la_{2^{n-1}}(\overline{\la_1\ldots \la_{2^{n-1}}}\,^+)^2)=\la_1\ldots\la_{2^{n-1}}\overline{\la_1\ldots\la_{2^{n-1}}}^+=
 \la_1\ldots \la_{2^n}.
\]
\end{proof}

Now we state an adapted version of Lemma \ref{lem:316} for $*$-irreducible sequences. The proof is essentialy the same so it is ommited.

\begin{lemma}\label{lem:57}
Let $q\in(q_c,q_T]\cap \vl$ be such that $\al(q)$ is $*$-irreducible and $q \in (q_{T_{n+1}}, q_{T_{n}})$.  Then there exist infinitely many integers $m > 2^{n+1}$ such that $\a=\al_1(q)\ldots\al_m(q)$ is primitive.

Moreover, for each of these $m$ there exists $N=N(m) \in \N$ such that for all $u\in\N$ and any $j\ge 0$  we have
\begin{align}
\overline{\al_1(q)\ldots\al_N(q)}&\prec\si^j(\a(\overline{\a}\,^+)^\f)\prec\al_1(q)\ldots\al_N(q),\label{eq:510}\\
  \overline{\al_1(q)\ldots\al_N(q)}&\prec\si^j((\a^-)^u(\R(\a)^-)^\f)\prec\al_1(q)\ldots\al_N(q),\label{eq:511}
\end{align}
and
\begin{align*}
 \overline{\al_1(q)\ldots\al_N(q)}&\prec\si^j(\overline{\a}\,( {\a}^-)^\f)\prec\al_1(q)\ldots\al_N(q),\\
  \overline{\al_1(q)\ldots\al_N(q)}&\prec\si^j((\overline{\a}\,^+)^u(\overline{\R(\a)}\,^+)^\f)\prec\al_1(q)\ldots\al_N(q).
\end{align*}
\end{lemma}

Using Lemma \ref{lem:57} we prove the following result which plays a key role in proving Lemma \ref{lem:59} and Proposition \ref{prop:510}, i.e. the existence and uniqueness of transitive components of full topological entropy.

\begin{lemma}
  \label{lem:58}
  Let $[p_L, p_R]$ be an $n$-irreducible interval. Then
  for any word $\om\in\L(\VB_{p_L})$ there exists a word $\ep\in\L(\VB_{p_L})$ such that
 \[\begin{array}{lll}
  \om\ep(\la_1\ldots\la_{2^n}^-)^\f\in\VB_{p_L}
  & \textrm{if}&  M=2k,\\
  \om\ep(\la_1\ldots\la_{2^{n+1}}^-)^\f\in\VB_{p_L}&\textrm{if}  &M=2k+1.
  \end{array}\]
\end{lemma}
\begin{proof}
Since the proof for $M=2k+1$ is similar, we will assume $M=2k$.

Take $\om=\om_1\ldots \om_n\in\L(\VB_{p_L})$. Firstly, recall that since $p_L \in(q_{T_{n+1}}, q_{T_n})$ we have
\[
 \al(q_{T_{n+1}})=\la_1\ldots\la_{2^n}(\overline{\la_1\ldots\la_{2^n}}^+)^\f\prec \al(p_L)\prec\la_1\ldots\la_{2^n}(\overline{\la_1\ldots\la_{2^{n-1}}}^+)^\f.
\]
This implies that
 $$(\la_1\ldots\la_{2^n}^-)^\f,\quad (\overline{\la_1\ldots\la_{2^n}}^+)^\f \in \VB_{p_L},
 $$
and then
$
\la_1\ldots\la_{2^n}^-, \quad\overline{\la_1\ldots\la_{2^n}}^+ \in \L(\VB_{p_L}).$

 Let $(\al_i)=\al(p_L)$. Then $\al_1\ldots \al_{2^n}=\la_1\ldots \la_{2^n}$.  Note that $p_L\in\overline{\ul}$.
Then by \cite[Lemma 4.1]{Komornik_Loreti_2007} there exists $m \in \N$ with $m>\max\set{|\om|, 2^{n}}$ such that $\a:=\al_1\ldots \al_m$ is primitive.    By Lemma \ref{lem:38} there exists a word $\de\in\L(\VB_{p_L})$ such that $\om\de$ has a suffix $\al_1\ldots \al_m$ or $\overline{\al_1\ldots \al_m}$. By symmetry we may assume ${\al_1\ldots \al_m}=\a$ is a suffix of $\om\de$.

Since $(\al_i)$ is $n$-irreducible we have $\al_1\ldots\al_j(\overline{\al_1\cdots\al_j}\,^+)^\f\prec(\al_i)$ whenever $2^n\le j\le m$ and $(\al_1\ldots \al_j^-)^\f\in\vb$. Let $N$ be a large integer such that
\[
\al_1\ldots\al_j(\overline{\al_1\cdots\al_j}\,^+)^\f\prec \al_1\ldots \al_N\quad\textrm{whenever}\quad (\al_1\ldots \al_j^-)^\f\in\vb\textrm{ with }2^n\le j\le m.
\]
From Lemma \ref{lem:57} (\ref{eq:510}) and the symmetric version of Lemma \ref{lem:57} (\ref{eq:511}) we know that $\a(\overline{\a}^+)^\f, \overline{\a}\,^+(\overline{\R(\a)}\,^+)^\f \in  \VB_{p_L}$. Applying a similar argument to the one in Proposition \ref{prop:317} we obtain that
\begin{equation}\label{eq:512}
\overline{\al_1\ldots\al_N}\prec\si^i(\om\de(\overline{\a}^+)^N(\overline{\R(\a)}\,^+)^\f)\prec\al_1\ldots\al_N\quad\textrm{for all }i\ge 1.
\end{equation}
Suppose $\R(\a)=\a$. Then by Definition \ref{def:311} it follows that
\[
\overline{\al_1\ldots\al_{m-i}}\prec\overline{\al_{i+1}\ldots\al_m}^+\prec \al_1\ldots\al_{m-i}\quad\textrm{for all}\quad 0\le i<m.
\]
Since $(\la_1\ldots\la_{2^n}^-)^\f \in \VB_{p_L}$ we obtain that $$\om\de(\overline{\a}^+)^N (\la_1\ldots\la_{2^n}^-)^\f \in \VB_{p_L}$$ which implies our result.

\vspace{0.7em}Now we assume $\R(\a)\ne\a$. Note by Lemma \ref{lem:313} that the length of $\R(\a)$ satisfies  $|\a|/2\le |\R(\a)|\le |\a|$. So there exists $i\in\N$ such that $2^n\le |\R^i(\a)|<2^{n+1}$. Note by Lemma \ref{lem:314} that $\R^i(\a)$ is primitive.   By Lemma \ref{lem:56} it follows that there exists  $j\in\set{i, i+1}$ such that
\[
\R^j(\a)=\al_1\ldots\al_{2^n}=\la_1\ldots\la_{2^n}.
\]
Therefore, by repeatedly applying Lemma \ref{lem:57} and following a similar argument as in the previous case we can conclude that
\[
\om\de(\overline{\a}^+)^N(\overline{\R(\a)}\,^+)^N (\overline{\R^2(\a)}\,^+)^N\cdots (\overline{\R^{j-1}(\a)}\,^+)^N (\overline{\la_1\ldots\la_{2^n}}\,^+)^\f\in\VB_{p_L}.
\]
Note by (\ref{eq:24}) that
\[(\overline{\la_1\ldots\la_{2^n}}\,^+)^\f=(\overline{\la_1\ldots\la_{2^{n-1}}}\la_1\ldots\la_{2^{n-1}})^\f=\overline{\la_1\ldots\la_{2^{n-1}}}(\la_1\ldots\la_{2^{n}}^-)^\f.\] This establishes the lemma.
\end{proof}

Now we show that for each $*$-irreducible interval $[p_L, p_R]$ the subshift $(\VB_{p_L}, \si)$ contains a unique transitive subshift of finite type of full topological entropy.
\begin{lemma}
\label{lem:59}
  Let $[p_L, p_R]$ be a $*$-irreducible interval. Then the subshift $(\VB_{p_L}, \si)$ contains a  unique transitive subshift of finite type $({}  X_{p_L}, \sigma)$  such that $h({}  X_{p_L})=h(\VB_{p_L})$. Furthermore, $\al(p_L)\in {}  X_{p_L}$.
\end{lemma}

\begin{proof}
Since the proof for $M=2k+1$ is similar,  we assume $M=2k$.

Let $[p_L, p_R]$ be a $*$-irreducible interval generated by $a_1\ldots a_m$. Then by Lemma \ref{lem:44}  there exists $n\in\N$ such that $[p_L, p_R]\subseteq(q_{T_{n+1}}, q_{T_n})$.
 By the proof of Lemma \ref{lem:54}, it follows that $\VB_{q_{T_{n+1}}}$ contains a unique transitive sofic subshift $({}  X_{q_{T_{n+1}}}, \sigma)$, which is represented by the labeled  graph $\G_6=(G_6, \LB_6)$, as in Figure \ref{Fig:7}, such that
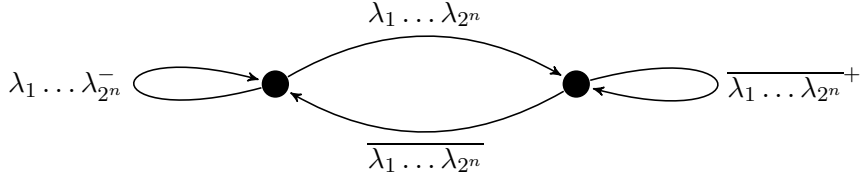
\begin{figure}[h!]
  \centering
\begin{tikzpicture}[->,>=stealth',shorten >=1pt,auto,node distance=4cm,
                    semithick]

  \tikzstyle{every state}=[minimum size=0pt,fill=black,draw=none,text=black]

  \node[state] (A)                    { };
  \node[state]         (B) [ right of=A] { };

  \path[->,every loop/.style={min distance=0mm, looseness=60}]
   (A) edge [loop left,->]  node {$\lambda_1\ldots\lambda_{2^n}^-$} (A)
            edge  [bend left]   node {$\lambda_1\ldots \lambda_{2^n}$} (B)

        (B) edge [loop right] node {$\overline{\lambda_1\ldots\lambda_{2^n}}^+$} (B)
            edge  [bend left]            node {$\overline{\lambda_1\ldots\lambda_{2^n}}$} (A);
\end{tikzpicture}
  \caption{The labeled graph $\G_6=(G_6, \LB_6)$.}\label{Fig:7}
\end{figure}
  $h({}  X_{q_{T_{n+1}}})=h(\VB_{q_{T_{n+1} } })$.

Note that $\al(p_L)=(a_1\ldots a_m)^{\infty}$ is periodic. This implies that
  $(\VB_{p_L}, \si)$ is a subshift of finite type. Observe that $\VB_{q_{T_{n+1}}}\subseteq\VB_q$.
  Let $({}  X_{p_L}, \sigma) $ be the maximal transitive subshift of finite type of $(\VB_{p_L}, \sigma)$ with respect to the inclusion containing ${}  X_{q_{T_{n+1}}}$. Note that ${}  X_{p_L}$ is well defined since $(\VB_{p_L}, \sigma)$ is a subshift of finite type (see \cite[Section 4.4]{Lind_Marcus_1995}).   Now we will show that ${}  X_{p_L}$ is indeed the unique transitive subshift of finite type in $\VB_{p_L}$ satisfying $h({}  X_{p_L})=h(\VB_{p_L})$. First we make the following claim.

 \emph{Claim: If $(x_i)\in\VB_{p_L}$ have a prefix $\la_1\ldots\la_{2^n}$ or $\overline{\la_1\ldots\la_{2^n}}$, then $(x_i)\in {} X_{p_L}$.}

By symmetry we may assume $x_1\ldots x_{2^n}=\la_1\ldots\la_{2^n}$. To prove the claim, by \cite[Corollary 1.3.5]{Lind_Marcus_1995} it suffices to show that for all $N> 2^{n+1}$ we have $x_1\ldots x_N\in\L({{}  X}_{p_L})$.

By Lemma \ref{lem:58} it follows that there exists a word $\varepsilon \in \L(\VB_{p_L})$ such that
\begin{equation}\label{eq:513}
x_1\ldots x_N \varepsilon (\la_1\ldots\la_{2^{n}}^-)^\f \in \VB_{p_L}.
\end{equation}
Observe that $(\la_1\ldots\la_{2^{n}}^-)^\f\in{{}  X}_{q_{T_{n+1}}}\subseteq{}  X_{p_L}$. Moreover,  $({{}  X}_{p_L}, \sigma)$ is a transitive subshift of finte type. Take a large integer $m>2^{n+1}$.  Then for any word  $\nu \in \L({\widetilde X}_{p_L})$ there exists a word $\eta \in \L({{}  X}_{p_L})$ such that
 $$(\la_1\ldots\la_{2^{n}}^-)^m\eta \nu \in  \L({{}  X}_{p_L}).$$ Note that $(\VB_{p_L}, \si)$ is a subshift of finite type. By (\ref{eq:513}) and \cite[Theorem 2.1.8]{Lind_Marcus_1995} it follows  that
\begin{equation}\label{eq:514}
x_1\ldots x_N \varepsilon (\la_1\ldots\la_{2^{n}}^-)^m\eta \nu \in  \L(\VB_{p_L}).
\end{equation}

 On the other hand, for any word $\om\in\L({{}  X}_{p_L}),$ we have by Lemma \ref{lem:58} that there exists a word $\de$ such that $\om\de(\la_1\ldots\la_{2^{n}}^-)^\f\in\VB_{p_L}$. Observe that  $x_1\ldots x_{2^n}=\la_1\ldots\la_{2^n}$. Then  by Lemma \ref{lem:42} it follows that
  $$\overline{x_1\ldots x_N} \prec \sigma^j((\la_1\ldots\la_{2^{n}}^-)^m x_1\ldots x_N) \prec x_1\ldots x_N \textrm{ for every } 0 \le j \le m\cdot 2^n.$$
  Again by noting that $(\VB_{p_L}, \si)$ is a subshift of finite type. Then it follows that
\begin{equation}\label{eq:515}
 \om\de(\la_1\ldots\la_{2^{n}}^-)^m x_1\ldots x_N\in\L(\VB_{p_L}).
 \end{equation}

By (\ref{eq:514}) and (\ref{eq:515})  it follows that any word in $B_*({}  X_{p_L})$ can be connected in $B_*(\VB_{p_L})$ to and from the word $x_1\ldots x_N$. Note that $({}  X_{p_L}, \si)$ is a maximal transitive subshift  of finite type in $(\VB_{p_L}, \si)$.
 Thus, $x_1\ldots x_N\in\L({{}  X}_{p_L})$. Since $N>2^n$ was chosen arbitrarily, we conclude that $(x_i)\in {{}  X}_{p_L}$.

\vspace{1em} Now we consider another subshift $(Y, \sigma)$ such that $Y\subseteq \VB_{p_L}\setminus {{}  X}_{p_L}$. Observe from the claim  that if a sequence $(x_i) \in \VB_{p_L}$ has $\la_1\ldots \la_{2^n}$ or $\overline{\la_1\ldots \la_{2^n}}$ as a prefix then $(x_i) \in {{}  X}_{p_L}$.   Note that $Y$ is forward $\sigma$-invariant, i.e., $\si(Y)\subseteq Y$. Then it follows that the words $\la_1\ldots \la_{2^n}$ and its reflection $\overline{\la_1\ldots \la_{2^n}}$ are forbidden in $Y$. Therefore,
\begin{align*}
  Y&\subseteq \set{(x_i): (\overline{\la_1\ldots\la_{2^n}}\,^+)^\f\lle\si^j((x_i))\lle(\la_1\ldots\la_{2^n}^-)^\f\textrm{ for all }j\ge 0}\\
  &\subseteq\set{(x_i): \overline{\al(q_c)}\lle\si^j((x_i))\lle\al(q_c)\textrm{ for all }j\ge 0}\\
  &=\VB_{q_c}.
\end{align*}
This implies that $h(Y)=0$.

So $({{}  X}_{p_L}, \sigma)$ is the unique transitive subshift of finite type such that $h({{}  X}_{p_L})=h(\VB_{p_L})$. Observe that $\al_1(p_L)\ldots \al_{2^n}(p_L)=\la_1\ldots\la_{2^n}$. In terms of  the above arguments we conclude that $\al(p_L)\in{}  X_{p_L}$.
\end{proof}

Building on Lemma \ref{lem:59} we now prove that for each $q$ inside a $*$-irreducible interval the subshift $(\VB_q, \sigma)$ contains a unique transitive subshift of finite type of full topological entropy.
\begin{proposition}
  \label{prop:510}
  Let $[p_L, p_R]$ be a $*$-irreducible interval. Then for each $q\in[p_L, p_R]$  there exists a unique transitive subshift of finite type $({{}  X}_q, \sigma)$ with ${{}  X}_q\subseteq\VB_q$ such that $h({{}  X}_q)=h(\VB_q)$.
\end{proposition}
\begin{proof}
Since the proof for $M=2k+1$ is similar, we only give the proof for $M=2k$.

Let $[p_L, p_R]$ be a $*$-irreducible interval generated by $a_1\ldots a_m$. Then by Lemma \ref{lem:44} there exists $n\in\N$ such that $[p_L, p_R]\subseteq(q_{T_{n+1}}, q_{T_n})$. Furthermore, $m>2^n$.
By Lemma \ref{lem:59} there exists a unique transitive subshift of finite type $({{}  X}_{p_L}, \sigma)$ with ${{}  X}_{p_L}\subseteq\VB_{p_L}$ of full topological entropy.

Now take $q\in(p_L, p_R]$. Note that $\VB_{q_{T_{n+1}}}\subseteq\VB_{q}$. Then by Lemmas \ref{lem:54} and \ref{lem:55} it follows that
\begin{equation}
\label{eq:516}
h(\VB_q)=h(\VB_{p_L})=h({{}  X}_{p_L})\geq h(\VB_{q_{T_{n+1}}})=\frac{\log 2}{2^n}.
\end{equation}

Take $(x_i)\in\VB_{p_R}\setminus\VB_{p_L}$. Then there exist $j\in\N\cup\set{0}$ such that
\[
x_{j+1}\ldots x_{j+m}=a_1\ldots a_m^+\quad\textrm{or}\quad x_{j+1}\ldots x_{j+m}=\overline{a_1\ldots a_m^+}.
\]
Note that $\al(p_R)=a_1\ldots a_m^+(\overline{a_1\ldots a_m})^\f$. Then by a similar argument as used in the proof of Proposition \ref{prop:37} we can show that
$(x_i)$ must end with an element of a sofic subshift $(X,\sigma)$ with topological entropy
$
h(X)=\frac{\log 2}{m}.
$
Using the fact $m>2^n$ and (\ref{eq:516}) we have $h(X)<h(\VB_q)$.

Therefore we may conclude that $({{}  X}_{p_L}, \sigma)$ is the unique transitive subshift of finite type of $(\VB_q, \sigma)$ such that $h({{}  X}_{p_L})=h(\VB_q)$.
\end{proof}

Now we prove Theorem \ref{th2} for the interval $(q_c, q_T)$.
\begin{proposition}
  \label{prop:511}
 { The interval} $[p_L, p_R]\subseteq(q_c, q_T)$ is an entropy plateau of $H(q)$ if and only if $[p_L, p_R]$ is a $*$-irreducible interval.
\end{proposition}
\begin{proof}

First we prove the sufficiency. Take $n\in\N$ and let $[p_L, p_R]\subseteq(q_{T_{n+1}}, q_{T_n})$ be a $n$-irreducible interval. Then by Lemma \ref{lem:55} it follows that $H(q)=H(p_L)$ for all $q\in[p_L, p_R]$. So, to prove that $[p_L, p_R]$ is an entropy plateau it suffices to show that $H(q)\ne H(p_L)$ for all $q\notin[p_L, p_R]$.

 Without loss of generality we assume $q\in(q_{T_{n+1}}, p_L)$. By Proposition \ref{prop:411} there exists a $*$-irreducible interval
 $[p,r]\subseteq (q, p_L).$ Then
 \[q_{T_{n+1}}<q<p<r<p_L.\]
 By Lemma \ref{lem:59} it follows that $\VB_p$ contains  a unique transitive subshift of finite type $({{}  X}_p, \sigma)$ such that $h({{}  X}_p)=h(\VB_p)$,  and $\VB_{p_L}$  contains a unique transitive subshift of finite type $({{}  X}_{p_L}, \sigma)$ such that $h({{}  X}_{p_L})=h(\VB_{p_L})$. Note that by the proof of Lemma \ref{lem:59} both ${{}  X}_p$ and ${{}  X}_{p_L}$ contain the maximal transitive subshift of finite type  $({{}  X}_{q_{T_{n+1}}}, \sigma)$ of $(\VB_{q_{T_{n+1}}}, \sigma)$. This implies that
   \[{{}  X}_p\subseteq {{}  X}_{p_L}.\]
    Observe by Lemma \ref{lem:59} that $\al(p_L)\in {{}  X}_{p_L}\setminus {{}  X}_p$. Then  by \cite[Corollary 4.4.9]{Lind_Marcus_1995} it follows that
 \[
 H(q)\le H(p)=h({{}  X}_p)<h({{}  X}_{p_L})=H(p_L).
 \]

 Similarly, we can also prove that $H(p)>H(p_L)$ if $p>p_R$. Therefore, $[p_L, p_R]$ is an entropy plateau.

Now we turn to prove the necessity. By the same argument used in the proof of Proposition \ref{prop:52} we could prove by using Lemma \ref{lem:49} that the union of all $n$-irreducible intervals covers $(q_{T_{n+1}}, q_{T_n})$ up to a set of measure zero. Hence, we conclude that the union of all $*$-irreducible intervals covers $(q_c, q_T)$ up to a set of measure zero. This establishes the necessity part of our theorem.
\end{proof}
\begin{remark}\label{rem:512}
\begin{itemize}
\item By the proofs of Propositions \ref{prop:52} and \ref{prop:511} it follows that the union of irreducible and $*$-irreducible intervals covers almost every point of $(q_c, M+1]$.

\item In terms of Remark \ref{rem:53} and \cite[Theorems 1.1 and 1.2]{Kong_Li_Lu_Vries_2016} we could also deduce from the proof of Proposition \ref{prop:511} that each entropy interval $[p_L, p_R]\subset(q_c, q_T)$ contains infinitely many connected components of $(q_c, q_T)\setminus\overline{\ul}$.
\end{itemize}
\end{remark}
\begin{proof}
  [{\bf Proof of Theorem \ref{th2}}]
  Theorem \ref{th2} follows from Propositions \ref{prop:52} and \ref{prop:511}.
\end{proof}

\begin{example}\label{ex:513}
Let $M=8$. Then by (\ref{eq:23}) and (\ref{eq:22}) we have
\[
\al(q_c)=(\la_i)=5435\, 3454\; 3453\,5435\ldots\quad\textrm{and}\quad \al(q_T)=5\,4^\f.
\]
This implies that $q_c\approx 5.80676$ and $q_T=3+2\sqrt{2}\approx 5.82843$.

Recall  that $I^*(c_1\ldots c_n)$  is the $*$-irreducible interval generated by $c_1\ldots c_n$. By Lemma \ref{lem:45} and Theorem \ref{th2} it follows that the following $*$-irreducible intervals are entropy plateaus of $H$.
\[
I^*(543),\quad I^*(543534),\quad I^*(543534543453),\quad \cdots.
\]
For example, for any $q\in I^*(543)$ the topological entropy of $\ub_q$ is given by
\[H(q)= \log\left(\frac{1+\sqrt{5}}{2}\right),\] which can be calculated via the subshift of finite type $(X, \sigma)$ where
\[
X=\set{(d_i)\in \set{3,4,5}^\f:  345\preccurlyeq d_{n+1}d_{n+2}d_{n+3}\preccurlyeq 543~\textrm{for any }n\ge 0}.
\]

 Moreover, recall that $I(d_1\cdots d_m)$ is the irreducible interval generated by a $d_1\ldots d_m$. By Table \ref{tab:1} and Theorem \ref{th2} it follows that the following irreducible intervals are entropy plateaus of $H$.
\begin{align*}
&I(5),\quad I(54),\quad I(55 3), \quad I(554),\quad\cdots;\\
& I(6), \quad I(63),\quad I(64), \quad I(65),\quad \cdots;\\
& I(7),\quad I(72),\quad I(73),\quad I(74),\quad I(75),\quad I(76),\quad \cdots;\\
& I(8),\quad I(81),\quad I(82),\quad I(83),\quad I(84),\quad I(85),\quad I(86),\quad I(87),\quad \cdots.
\end{align*}
When $q\in I(a)$ with $a\in\set{5,6,7,8}$ we have
\[
H(q)=\log (2a-7).
\]
When $q\in I(ab)$ with $a\in\set{5,6,7,8}$ and $b\in\set{9-a,\ldots, a-1}$, by \cite[Theorem 7.2]{Kong_Li_2015} it follows that
\[
H(q)=\log\left(a-4+\sqrt{(a-4)^2+2b-7} \right).
\]

We plot the asymptotic graph of $H$ as in Figure \ref{Fig:1}.
\end{example}

\noindent \textbf{Final remarks on the dynamical properties of $(\VB_q, \sigma)$ for $q\in(q_c,q_T)$. }

We end our study of $*$-irreducible intervals and $*$-irreducible sequences with some remarks. Firstly, note that by the proof of Lemma \ref{lem:59} combined with Lemma \ref{lem:25} we obtain that if $\alpha(q)$ is a $*$-irreducible sequence with $q_{T_{n+1}} < q < q_{T_n}$ the transitive components of $(\VB_q, \sigma)$ are $({{}  X}_{p_L}, \sigma)$ and finitely many transitive  subshifts of zero topological entropy  contained in a subshift parametrized by $(\lambda_1 \ldots \lambda_{2^n}^-)^\f$. This observation combined with Proposition \ref{prop:411} suggests that a similar phenomena might occur for every $q \in \mathcal{I}^*_r$. Moreover, from Theorem \ref{th2} and Remark \ref{rem:512} we obtain that for almost every $q \in [q_c, M+1]$, $(\VB_q, \sigma)$ contains a unique transitive component of full topological entropy and this component is a subshift of finite type. We may ask similar questions to those posed at the end of Section \ref{sec:3}. Can we characterize the set of $q$ such that $(\VB_q, \sigma)$ contains a unique transitive component of full topological entropy and this component is a subshift of finite type?

\section{The bifurcation set $\E$}\label{sec:6}

\noindent In this section we will investigate the bifurcation  set $\E$ and  prove Theorem \ref{th3}.

 Recall from (\ref{eq:412}) that $\I$ is the set of $q\in[q_T, M+1]$ for which $\al(q)$ is irreducible, and $\I^*$ is the set of $q\in(q_c, q_T)$ for which $\al(q)$ is $*$-irreducible. In the following proposition we express $\E$ in terms of $\I$ and $\I^*$
\begin{proposition}
  \label{prop:61}
$
 \E=\overline{\I\cup \I^*}.
$
\end{proposition}
\begin{proof}
Note  that $\E\subseteq[q_c, M+1]$. So it suffices to prove that $\E\cap[q_T, M+1]=\overline{\I}$ and $\E\cap[q_c, q_T]=\overline{\I^*}$.
Since the proof that $\E\cap[q_c, q_T]=\overline{\I^*}$ is similar, we only prove $\E\cap[q_T, M+1]=\overline{\I}$.
We begin our proof by remarking that by Proposition \ref{prop:52} we have
 \begin{equation}\label{eq:61}
 \E\cap[q_T, M+1]=[q_T, M+1]\setminus\bigcup(p_L, p_R),
 \end{equation}
 where the union is taken over all irreducible intervals.

We now prove $\E\cap[q_T, M+1]\subseteq\overline{\I}$. Take $q\in\E\cap[q_T, M+1]$. If $q$ is the endpoint of an irreducible interval, then by Proposition \ref{prop:411} it follows that there exists a sequence $(q_n)\subseteq \I$ such that $q_n\ra q$ as $n\ra\f$. This implies that $q\in\overline{\I}$.

If $q$ is not the endpoint of an irreducible interval, then by (\ref{eq:61}) we have $q\in[q_T, M+1]\setminus\bigcup[p_L,p_R]$. Note that $[q_T, M+1]\setminus\bigcup(p_L, p_R)\subseteq\overline{\ul}$ is a Lebesgue null set. Moreover, by Lemma \ref{lem:46} these irreducible intervals are pairwise disjoint. Therefore there exists a sequence of irreducible intervals $\set{[p_L(n), p_R(n)]}$ such that $p_L(n)\ra q$ as $n\ra\f$. Note that $p_L(n)\in \I$ for each $n\in\N$. This implies that
  $q\in\overline{\I}$. Therefore by (\ref{eq:61}) we may conclude that $\E\cap[q_T, M+1]\subseteq\overline{\I}$.

 Now we prove $\overline{\I}\subseteq\E\cap[q_T, M+1]$. By (\ref{eq:61}) it suffices to prove that $\I\cap(p_L, p_R)=\emptyset$ for any irreducible interval $[p_L, p_R]\subseteq[q_T, M+1]$. Let $[p_L, p_R]$ be an irreducible interval generated by $a_1\ldots a_m$.

 Suppose on the contrary that $\I\cap(p_L, p_R)\ne\emptyset$, and take $q\in \I\cap(p_L, p_R)$. Then by Lemma \ref{lem:21} we have
 \begin{equation}
   \label{eq:62}
   (a_1\ldots a_m)^\f=\al(p_L)\prec\al(q) \prec\al(p_R)=a_1\ldots a_m^+(\overline{a_1\ldots a_m})^\f.
 \end{equation}
 Note that $\si^n(\al(q))\lle\al(q)$ for all $n\ge 0$. By (\ref{eq:62}) this implies that $\al_1(q)\ldots \al_m(q)=a_1\ldots a_m^+$. However, $(\al_1(q)\ldots \al_m(q)^-)^\f=(a_1\ldots a_m)^\f\in\vb$. By (\ref{eq:62}) and Definition \ref{def:26} it follows that
 $\al(q)$ is not irreducible, leading to a contradiction with $q\in \I$.
\end{proof}

Now we turn to investigate the Hausdorff dimension of $\E$.  By Proposition \ref{prop:61} it follows  that $\I\subseteq\E$. So to prove Theorem \ref{th3} it suffices to show that $\dim_H \I=1$.
Inspired by the proof of \cite[Theorem 1.6]{Komornik_Kong_Li_2015_1} we construct the following sequence $\set{ \I_N}_{N=2}^\f,$ consisting of subsets of $\I$.

 Fix an integer $N\ge 2$. Let $ \IB_N$ be the set of sequences $(a_i)\in\set{0,1,\ldots,M}^\f$ satisfying $a_1\ldots a_{2N}=M^{2N-1}0$, and the lexicographical inequalities
\[
0^N\prec a_{sN+1}\ldots a_{sN+N}\prec M^N\quad\textrm{for}\quad s=2,3,\ldots.
\]
Then any sequence $(a_i)\in  \IB_N$ satisfies
$
\overline{(a_i)}\prec\si^n((a_i))\prec (a_i)$ for all $n\ge 1.
$
This implies that
\begin{equation*}
 \I_N:=\set{q: \al(q)\in \IB_N}\subseteq\ul\quad\textrm{for all }N\ge 2.
\end{equation*}

In the following lemma we show that $ \I_N$ is a subset of $\I$ for all $N\ge 2$.
\begin{lemma}\label{lem:62}
  $ \I_N\subseteq \I$ for all $N\ge 2$.
\end{lemma}
\begin{proof}
  Take $q\in  \I_N$. Then $\al(q)=(\al_i)$ satisfies
  \begin{equation}
    \label{eq:63}
    \al_1\ldots\al_{2N}=M^{2N-1}0\quad\textrm{and}\quad 0^N\prec\al_{sN+1}\ldots\al_{sN+N}\prec M^N
  \end{equation}
  for any $s=2,3,\ldots$. Clearly, $(\al_i)\in\vb$. Suppose $(\al_1\ldots \al_j^-)^\f\in\vb$ for some $j\ge 1$. We claim that
  \begin{equation*}
    \al_1\ldots\al_j(\overline{\al_1\ldots\al_j}\,^+)^\f\prec(\al_i).
  \end{equation*}

    Note that for such a $j$ we must have $j\ne 2N$.  We split the proof into the following two cases.

    Case (I). $1\le j<2N$. Then by (\ref{eq:63}) it follows that
    \[
    \al_1\ldots\al_j(\overline{\al_1\ldots \al_j}\,^+)^\f=M^j(0^{j-1}1)^\f\prec (\al_i).
    \]

    Case (II). $j>2N$. By (\ref{eq:63}) we have $\al_{n+1}\ldots\al_{n+2N-1}\succ 0^{2N-1}$ for all $n\ge 1$. This implies
    \[
   \al_1\ldots\al_j(\overline{\al_1\ldots \al_j}\,^+)^\f=\al_1\ldots\al_j(0^{2N-1}M\overline{\al_{2N+1}\ldots \al_j}\,^+)^\f\prec(\al_i).
    \]

    Therefore $(\al_i)$ is irreducible and $ \I_N\subseteq \I$.
\end{proof}

The following property of the Hausdorff dimension is well-known (cf.~\cite[Proposition 2.3]{Falconer_1990}).
\begin{lemma}
  \label{lem:63}
  Let $f: (X, d_1)\ra (Y, d_2)$ be a map between two metric spaces. If there exists a constants $C>0$ such that
  \[
  d_2(f(x), f(y))\le C\cdot d_1(x, y)\quad\textrm{for all }x, y\in X,
  \]
  then $\dim_H X\ge \dim_H f(X)$.
\end{lemma}

Using Lemma \ref{lem:63} we can obtain a lower bound for the Hausdorff dimension of $ \I_N$.
\begin{lemma}
  \label{lem:64}
  Let $N\ge 2$. Then
  \[
  \dim_H  \I_N\ge \frac{\log((M+1)^N-2)}{N\log(M+1)}.
  \]
\end{lemma}
\begin{proof}
  Note that $\pi_{M+1}( \IB_N)$ is a scaled copy of the set
  \[
  E=\set{\pi_{M+1}((a_i)): 0^N\prec a_{sN+1}\ldots a_{sN+N}\prec M^N\textrm{ for all }s\ge 0}
  \]
  under an affine transformation. Therefore $\dim_H\pi_{M+1}( \IB_N)=\dim_H E$. Note that $E$ is a self-similar set generated by the iterated function system
  \[
  \set{f_{a_1\ldots a_N}(x)=\frac{x}{(M+1)^N}+\sum_{i=1}^N\frac{a_i}{(M+1)^i}: 0^N\prec a_1\ldots a_N\prec M^N}.
  \]
  This iterated function system satisfies the open set condition \cite{Falconer_1990}. Therefore
  \[
  \dim_H\pi_{M+1}( \IB_N)=\dim_H E=\frac{\log ((M+1)^N-2)}{N\log (M+1)}.
  \]
  To prove our result it suffices to prove that $\dim_H  \I_N\ge\dim_H\pi_{M+1}( \IB_N)$.

    Take $q_1, q_2\in  \I_N$ with $q_1<q_2$. Then by Lemma \ref{lem:21} we have $\al(q_1)<\al(q_2)$. Let $j\ge 1$ be the   integer satisfying
    \begin{equation}\label{eq:64}
    \al_1(q_1)\ldots \al_{j-1}(q_1)=\al_1(q_2)\ldots\al_{j-1}(q_2)\quad\textrm{and}\quad \al_j(q_1)<\al_j(q_2).
    \end{equation}
    This implies that
    \begin{equation}\label{eq:65}
      \begin{split}
         \pi_{M+1}(\al(q_2))-\pi_{M+1}(\al(q_1))=\sum_{i=j}^\f\frac{\al_i(q_2)-\al_i(q_1)}{(M+1)^i}\le\sum_{i=j}^\f\frac{M}{(M+1)^i}=(M+1)^{1-j}.
      \end{split}
    \end{equation}

Note that each $q_2\in  \I_N$ satisfies
\[
0^{2N}\prec\al_{i+1}(q_2)\ldots\al_{i+2N}(q_2)\prec M^{2N}\quad\textrm{for all }i\ge 1.
\]
Then by (\ref{eq:64}) we have
\[
\sum_{i=1}^j\frac{\al_i(q_2)}{q_1^i}\ge\sum_{i=1}^\f\frac{\al_i(q_1)}{q_1^i}=1=\sum_{i=1}^\f\frac{\al_i(q_2)}{q_2^i}\ge \sum_{i=1}^j\frac{\al_i(q_2)}{q_2^i}+\frac{1}{q_2^{j+2N}}.
\]
Observe that $q_c<q_1<q_2\le M+1$. Therefore,
\begin{equation}
  \label{eq:66}
  \begin{split}
    \frac{1}{(M+1)^{j+2N}}&\le\frac{1}{q_2^{j+2N}}\le\sum_{i=1}^j\Big(\frac{\al_i(q_2)}{q_1^i}-\frac{\al_i(q_2)}{q_2^i}\Big)\\
    &\le\sum_{i=1}^\f\Big(\frac{M}{q_1^i}-\frac{M}{q_2^i}\Big)=\frac{M(q_2-q_1)}{(q_1-1)(q_2-1)}\\
    &\le\frac{M}{(q_c-1)^2}(q_2-q_1).
  \end{split}
\end{equation}

Therefore, by (\ref{eq:65}) and (\ref{eq:66}) it follows that
\[
\pi_{M+1}(\al(q_2))-\pi_{M+1}(\al(q_1))\le (M+1)^{1-j}\le\frac{(M+1)^{2N+2}}{(q_c-1)^2}(q_2-q_1).
\]
Note that $\pi_{M+1}(\al(q_2))-\pi_{M+1}(\al(q_1))>0$. Hence, by using
\[
f:=\pi_{M+1}\circ\al:  \I_N\ra\pi_{M+1}( \IB_N)
\]
in  Lemma \ref{lem:63} we conclude that $\dim_H  \I_N\ge \dim_H\pi_{M+1}( \IB_N)$.
\end{proof}

\begin{proof}[{\bf Proof of Theorem \ref{th3}}]
By Proposition \ref{prop:61} it remains to show that $\dim_H\E=1$.

Note by Lemma \ref{lem:62} that
$
 \I_N\subseteq \I\subseteq\E$ {for all }$ N\ge 2.
$
Then by Lemma \ref{lem:64} it follows that
\[
\dim_H\E\ge \dim_H  \I_N\ge\frac{\log ((M+1)^N-2)}{N\log (M+1)}
\]
for all $N\ge 2$. Letting $N\ra\f$ we conclude  $\dim_H\E=1$.
\end{proof}

By Theorem \ref{th2} the function $H$ is differentiable inside each entropy plateau $(p_L, p_R)$. Then the set of points where $H$ is not differentiable   is a subset of $\E$. Here   we make the following conjecture.
\begin{conjecture}
$\E$ is the set of points where $H$ is not differentiable.
\end{conjecture}

Note that $\E\subset\overline{\ul}$, and  $\overline{\ul}$ is a Cantor set of full Hausdorff dimension. By Theorem \ref{th3} it follows that both sets have full Hausdorff dimension. Then it is natural to ask the size of the difference set $\overline{\ul}\setminus\E$. { In a recent paper \cite{Bak-Kon-17} the second and the third authors showed that when the digit set is $\set{0, 1}$ the Hausdorff dimension of the difference set $\overline{\ul}\setminus\E$ is approximately equal to $0.368699$. It would be interesting to determine the Hausdorff dimension of $\overline{\ul}\setminus\E$ for more general digit set $\set{0, 1,\ldots, M}$ with $M\ge 2$.}

\section{Box dimension   of $\u_q$}\label{sec:7}
\label{sec:8}

\noindent In this section we will show that the Hausdorff dimension and box dimension of $\u_q$ coincide for all $q\in(1,M+1]$.
 For $q\in(1, M+1]$ let
 \[
 \W_q:=\set{\pi_q((x_i)): \overline{\al(q)}\prec\si^n((x_i))\prec\al(q)\textrm{ for all }n\ge 0},
 \]
 and let $\WB_q$ be the set of corresponding expansions of $x\in\W_q$. It is easy to check that $\W_q\subseteq \u_q\cap[\frac{M}{q-1}-1, 1]$.

 We start by proving that the upper box dimension of $\u_q$ and $\W_q$ coincide.

\begin{lemma}\label{lem:71}
Let $q\in(1,M+1]$. Then $\overline{\dim}_\text{B}(\u_q)=\overline{\dim}_\text{B}(\W_q)$.
\end{lemma}
\begin{proof}
Unlike Hausdorff dimension the upper box dimension is not countably stable. So this lemma requires a little work. We begin by recalling an observation made in \cite[Lemma 2.5]{Komornik_Kong_Li_2015_1}. Namely
\begin{align*}
\u_{q}=\Big\{& 0,\frac{m}{q-1},\\
& \frac{d_{1}}{q}+\frac{\W_{q}}{q},\, d_{1}=1,\ldots,M-1\\
& \frac{d_{m}}{q^{m}}+\frac{\W_{q}}{q^{m}},\, m=2,3,\ldots,\, d_{m}=1,\ldots,M\\
&\sum_{i=1}^{m-1}\frac{M}{q^{i}}+\frac{d_{m}}{q^{m}}+\frac{\W_{q}}{q^{m}},\, m=2,3,\ldots,\, d_{m}=0,\ldots, M-1\Big\}.
\end{align*}
So $\u_{q}$ consists of the endpoints of $I_{q}$ and scaled copies of $\W_{q}$.

 Let us now fix some $\delta\in(0, 1)$. Note that $\W_q\subseteq[\frac{M}{q-1}-1, 1]$. It is straightforward to show that
  $$\frac{c_{m}}{q^{m}}+\frac{\W_{q}}{q^{m}}\in[0,\delta],\quad \sum_{i=1}^{m-1}\frac{M}{q^{i}}+\frac{d_{m}}{q^{m}}+\frac{\W_{q}}{q^{m}}\in\Big[\frac{M}{q-1}-\delta,\frac{M}{q-1}\Big]$$
  for any $c_{m}=1,\ldots,M$  and  for any $d_{m}=0,\ldots, M-1$, whenever
  $$m\geq \Big[ \frac{\log(\frac{M+1}{\delta})}{\log q}\Big]+1.$$

   Now we consider covers by closed intervals of width $\delta$. Let $N_{\delta}(X)$ denote the minimal number of such intervals required to cover a set $X\subset \mathbb{R}$. Then by the above we have
\begin{align*}
N_{\delta}(\u_{q})&\leq 2 +\sum_{d_{1}=1}^{M-1} N_{\delta}\Big(\frac{d_{1}}{q}+\frac{\W_{q}}{q}\Big) +\sum_{m=2}^{\Big[ \frac{\log(\frac{M+1}{\delta})}{\log q}\Big]}\sum_{d_{m}=1}^{M}N_{\delta}\Big(\frac{d_{m}}{q^{m}}+\frac{\W_{q}}{q^{m}}\Big)\\
&+\sum_{m=2}^{\Big[ \frac{\log(\frac{M+1}{\delta})}{\log q}\Big]}\sum_{d_{m}=0}^{M-1}N_{\delta}\Big(\sum_{i=1}^{m-1}\frac{M}{q^{i}}+\frac{d_{m}}{q^{m}}+\frac{\W_{q}}{q^{m}}\Big)
\end{align*}
Within each of our summations we are considering minimal $\delta$-covers of scaled down copies of $\W_{q}$. In which case each term can be bounded above by $N_{\delta}(\W_{q}),$ and we obtain the following upper bound
\begin{equation*}
N_{\delta}(\u_{q})\leq 2+M\Big(1+2\Big[ \frac{\log(\frac{M+1}{\delta})}{\log q}\Big]\Big)N_{\delta}(\W_{q})\le 3 M \Big[ \frac{\log(\frac{M+1}{\delta})}{\log q}\Big]N_{\delta}(\W_{q}).
\end{equation*}
This implies that
\begin{equation}
  \label{eq:71}
  \begin{split}
  \overline{\dim}_B(\u_q)&=\limsup_{\de\ra 0^+}\frac{\log N_\de(\u_q)}{-\log \de}\\
  &\le\limsup_{\de\ra 0^+}\frac{\log (3M)+\log \Big[ \frac{\log(\frac{M+1}{\delta})}{\log q}\Big] +\log N_{\delta}(\W_{q})}{-\log \de}\\
  &=\limsup_{\de\ra 0^+}\frac{\log N_\de(\W_q)}{-\log \de}=\overline{\dim}_B\W_q.
  \end{split}
\end{equation}

On the other hand,
we trivially have the lower bound $N_{\delta}(\W_{q})\leq N_{\delta}(\u_{q})$ for all $\delta>0$, and therefore $\overline{\dim}_{B}(\W_{q})\leq \overline{\dim}_{B}(\u_{q})$. This lower bound combined with (\ref{eq:71}) implies $\overline{\dim}_{B}(\u_{q})=\overline{\dim}_{B}(\W_{q})$.
\end{proof}

\begin{proof}[\bf{Proof of Theorem \ref{th4}}]
By \cite[Theorem 1.3 and Proposition 2.8]{Komornik_Kong_Li_2015_1}  it follows that
\[\dim_{H}(\u_{q})=\frac{h(\ub_q)}{\log q}=\frac{h(\WB_{q})}{\log q}.\]
Thus, by Lemma \ref{lem:71}, and the trivial bound $\dim_{H}(\u_{q})\leq\overline{\dim}_{\text{B}}(\u_{q}),$ it suffices to show that
$$\overline{\dim}_{\text{B}}(\W_{q}) \leq \frac{h(\WB_{q})}{\log q}.$$

From the definition of topological entropy we immediately obtain
\begin{equation}\label{eq:72}
\#B_{n}(\WB_{q})\leq 2^{n(h(\WB_{q})+o(1))}.
\end{equation}
Where we have used the standard little $o$ notation. To each $\om=\om_1\ldots \om_n\in B_{n}(\WB_{q})$ we associate the interval $${J}_{\om}:=\Big[\sum_{i=1}^{n}\frac{\om_{i}}{q^{i}},\sum_{i=1}^{n}\frac{\om_{i}}{q^{i}}+\frac{M}{q^{n}(q-1)}\Big].$$
Importantly ${J}_{\om}$ is precisely the set $x\in [0,M/(q-1)]$ for which $x$ has a $q$-expansion beginning with $\om.$ Therefore $\{{J}_{\om}\}_{\om\in B_{n}(\WB_{q})}$ is a cover of $\W_{q}.$

 The following gives us the required upper bound. Let $\de_n=\frac{M}{q^n(q-1)}$.  Note that when calculating $\overline{\dim}_\text{B} (\W_{q}) $ it suffices to consider the restricted sequence $N_{\de_n}(\W_{q})$ (see \cite{Falconer_1990}). Then by (\ref{eq:72}) it follows that
\begin{align*}
\overline{\dim}_\text{B} (\W_{q}) &=\mathop{\limsup}\limits_{n\to\infty}\frac{\log N_{\de_n}(\W_{q})}{-\log \de_n} \\
&\leq \limsup_{n\to\infty} \frac{\log 2^{n(h(\WB_{q})+o(1))}}{-\log \frac{M}{q-1}+n\log q}
 =\frac{h(\WB_{q})}{\log q}.
\end{align*}
\end{proof}

\section*{Acknowledgements}

The authors   thank the anonymous referees for many useful comments and suggestions. This work was initiated during the authors' visit to the University of Utrecht in August 2015. We are grateful to Karma Dajani (granted by NWO no.613.001.022) and Kan Jiang for their hospitality.  The authors would also like to thank Wenxia Li for his suggestions on the set $\E$ of bifurcation bases.


\end{document}